\documentclass[a4paper,10pt]{article}
\usepackage{geometry}                
\usepackage[active]{srcltx}
\usepackage{hyperref}
\usepackage[normalem]{ulem}

\usepackage{amsmath}
\usepackage{amssymb}
\usepackage{parskip}
\usepackage{mathtools}
\usepackage{fullpage}
\usepackage{mathrsfs}
\usepackage{amsthm}
\usepackage{physics}
\usepackage{bbm}
\usepackage[utf8]{inputenc}
\usepackage{caption}
\usepackage{subcaption}
\usepackage{graphicx}
\usepackage{attachfile}
\usepackage{navigator}
\usepackage{theoremref}
\usepackage{faktor}
\usepackage{musicography}
\usepackage{accents}
\usepackage{tikz}
\usetikzlibrary{angles,quotes}

\newtheorem{theorem}{Theorem}[section]
\newtheorem{maintheorem}{Theorem}

\newtheorem{definition}[theorem]{Definition}
\newtheorem{proposition}[theorem]{Proposition}
\newtheorem{corollary}[theorem]{Corollary}
\newtheorem{lemma}[theorem]{Lemma}
\newtheorem{remark}[theorem]{Remark}

\numberwithin{equation}{section}

\newcommand{\R}{\mathbb{R}}
\newcommand{\Q}{\mathbb{Q}}
\newcommand{\C}{\mathbb{C}}
\newcommand{\N}{\mathbb{N}}
\newcommand{\Z}{\mathbb{Z}}

\newcommand{\e}{\varepsilon}
\newcommand{\p}{\varphi}
\newcommand{\supp}{\mathrm{supp}}
\newcommand{\Lip}{\mathrm{Lip}}

\newcommand{\dist}{\mathrm{dist}}
\newcommand{\Ric}{\mathrm{Ric}}

\newcommand{\vol}{\mathrm{vol}}
\newcommand{\Hess}{\mathrm{Hess}}
\newcommand{\Riem}{\mathrm{Riem}}
\newcommand{\Ch}{\mathrm{Ch}}
\newcommand{\cd}{\mathsf{CD}}
\newcommand{\rcd}{\mathsf{RCD}}
\newcommand{\be}{\mathrm{BE}}
\newcommand{\sfd}{\mathsf d}
\newcommand{\di}{{\rm d}}

\title{On the equivalence of distributional and synthetic Ricci curvature lower bounds}
\author{Andrea Mondino and Vanessa Ryborz}
\date{}

\begin{document}

\maketitle

{

\begin{abstract}
The goal of the paper is to prove the equivalence of distributional and synthetic Ricci curvature lower bounds for a weighted Riemannian manifold with continuous metric tensor having Christoffel symbols in $L^2_{{\rm loc}}$, and with weight in $C^0\cap W^{1,2}_{{\rm loc}}$. 

\end{abstract}
}

\tableofcontents

\section{Introduction}
Let $(M, g)$ be a smooth $n$-dimensional Riemannian manifold without boundary.  
One can locally compute the coefficients of the Ricci curvature tensor in terms of the metric tensor $g$ and its first two derivatives. More precisely, denoting by $\Gamma^{k}_{ij}$ the Christoffel symbols of $g$ 
$$
\Gamma^k_{ij}= \frac{1}{2} g^{kl} (\partial_j g_{li}+\partial_{i} g_{jl}-\partial_l g_{ij}), \quad i,j,k=1,\ldots,n,
$$
the Ricci curvature can be written locally as
\begin{equation}\label{eq:defRicciChrist}
     \Ric_{ij}=\partial_k\Gamma^k_{ij}-\partial_i\Gamma^k_{kj}+\Gamma^k_{ji}\Gamma^l_{lk}-\Gamma^k_{jl}\Gamma^l_{ik}, \quad i,j=1,\ldots,n,
\end{equation} 
where $\partial_i$ is the shorthand notation for $\frac{\partial}{\partial x^i}$, and where  we used the Einstein convention that repeated indices are summed.

Given a smooth function $V$ on $M$, one can define the weighted measure  $\di\mu = e^{-V}\di\vol_g$. Given $N \in [n, \infty]$, the Bakry-Émery $N$-Ricci curvature tensor of the weighted Riemannian manifold $(M,g,\mu)$ is defined by
\begin{equation}\label{eq:defBENRicci}
    \Ric_{\mu, N} := \Ric + \Hess V - \frac{1}{N-n}\nabla V \otimes \nabla V. 
\end{equation}
In case that $N = n$, we adopt the convention that $V$ must be constant and that  $\Ric_{\vol_g, n} = \Ric$.  We also adopt the standard convention that $1/\infty=0$, so that for $N=\infty$ the last adding term in the right hand side of \eqref{eq:defBENRicci} disappears.
For $K\in \R$, we say that the Bakry-Émery $N$-Ricci curvature tensor is bounded below by $K$, if $\Ric_{\mu, N}(X,X) \geq Kg(X,X)$ for all smooth vector fields $X$. 

{Since the local expression of  the Ricci tensor involves two derivatives of the metric tensor $g$, and the Bakry-Émery $N$-Ricci curvature tensor involves the Ricci tensor and the Hessian of the weight function $V$, some care is needed if $g$ or $V$ are not twice differentiable. We will work with two approaches to Ricci curvature lower bounds for weighted manifold of regularity below $C^2$:  the distributional and the synthetic approaches.

Via the distribution theory on smooth manifolds (cf.\;\cite{lefloch2007definition, grosser2013geometric, graf2020singularity}), one can define a distributional Ricci curvature and generalise the notion of Ricci curvature lower bounds. This approach is suitable to handle the case of a  smooth manifold $M$ endowed with a continuous metric $g$ having Christoffel symbols in $L^2_{{\rm loc}}$, as it is apparent from \eqref{eq:defRicciChrist}. Moreover, assuming $V \in C^{0}\cap W^{1,2}_{{\rm loc}}$ suffices to define a distributional Bakry-\'Emery $N$-Ricci curvature tensor of the corresponding weighted space, as one can check using \eqref{eq:defBENRicci}. }
Note that this approach  assumes the space to be smooth: all the non-smoothness is encoded in the metric tensor and, possibly, in the weighted measure. 
Let us also mention  \cite{li2022existence} for a notion of Ricci curvature lower bounds for continuous metrics conformal to smooth ones, based on the concept of viscosity solutions for non-linear elliptic partial differential equations.

A different approach is to drop also the smoothness assumption of the underlying space and consider the general framework of metric measure spaces. A metric measure space is a triplet $(X, \sfd, \mathfrak{m})$, where $(X,\sfd)$ is a  complete and separable metric space and  $\mathfrak{m}$ is a non-negative $\sigma$-finite Borel measure (playing the role of reference volume measure). 
{By analysing convexity properties of suitable entropy functionals on the space of probability measures endowed with the Kantorovich-Wasserstein quadratic transportation distance $W_2$, Sturm \cite{sturm2006geometryI, sturm2006geometryII} and Lott-Villani \cite{lott2009ricci, LottVillaniJFA} devised synthetic notions of Ricci curvature bounded below by some constant $K\in \R$ and dimension bounded above by some $N\in [1,\infty]$. The metric measure spaces satisfying such a synthetic notion of  Ricci curvature bounded below by $K\in \R$ and dimension bounded above by  $N\in [1,\infty]$ are called  $\cd(K,N)$ spaces. Such a synthetic notion is consistent with the smooth definitions.
Moreover, the class of $\cd(K,N)$ spaces is stable under pointed measure Gromov-Hausdorff convergence \cite{sturm2006geometryI, sturm2006geometryII,lott2009ricci, villani2009optimal}, and satisfies several geometric and analytic properties such as Brunn-Minkowski and Bishop-Gromov inequalities \cite{sturm2006geometryII}, Poincar\'e inequality \cite{LottVillaniJFA, RajalaCVPDE}, L\'evy-Gromov isoperimetric inequality \cite{CavMondInvent} (under an essentially non-branching assumption on geodesics).}

Since the conditions $\cd(K, \infty)$ and $\cd(K, N)$  provide a definiton of Ricci curvature lower bounds for manifolds with metrics of regularity below $C^2$,  a natural question to ask is whether distributional and synthetic  Ricci curvature lower bounds continue to agree on metrics with regularity below $C^2$.  
Using regularisations of the metric tensor, Kunzinger, Oberguggenberger and Vickers \cite{kunzinger2022synthetic} proved that distributional Ricci curvature bounded below by $K$ implies the CD$(K, \infty)$-condition in the case that $M$ is  a compact manifold endowed with a $C^1$ Riemannian metric, and established the reverse implication under the stronger assumption that $g$ is a   $C^{1,1}$ metric tensor satisfying an additional convergence condition on its regularisations.
\\In this paper, we will prove the following equivalence result:

\smallskip

\begin{maintheorem}[see Theorem \ref{equivalence_result}]\label{main_result_intro}
     Let $M$ be a smooth manifold, $g$ a continuous Riemannian metric with Christoffel symbols in $L^2_{\rm{loc}}$,  and  $V  \in C^{0}\cap W^{1,2}_{\rm{loc}}(M)$ a positive function on $M$. Define the weighted measure $\mu$ as $\di\mu:= e^{-V}\di\vol_g$. Let  $N \in [n, \infty]$ and $K \in \R$. The following are equivalent:
     \begin{itemize}
         \item[$\mathrm{(i)}$] $(M, \sfd_g, \mu)$ is a $\cd(K, N)$ space.
        \item[$\mathrm{(ii)}$] The distributional Bakry-Émery $N$-Ricci curvature tensor is bounded below by $K$ and $\mu$ has at most exponential volume growth in the sense of \eqref{condition_on_measure}.  
     \end{itemize}
\end{maintheorem}
\smallskip

\begin{remark}[On the smoothness assumption on $M$]
Note that, for $C^{0}$-Riemannian metrics, the natural class of differentiability of the manifolds is $C^{1}$. However, by a classical result of Whitney \cite{Whit} (see also \cite[Sect.\;4]{Munk}), a $C^{1}$-manifold always possesses a $C^{\infty}$-sub-atlas, unique up to diffeomorphisms; we will choose some such sub-atlas whenever convenient.
\end{remark}

\begin{remark}[On the volume growth assumption]
The volume growth assumption \eqref{condition_on_measure} is satisfied by any $\cd(K,N)$ metric measure space, $N\in [1,\infty]$, as proved in \cite[Thm.\;4.24]{sturm2006geometryI}. Thus it is necessary for the implication from distributional to synthetic Ricci lower bounds to hold.  In Proposition \ref{prop:VolumeW1p} (resp.\;Proposition \ref{prop:VolumeC1}), we show that \eqref{condition_on_measure} is satisfied for a smooth manifold endowed with a $C^0\cap W^{1,p}_{loc}$ Riemannian metric, $p>n$, satisfying a distributional lower bound on the Ricci curvature   (resp.\;for a smooth manifold endowed with a $C^1$-Riemannian metric and a $C^1$-weighted measure satisfying a distributional lower bound on the $N$-Bakry-\'Emery Ricci tensor). See Corollary \ref{cor:FinalEquivalence} and Remark \ref{rem:FinalEquiv} for the equivalence result taking in consideration the validity of \eqref{condition_on_measure} in the aforementioned cases.
\end{remark}

 Few months after the present work was posted on arXiv, an independent paper \cite{KOV24} by Kunzinger, Ohanyan and Vardabasso appeared, where the authors prove the implication \emph{from distributional  to synthetic} Ricci lower bounds, under the stronger conditions that the manifold is  endowed with a $C^{0,1}_{loc}$-Riemannian metric and the volume is  non-weighted. Their approach is different, and based on the stability of the synthetic Ricci lower bounds under convergence.

\smallskip

\textbf{Main ingredients in the proof and organisation of the paper.}
{ A first important observation is that a smooth manifold endowed with a continuous Riemannian metric and a continuous weight on the volume measure is an infinitesimally Hilbertian metric measure space (in the sense of \cite{Gigli15}, after \cite{ambrosio2014duke}) satisfying the Sobolev-to-Lipschitz property (see Cor.\;\ref{summary_chapter_4}). Thus it is a  $\cd(K, N)$ space if and only if it is a  $\rcd(K, N)$ space. Now, a deep result by Cavalletti-Milman \cite{CavallettiMilman21} (see also \cite{Li-AMPA} for the extension to the case of $\sigma$-finite measures) is that the $\rcd(K,N)$ condition is equivalent to the $\rcd^{*}(K,N)$ condition.

The main ingredient to prove Theorem \ref{main_result_intro}, is the equivalence of the  $\rcd^*(K, N)$ condition (corresponding to a Lagrangian formulation of the curvature-dimension condition) and the Bakry-\'Emery condition $\be(K, N)$, which roughly corresponds to the Bochner inequality (and provides a Eulerian formulation of the curvature-dimension condition).   Such an equivalence is a deep result that was first established in the  case $N=\infty$ by Ambrosio-Gigli-Savar\'e \cite{ambrosio2014duke,ambrosio2015bakry} (after the work \cite{GigliKuwadaOhtaCPAM} in Alexandrov spaces) and later for $N<\infty$ by Erbar-Kuwada-Sturm \cite{erbar2015equivalence} and Ambrosio-Mondino-Savar\'e   \cite{ambrosio2019nonlinear}. We will prove that the  $\be(K, N)$ condition  is equivalent to  having the distributional $N$-Bakry-\'Emery Ricci tensor bounded below by $K$. Most of the work consists in comparing the test objects used in metric measure setting for the $\be(K, N)$ condition, and the smooth functions and vector fields, which are the test objects in the distributional formulation. More precisely,  we will make use of  the second order calculus developed in \cite{gigli2018nonsmooth} in the metric measure framework. Such a non-smooth approach requires to work with gradient vector fields, while the distributional lower Ricci bounds are classically tested against  arbitrary smooth vector fields.  To this regard, a key technical tool will be an approximation result of smooth vector fields by gradient vector fields (see Lemma \ref{main_construction}). Also, since we will work with metrics of low regularity, we will use the parabolic version of the De Giorgi-Nash-Moser regularity (see Proposition \ref{c0_convergence_heat_flow}). The proof of the implication from $\rcd$ to distributional Ricci lower bounds will be obtained via a contradiction argument involving  suitable coverings by balls centred in Lebesgue points of the test objects and using the Besicovitch covering theorem and the  Hahn decomposition theorem for signed Radon measures. Such an implication will be established in Theorem \ref{first_main_theorem} for the case $N=\infty$ and in Theorem \ref{cdknforlipschitzmetric} for the case $N\in [n,\infty)$. Let us mention that the case  $N\in [n,\infty)$ is slightly more subtle, as the distributional $N$-Bakry-\'Emery Ricci tensor involves a quadratic non-linearity in the weight (term which is not present in the case $N=\infty$). The converse implication will be obtained in Theorem \ref{from_distri_torcd}. 

Along the way, in Sec.\;4,  we will compare the theory of classical Sobolev spaces on manifolds (see for instance \cite{hebey1996sobolev}) with the generalised notion of differentiability introduced by Cheeger \cite{cheeger1999differentiability} and further investigated by Ambrosio, Gigli and Savaré \cite{ambrosio2014inventio}; the framework will be a smooth manifold endowed with  a continuous Riemannian metric $g$ and a positive continuous weight on the volume measure.

\medskip

\textbf{Acknowledgements.}
A.M.\;is supported by the European Research Council (ERC), under the European Union Horizon 2020 research and innovation programme, via the ERC Starting Grant  “CURVATURE”, grant agreement No. 802689.
 }

\subsection*{Notation}
Consider a smooth manifold $M$. When equipped with a Riemannian metric, we will always denote the metric by $g$. Moreover, we denote by $\sfd_g$ the distance and by $\vol_g$ the volume form induced by $g$. \\
In local coordinates $g_{ij}$ will denote the coefficients of $g$ as a matrix, $g^{ij}$ the coefficients of the inverse matrix, and $|g|$  the determinant of $(g_{ij})_{ij}$. 
We will write $\Gamma$ for the Levi-Civita connection of $g$ and denote its coefficients by $\Gamma^k_{ij}$. \\
The covariant derivative with respect to $\Gamma$ will be denoted by either $\nabla$ or $\nabla_c$, with the only exception that $\nabla f = g^{ij} \partial_i f \partial_{x^j} = (df)^\sharp$ always denotes the gradient field of $f$ and never its covariant derivative (which in that particular case equals the differential $df$). \\
For a vector field $V$ and a differentiable function $f$, we define $V(f):= df(V)$ to be the action of $V$ on $f$.
For a vector field $V$ and a $k$-form $\omega \in \Omega^k(M)$ on $M$, where $k \geq 1$, we denote by $\iota_V\omega \in \Omega^{k-1}(M)$ the contraction of the first entry of $\omega$ with $V$.

Given a topological Hausdorff space $X$, we denote by $\mathrm{Meas}(X)$ the  Banach space of  finite Radon measures on $X$ equipped with the total variation norm $\norm{\cdot}_{TV}$. 
Given a finite dimensional vector space $V$ with an inner product $\langle\cdot, \cdot\rangle$, we denote by $|\cdot|_{HS}$ the Hilbert-Schmidt norm on the space of linear operators on $V$.

\section{Distributional calculus with a Riemannian metric of low regularity}

\subsection{Notation and basic definitions}
In this section, we recall some elements of the theory of distributions on manifolds, including the distributional Riemann tensor; for more details, we refer to \cite{graf2020singularity}, \cite{grosser2013geometric}, \cite{kunzinger2022synthetic}, and \cite{lefloch2007definition}.

Let $M$ be a smooth real manifold of dimension $n$ without boundary, and let $(E, M, \pi)$ be a vector bundle over $M$. To keep notation short, we will simply write $E$ instead of $(E, M, \pi)$.  For $0 \leq k \leq \infty$, we denote by $\Gamma^k(M, E)$ (resp. $\Gamma^k_c(M, E)$) the space of $C^k$-sections (resp. with compact support). In the case $k = \infty$, we often drop the superscript to simplify the notation. 
We denote $T^r_s(M):= (\bigotimes_{i=1}^r TM) \otimes (\bigotimes_{i=1}^s T^*M)$ and $\mathcal{T}^r_s := \Gamma(M,T^r_s(M))$. We will also write $\mathfrak{X}(M)$ for $\mathcal{T}^1_0(M)$ and $\mathfrak{X}^*(M)$ for $\mathcal{T}^0_1(M)$.
For some vector bundle $F$ and $V \subset M$ open, we say that a sequence $\omega_j$ converges to $\omega$ in $\Gamma^k_c(V, F)$ if there exists a compact set $K \subset V$ such that $\supp\, \omega_j, \supp\, \omega \subset K$ and $\nabla^l \omega_j \to \nabla^l \omega$ uniformly on $K$ for all $l \leq k$.

We denote by $\mathrm{Vol}(M)$ the vector bundle of $1$-densities, i.e. the one dimensional vector bundle with transition functions $ \Psi_{\alpha\beta} (x) = |\det D(\psi_\beta \circ \psi_\alpha^{-1})(x)|$, where $\psi_\alpha, \psi_\beta$ denote local charts into $\R^n$.  A section $\omega \in \Gamma_c^k(M, \mathrm{Vol}(M))$ is called   a \emph{test volume}.
The \emph{space of distributions of order $k$ on $M$} is defined as the topological dual of $\Gamma_c^k(M, \mathrm{Vol}(M))$ (\cite{grosser2013geometric}, Sec. 3.1), i.e. 
\begin{align*}
    {\mathcal{D}'}^{(k)}(M) := \Gamma_c^k(M, \mathrm{Vol}(M))'.
\end{align*}
The space of distributional $(r,s)$-tensor fields of order $k$ is given by 
\begin{align*}
    {\mathcal{D}'}^{(k)}\mathcal{T}^r_s(M) := \Gamma_c^k(M, T^s_r(M) \otimes \mathrm{Vol}(M))'.
\end{align*}
It is known that
    \begin{align*}
        {\mathcal{D}'}^{(k)}(\mathcal{T}^r_s)_{C^k}(M) \cong \mathcal{D}'(M) \otimes_{C^k(M)} {\mathcal{T}}^r_s(M) \cong L_{C^k(M)}(\mathfrak{X}^*(M)_{C^k}^r\times \mathfrak{X}(M)_{C^k}^s; \mathcal{D}'(M)),
    \end{align*}
    where the latter denotes the $C^k(M)$-module of $C^k(M)$-multilinear maps from  $\mathfrak{X}^*(M)^r_{C^k}\times  \mathfrak{X}(M)^s_{C^k}$ to ${\mathcal{D}'}^{(k)}(M)$  (cf. \cite{grosser2013geometric}).
      
\begin{definition}[Action of a vector field on a distribution, \cite{lefloch2007definition}]
    For a scalar distribution $T \in \mathcal{D}'(M)$ and a vector field $X$, we define the action of $X$ on $T$ via 
    \begin{align*}
        \langle X(T), \omega \rangle_{\mathcal{D}', \mathcal{D}} := -\langle T, \mathcal{L}_X\omega \rangle_{\mathcal{D}', \mathcal{D}},
    \end{align*}
    where $\omega$ is a test volume and $\mathcal{L}_X \omega$ denotes the Lie derivative of $\omega$ in the direction of $X$. 
\end{definition}

\begin{definition}[\cite{lefloch2007definition}, Def. 4.1]
    A distribution $g \in {\mathcal{D}'}\mathcal{T}^0_2(M)$ is called a generalised Riemannian metric, if it satisfies
    \begin{align*}
      &  g(X, Y)  = g(Y, X) \quad \mathrm{in} \ \mathcal{D}'(M) \quad \mathrm{and} \\
       &  g(X, Y) = 0 \quad  \forall Y\in \mathfrak{X}(M)  \quad \Longrightarrow \quad  X=0. 
    \end{align*}
\end{definition} 
Using a partition of unity, it is clear that one can localise and reduce the computations to distributions on $\R^n$.

In order to generalise the Riemann curvature tensor, it is appropriate to consider a Riemannian metric that is locally given as a positive definite matrix such that the coefficients $g_{ij}$ and the coefficients of the inverse matrix $g^{ij}$ are both in $L^\infty_{\rm{loc}}\cap H^1_{\rm{loc}}$ (cf. \cite{grosser2013geometric}, 5.2.1). 
This condition is automatically met, if $g \in C^0(M)$ and admits $L^2_{\rm{loc}}$-Christoffel symbols (see Lemma \ref{christoffel_l2_implies_dg_l2}). \\
In order to define a distributional covariant derivative, one generalises the first Koszul formula. For smooth vector fields $X, Y, Z$ one can define $\nabla^{\flat}_XY \in \mathcal{D}'\mathcal{T}^0_1(M)$ as:\begin{align*}
    \langle \nabla^{\flat}_XY, Z \rangle :=& \frac{1}{2}\left(X(g(Y, Z))+ Y(g(X, Z))-Z(g(X, Y))-g(X, [Y,Z])-g(Y, [X, Z])+g(Z, [X, Y])\right)
\end{align*}
in $\mathcal{D}'(M)$. Raising the index gives the distributional covariant derivative $\nabla_XY := (\nabla^{\flat}_XY)^\sharp$ and a computation in local coordinates yields that indeed
\begin{align*}
    (\nabla_XY)_s &= \delta_i^s(X^k\partial_kY^{i}) + \Gamma^s_{ik}X^{i}Y^{k}\ \mathrm{in}\ \mathcal{D}'(\R^n).
\end{align*}
For a $1$-form $\theta$, the distributional covariant derivative is defined via
\begin{align*}
    \langle \nabla_X\theta, Z \rangle := X(\langle \theta, Z \rangle)- \langle\theta, \nabla_XZ \rangle\ \mathrm{in}\ \mathcal{D}'(M),
\end{align*}
which, when spelled out in local coordinates, coincides with the classical covariant derivative of a $1$-form.
We have all the ingredients  to define the distributional Riemann tensor. Here, it is crucial that the coefficients of the covariant derivatives are in $L^2_{\rm{loc}}$, as the expression of the Riemann tensor involves a quadratic term in the Christoffel symbols.  Following the notation of \cite[Definition 3.3]{lefloch2007definition}, we have that 
\begin{align*}
    \langle\mathrm{Riem}(X, Y)Z, \theta \rangle    &= X \langle\nabla_YZ, \theta \rangle - Y \langle\nabla_XZ, \theta \rangle - \langle \nabla_YZ, \nabla_X\theta \rangle_{L^2} + \langle \nabla_XZ, \nabla_Y\theta \rangle_{L^2} - \langle \nabla_{[X, Y]}Z, \theta \rangle
\end{align*}
 in $\mathcal{D}'(M)$. All the quantities on the right hand side can be computed locally. Thus, in local coordinates, the Riemann tensor can be written as
    \begin{align}\label{eq:defRiemDistr}
        R^l_{ijk}= \partial_i\Gamma^l_{jk}-\partial_j\Gamma^l_{ik}+\Gamma^s_{kj}\Gamma^l_{is}-\Gamma^s_{ki}\Gamma^l_{js},
    \end{align}
    in the sense of distributions on $\R^n$. 
    Note that in local coordinates, the components of $X, Y, Z$ and $\theta$ ``are part of the test function''.  Note that, if $g$ is a smooth Riemannian metric, the expression \eqref{eq:defRiemDistr} coincides with the classical Riemann tensor (cf.\;\cite{petersen2006riemannian}). 
For a $C^{0}$-metric $g$ admitting $L^2_{\rm{loc}}$-Christoffel symbols, let $e_i$ be a $C^{0}\cap W^{1,2}_{\rm{loc}}$-local frame in $\mathfrak{X}_{C^0}(M)$. Note that this exists by Lemma \ref{christoffel_l2_implies_dg_l2}.
For two smooth vector fields $X, Y$, the distributional Ricci tensor of is then given by
\begin{align*}
    \Ric(X,Y) = \sum_i g(e_i, \Riem(e_i, X)Y).
\end{align*} 
It is immediate to check that $\Ric$ is a symmetric bilinear form that can  be expressed in local coordinates as:
\begin{align}\label{formuladistriricci}
     \Ric_{jk}=R^p_{pjk} = \partial_p\Gamma^p_{jk}-\partial_j\Gamma^p_{pk}+\Gamma^s_{kj}\Gamma^p_{ps}-\Gamma^s_{kp}\Gamma^p_{js},
\end{align}
which coincides with the classical expression in the smooth case. 
We next define distributional lower bounds of the Ricci curvature. We say that a one-density is non-negative if it can be written as $\phi \vol_g \in \mathrm{Vol}(M)$ for some non-negative $\phi \in C^\infty(M)$.
\begin{definition}
   Let $K\in \R$. We say that a $C^{0}$-Riemannian metric $g$ with $L^2_{\rm{loc}}$ Christoffel symbols satisfies $\Ric \geq K$ in the \emph{distributional sense} if, for all $X \in \mathfrak{X}(M)$, it holds $\Ric(X, X) \geq Kg(X, X)$ in  ${\mathcal{D}'}(M)$, i.e. for all non-negative test volumes $\omega$ it holds, 
    \begin{align*}
        \int \Ric(X, X)\omega \geq \int Kg(X, X)\omega.
    \end{align*}
\end{definition}
 We conclude this section by recalling how to perform convolutions on  a manifold.
 Fix a function $\rho \in C_c^\infty(\R^n)$ such that $\supp\, \rho \subset B_1(0)$, $\rho \geq 0$, and $\int_{\R^n} \rho = 1$. Define $\rho_\e(x):= \frac{1}{\e^n} \rho(\frac{x}{\e})$. For a smooth manifold $M$, fix an atlas $(U_\alpha, \psi_\alpha)$ such that $(U_\alpha)_\alpha$ is locally finite and each $U_{\alpha}$ is relatively compact, and fix a partition of unity $\eta_\alpha$ subordinate to the atlas. Choose a family of smooth cutoff functions $\chi_\alpha$ such that $\supp\, \chi_\alpha \subset U_\alpha$ and $\chi_\alpha =1$ on $\supp\, \eta_\alpha$. For any tensor $T \in L^1_{\rm{loc}}(T^r_s(M))$, define 
 \begin{align}\label{convolution_on_manifold_def}
    T* \rho_\e (x) := \sum_\alpha \chi_\alpha(x)\psi_\alpha^*\big(({\psi_\alpha}_*(\eta_\alpha T))* \rho_\e\big)(x).
 \end{align}
 {
 \subsection{Some regularisation results and their application to Sobolev Riemannian metrics}\label{subsec:regularisation}
In this section, we  present some variants of the convergence results obtained in \cite{graf2020singularity, calisti2025hawking}, which will be key to obtain  $L^p$-convergence of the Ricci curvature tensor under regularisation of the Riemannian metric. More precisely, we will follow the strategy of \cite[Lemma 3.3]{calisti2025hawking} with the following difference: we will work with $L^p$-functions instead of locally Lipschitz functions, and we will keep track of the exponents. The following lemma generalises  \cite[Lemma 4.7]{graf2020singularity} to the $L^p$ case. 
\begin{lemma}\label{friedrichs_lemma_lp}
   Let $f \in L^p_{loc}(\R^n)$, for some $p \in [1, \infty)$. Let $\rho \in C_c^\infty(B_1^{euc}(0), [0, \infty))$ be a rotationally symmetric standard mollifier and define $\rho_\e(x):= \frac{1}{\e^n}\rho(\frac{x}{\e})$. Then, for any compact set $K \subset\R^n$, it holds that 
    $$
    \e\norm{\partial_j\rho_\e*f}_{L^p(K)} \to 0, \quad \text{ as } \e \to 0.
    $$
\end{lemma}
\begin{proof}
    Note that, for any constant $c \in \R$, integration by parts gives that $\partial_j \rho_\e*c=0$.
    Hence, using the definition of $\rho_\e$, Jensen's inequality, and Fubini's theorem, we get that
    \begin{align*}
        \e \Big(\int_K |\partial_j\rho_\e*f|^p \Big)^{\frac{1}{p}} &= \e \Big(\int_K \Big|\int_{B_\e(0)}\partial_j\rho_\e(y) \big(f(y-x)-f(x)\, \big)d y\Big|^p d x \Big)^{\frac{1}{p}} \\
        &=  \Big(\int_K \Big|\int_{B_\e(0)}\frac{\e}{\e^{n+1}}\partial_j\rho\Big(\frac{y}{\e}\Big)\big(f(y-x)-f(x)\big)\, d y\Big|^p d x \Big)^{\frac{1}{p}} \\
        &\leq C(n) \Big(\int_K\frac{1}{\e^n} \int_{B_\e(0)}|\partial_j\rho\Big(\frac{y}{\e}\Big)|^p|f(y-x)-f(x)|^p\, d y d x \Big)^{\frac{1}{p}} \\
        &\leq C(n, \rho) \Big(\frac{1}{\e^n}\int_{B_\e(0)} \int_K |f(y-x)-f(x)|^p\, d y d x \Big)^{\frac{1}{p}}.
   \end{align*}
   The last integral converges to $0$, since 
   \begin{align*}
      \lim_{h \to 0} \norm{f-f(\cdot-h)}_{L^p(K)} =0. 
   \end{align*}
\end{proof}
\begin{lemma}[A Friedrichs type lemma]\label{friedrichs_lemma_applied_1}
    Let $p \in [2, \infty)$ and $a \in C^0 \cap W^{1, p}_{loc}(\R^n)$ and $f \in L^p_{loc}(\R^n)$. Then, for any compact set $K$, it holds that 
    \begin{align*}
        \norm{(a*\rho_\e)(\rho_\e*f) - \rho_\e*(af)}_{W^{1, \frac{p}{2}}(K)} \to 0, \quad \text{ as } \e\to 0.
    \end{align*}
\end{lemma}
\begin{proof}
    We will proceed as in the proof of  \cite[Lemma 4.8]{graf2020singularity}, see also \cite[Sect.\;2]{CDLS-2012}. Fix a compact set $K$. We need to estimate the $L^{\frac{p}{2}}(K)$-norm of 
    \begin{align*}
        h_\e(x) &= ((af)* \partial_j\rho_\e - (\partial_j\rho_\e*a)(f*\rho_\e) - (\rho_\e*a)(f*\partial_j\rho_\e)) (x) \\
        &=\underbrace{\big(((a-a(x))(f-f(x)))* \partial_j \rho_\e\big)(x)}_{=:F_1(x)} - \underbrace{\big((a * \partial_j \rho_\e)((f-f(x))*\rho_\e)\big)(x)}_{=:F_2(x)}-\underbrace{\big(((a-a(x))*\rho_\e)(f*\partial_j\rho_\e)\big)(x)}_{=:F_3(x)}.
    \end{align*}
Notice that the second line follows by the fact that $c * \rho_\e=c$ and $c* \partial_j \rho_\e=0$, for every constant $c\in \R$.
   We start by estimating $F_1$.
   \begin{align*}
     \norm{F_1}_{L^{\frac{p}{2}}(K)} &= \Big(\int_K \big|\big(((a-a(x))(f-f(x)))* \partial_j \rho_\e\big)(x)\big|^{\frac{p}{2}}dx\Big)^{\frac{2}{p}}\\
       &= \Big(\int_K \Big|\int_{B_\e(0)}((a(x-y)-a(x))(f(x-y)-f(x)))\; \partial_j \rho_\e(y)dy\Big|^{\frac{p}{2}}dx\Big)^{\frac{2}{p}} \\
       &= \Big(\int_K \Big|\frac{1}{\e^n}\int_{B_\e(0)}\frac{1}{\e}((a(x-y)-a(x))(f(x-y)-f(x)))\; \partial_j \rho\Big(\frac{y}{\e}\Big)dy\Big|^{\frac{p}{2}}dx\Big)^{\frac{2}{p}} \\
       &\leq C(n) \Big(\int_K \frac{1}{\e^n}\int_{B_\e(0)}\Big|\frac{1}{\e}((a(x-y)-a(x))(f(x-y)-f(x)))\; \partial_j \rho\Big(\frac{y}{\e}\Big)\Big|^{\frac{p}{2}}dydx\Big)^{\frac{2}{p}} \\
        &\leq C(n, \rho) \Big( \frac{1}{\e^n}\int_{B_\e(0)}\int_K\Big|\frac{1}{\e}((a(x-y)-a(x))(f(x-y)-f(x)))\Big|^{\frac{p}{2}}dxdy\Big)^{\frac{2}{p}}\\
        &\leq  C(n, \rho) \Big(\Big(\frac{1}{\e^n}\int_{B_\e(0)}\int_K\Big|\frac{1}{\e}(a(x-y)-a(x))\Big|^pdxdy\Big)^{\frac{1}{2}}\Big(\frac{1}{\e^n}\int_{B_\e(0)}\int_K\Big|(f(x-y)-f(x))\Big|^pdxdy\Big)^{\frac{1}{2}}\Big)^{\frac{2}{p}}.
   \end{align*}
   It is a classical fact of Sobolev functions that 
   \begin{equation}\label{eq:nablaALp}
   \frac{1}{\e^n} \int_{B_\e(0)}\int_K\Big|\frac{1}{\e}(a(x-y)-a(x))\Big|^pdxdy\leq C(n) \norm{\nabla a}_{L^p(B_\e(K))}^p.
   \end{equation}
Indeed, by the density of $C^1$-functions in $W^{1,p}$, we can assume without loss of generality that $a\in C^1$. Using the fundamental theorem of calculus,  H\"older's inequality, and Fubini's theorem,  we estimate:
   \begin{align*}
      \frac{1}{\e^n} \int_{B_\e(0)}\int_K\Big|\frac{1}{\e}(a(x-y)-a(x))\Big|^pdxdy
      &\leq \frac{1}{\e^n} \int_{B_\e(0)}\int_K \Big|\frac{1}{\e}\int_0^{|y|} \nabla a\Big(x-t\frac{y}{|y|}\Big) dt \Big|^pdxdy \\
      &\leq \frac{1}{\e^n} \int_{B_\e(0)}\frac{|y|^{p-1}}{\e^p}\int_K \int_0^{|y|} \Big| \nabla a\Big(x-t\frac{y}{|y|}\Big) \Big|^p dt dxdy  \\
      &\leq \frac{1}{\e^n} \int_{B_\e(0)}\frac{|y|^{p-1}}{\e^p} \int_0^{|y|}\int_K \Big| \nabla a\Big(x-t\frac{y}{|y|}\Big) \Big|^p dx dt dy  \\
      &\leq \frac{1}{\e^n} \int_{B_\e(0)}\frac{\e^{p}}{\e^p} \norm{\nabla a}_{L^p(B_\e(K))}^pdy \leq C(n) \norm{\nabla a}_{L^p(B_\e(K))}^p,
   \end{align*}
   proving \eqref{eq:nablaALp}.
   Hence,  for $\e \leq 1$, we get that
   \begin{align*}
       \norm{F_1}_{L^{\frac{p}{2}}(K)}& \leq C(n,p)\norm{\nabla a}_{L^p(B_1(K))}\Big(\Big(\frac{1}{\e^n}\int_{B_\e(0)}\int_K\Big|(f(x-y)-f(x))\Big|^pdxdy\Big)^{\frac{1}{2}}\Big)^{\frac{2}{p}} \\
       & \leq C(n,p)\norm{\nabla a}_{L^p(B_1(K))} \Big( \sup_{|y|\leq \e} \norm{f-f(\cdot-y)}_{L^p(K)}^p \Big)^{\frac{1}{p}} \to 0,
   \end{align*}
   as $\e \to 0$. Furthermore,
   \begin{align*}
      \norm{F_2}_{L^{\frac{p}{2}}(K)} &\leq \norm{(a * \partial_j \rho_\e)}_{L^p(K)} \norm{((f-f(x))*\rho_\e)}_{L^p(K)}\nonumber \\
      & \leq \norm{\nabla a}_{L^p(B_1(K))} \Big(\int_K \Big|\int_{B_\e(0)} \rho_\e(y)((f(x-y)-f(x))dy \Big|^pdx \Big)^{\frac{1}{p}}\nonumber \\
      &\leq \norm{\nabla a}_{L^p(B_1(K))} \Big(\int_{B_\e(0)} \rho_\e(y)\int_K \Big|((f(x-y)-f(x))\Big|^pdx dy \Big)^{\frac{1}{p}}\nonumber \\
      &\leq  \norm{\nabla a}_{L^p(B_1(K))} \Big( \sup_{|y|\leq \e} \norm{f-f(\cdot-y)}_{L^p(K)}^p \Big)^{\frac{1}{p}} \to 0,
   \end{align*}
   as $\e \to 0$. We next estimate $F_3$. To this aim, we start by estimating 
   \begin{align}\label{Lp-difference_a_1}
       \norm{(a-a(x))*\rho_\e}_{L^p(K)} &= \Big(\int_K \Big|\int_{B_\e(0)}\rho_\e(y)(a(x-y)-a(x))dy\Big|^pdx\Big)^{\frac{1}{p}}\nonumber \\
       &\leq \Big(\int_K \int_{B_\e(0)}\rho_\e(y)|(a(x-y)-a(x))|^p dy dx\Big)^{\frac{1}{p}}\nonumber \\
       &= \Big(\int_{B_\e(0)}\rho_\e(y)\int_K |(a(x-y)-a(x))|^p dx dy\Big)^{\frac{1}{p}}.
   \end{align}
   For $|y|\leq \e$,   it is a classical fact for Sobolev functions that
   \begin{equation}\label{Lp-difference_a_2}
     \Big(\int_K |(a(x-y)-a(x))|^p dx\Big)^{\frac{1}{p}}\leq  \e \norm{\nabla a}_{L^p(B_1(K))}.
   \end{equation}
   Indeed, by the density of $C^1$-functions in $W^{1,p}$, we can assume without loss of generality that $a\in C^1$.  Denote $v = \frac{y}{|y|}\in \R^n$. Using the fundamental theorem of calculus,  H\"older's inequality, and Fubini's theorem,  we estimate:   
   \begin{align*}
     \Big(\int_K |(a(x-y)-a(x))|^p dx\Big)^{\frac{1}{p}} &= \Big(\int_K \Big|\int_0^{|y|} \nabla a(x-tv)dt\Big|^p dx \Big)^{\frac{1}{p}} \nonumber\\
     &\leq \Big(\int_K|y|^{p-1}\int_0^{|y|} |\nabla a|^p(x-tv)dt dx \Big)^{\frac{1}{p}} \nonumber \\
     &\leq \Big(\e^{p-1}\int_0^{|y|}\int_K |\nabla a|^p(x-tv) dx dt\Big)^{\frac{1}{p}} \nonumber \\
     &\leq \Big(\e^{p-1}\int_0^{|y|}\norm{\nabla a}_{L^p(B_1(K))}^p dt\Big)^{\frac{1}{p}} \leq \e \norm{\nabla a}_{L^p(B_1(K))}.
     \end{align*}
Combining the estimates above with H\"older's inequality, we obtain
   \begin{align*}
        \norm{F_3}_{L^{\frac{p}{2}}(K)} &\leq  \norm{(a-a(x))*\rho_\e}_{L^p(K)} \norm{\partial_j\rho_\e*f}_{L^p} \leq \e \norm{\nabla a}_{L^p(B_1(K))}\norm{\partial_j\rho_\e*f}_{L^p} \to 0,
   \end{align*}
   as $\e \to 0$, by Lemma \ref{friedrichs_lemma_lp}.
\end{proof}

\begin{lemma}\label{friedrichs_application_final}
    Let $p \in [2, \infty)$, $a \in C^0 \cap W^{1, p}_{loc}(\R^n)$ and $f \in L^p_{loc}(\R^n)$. Let $\rho_\e$ be as in the previous lemma. Let $a_\e \in C^\infty(\R^n)$ be such that $a_\e \to a$ locally uniformly and in $W^{1,p}_{loc}$; moreover, assume that, for each $K\subset\R^n$ compact subset there exists $C_K>0$ such that
    \begin{equation}\label{eq:HpAeaLp}
    \norm{a_\e-a}_{L^{p}(K)}\leq C_K \e.
    \end{equation}
    Then, for any compact set $K\subset\R^n$, it holds that 
    \begin{align*}
        \norm{a_\e(\rho_\e*f) - \rho_\e*(af)}_{W^{1, \frac{p}{2}}(K)} \to 0, \quad \text{as } \e \to 0.
    \end{align*}
\end{lemma}
\begin{proof}
    Write 
    \begin{align*}
        a_\e(\rho_\e*f) - \rho_\e*(af)= \underbrace{(a_\e-\rho_\e*a)(\rho_\e*f)}_{(I)}+\underbrace{(\rho_\e*a)(\rho_\e*f) - \rho_\e*(af)}_{(II)}.
    \end{align*}
    We only need to estimate the $W^{1,p}$-norm of term $(I)$, as Lemma \ref{friedrichs_lemma_applied_1} yields the desired estimate for term $(II)$.
    Write
    \begin{align*}
        \partial_j((a_\e-\rho_\e*a)(\rho_\e*f)) = (\partial_j(a_\e-\rho_\e*a))(\rho_\e*f) + (a_\e-\rho_\e*a)(\partial_j(\rho_\e*f)). 
    \end{align*}
    We start estimating the first term:
    \begin{align*}
        \norm{(\partial_j(a_\e-\rho_\e*a))(\rho_\e*f)}_{L^\frac{p}{2}(K)} &\leq \norm{\partial_j(a_\e-\rho_\e*a)}_{L^p(K)}\norm{(\rho_\e*f)}_{L^p(K)} \\
        &\leq \big(\norm{\partial_j(a_\e-a)}_{L^p(K)}+\norm{\partial_j(a-\rho_\e*a)}_{L^p(K)}\big)\;\norm{(\rho_\e*f)}_{L^p(K)}  \to 0,
    \end{align*}
    as $\e \to 0$. Moreover, for the second term, we use the assumption \eqref{eq:HpAeaLp} together with \eqref{Lp-difference_a_1} and \eqref{Lp-difference_a_2} and get that
    \begin{align*}
        \norm{a_\e-(a*\rho_\e)}_{L^p(K)} \leq \norm{a_\e-a}_{L^p(K)} + \norm{a-(a*\rho_\e)}_{L^p(K)} \leq C(K)\e. 
    \end{align*}
    Hence, 
    \begin{align*}
        \norm{ (a_\e-\rho_\e*a)(\partial_j(\rho_\e*f))}_{L^{\frac{p}{2}}(K)} \leq \norm{(a_\e-\rho_\e*a)}_{L^p(K)} \norm{(\partial_j(\rho_\e*f)}_{L^p(K)} \leq C_K\e\norm{(\partial_j(\rho_\e*f)}_{L^p(K)} \to 0 
    \end{align*}
    as $\e \to 0$, by Lemma \ref{friedrichs_lemma_lp}.
\end{proof}
Now, assume to have a Riemannian metric $g\in C^0\cap W^{1,p}_{loc}$, for some $p \in [2, \infty)$. As usual,  we write $g_{ij}$ for the coefficients of $g$ in local coordinates. Let $g_\e= \rho_\e*g$. We note that for all $i,j,k,l,m =1, \ldots, n$, the functions $f= \partial_k g_{lm}$, $a=g^{ij}$ and $a_\e:= ((g_\e)^{-1})_{ij}$ satisfy the assumptions of Lemma \ref{friedrichs_application_final}. This can be seen directly from \eqref{Lp-difference_a_1} and \eqref{Lp-difference_a_2}, together with the formula of the inverse coefficients of a matrix.
Hence, recalling the local expression \eqref{formuladistriricci} of the Ricci curvature tensor, together with Lemma \ref{friedrichs_application_final}, we get the following useful result.
\begin{proposition}\label{lp_ricci_conv}
    Let $M$ be a smooth manifold and let $g$ be a Riemannian metric of regularity $g \in W^{1,p}_{loc}$, for some $p \in [2, \infty)$. Then, for each compact set $K\subset M$, it holds
    \begin{align}
        \norm{\Ric[g_\e]-\rho_\e*\Ric[g]}_{L^{\frac{p}{2}}(K)} \to 0, \quad \text{as }\e \to 0.
    \end{align}
\end{proposition}
 
 }
\section{A weak formulation of Bochner's formula for metrics of low regularity} 
\subsection{Smooth calculus on manifolds}
In this subsection, we will recall some basic calculus concepts from smooth Riemannian geometry. More details can be found in \cite{petersen2006riemannian}.
\begin{lemma}
Let $g$ be a smooth Riemannian metric on $M$ and $f, h$ be smooth functions on $M$. The following equations hold in local coordinates:
\begin{itemize}
    \item[$\mathrm{(i)}$] $\nabla f = g^{ik}(\partial_if)\partial_{x^k}$. 
    \item[$\mathrm{(ii)}$] $\vol_g = \sqrt{|g|}dx^1 \land \ldots \land dx^n$.
    \item[$\mathrm{(iii)}$] $\Delta f = \partial_jg^{ij}\partial_if + g^{ij} \partial_{ij}f + \frac{1}{2}\mathrm{tr}(g^{-1}\partial_i g)g^{ij}\partial_jf= \frac{1}{\sqrt{|g|}}\partial_i(\sqrt{|g|}g^{ij}\partial_jf)$.
    \item[$\mathrm{(iv)}$] $\langle\nabla f, \nabla h \rangle_g = \partial_if g^{ij} \partial_jh$.
\end{itemize}
\end{lemma}
Note that these expressions still hold for $g \in C^{0}(M)$, $\Gamma^k_{ij} \in L^2_{\rm{loc}}$.
The Hessian is given by
\begin{align*}
   \mathrm{Hess} f (\partial_{x^{i}}, \partial_{x^j}) &= \partial_{ij}f - \Gamma^s_{ij}\partial_sf.
\end{align*}
For all smooth functions $\phi, h$, and $f$, it holds
\begin{align}\label{hessian_for_gradients}
    2\Hess f(\nabla \phi, \nabla h) = \langle \nabla \phi ,\langle \nabla f , \nabla h\rangle_g \rangle_g + \langle \nabla h,\langle \nabla f , \nabla \phi \rangle_g \rangle_g - \langle \nabla f ,\langle \nabla h, \nabla \phi\rangle_g \rangle_g.
\end{align}
Now let $X$ be a smooth vector field and $\omega$ a smooth $1$-form. Then the musical isomorphisms $\sharp: T^*M \to TM$ and $\flat:  TM \to  T^*M$ are defined via
\begin{align*}
    X^\flat(Y) &= \langle X, Y\rangle_g, \ \mathrm{and} \ \omega (Y) = \langle \omega^\sharp, Y\rangle_g, 
\end{align*}
for all vector fields $Y$. Note that the musical isomorphisms come from the smooth setting, but also make sense for $C^{0}$-Riemannian metrics. 
\subsection{Bochner's formula for vector fields on smooth weighted manifolds}
In this subsection we briefly recall the generalisation of Bochner's formula to the case of a smooth weighted manifold. Denote by $\Delta_H$ the Hodge-Laplacian and recall that for $f \in C^\infty(M)$ it holds $\Delta_H f= -\Delta f$. Then:
\begin{proposition}
    For a smooth manifold $M$ with a smooth Riemannian metric $g$ and any smooth vector field $X$, it holds
    \begin{align}\label{general_bochner}
        -\frac{1}{2} \Delta_H |X|^2 = -\langle (\Delta_H X^\flat)^\sharp, X \rangle + |\nabla X|^2_{\mathrm{HS}} + \Ric(X, X). 
    \end{align}
\end{proposition}
Now let $(M, g)$ be a smooth Riemannian manifold and let $\mu$ be a measure on $M$ such that $\di\mu = h^2\di\vol_g$ for a smooth and positive function $h$.  
As usual, we write for a $C^1$-function $u$, that $(\nabla u)^i = g^{ij}\partial_j u$.
We can now define the associated divergence operator $\mathrm{div}_\mu$ by 
\begin{align}\label{weighted_divergence_smooth}
    \mathrm{div}_\mu(X) = \frac{1}{h^2}\mathrm{div}(h^2X) = \frac{1}{h^2}\Big(h^2\partial_iX^{i} + 2h\partial_i h X^{i} + \frac{1}{2}h^2 X^{i}\mathrm{tr}(g^{-1}\partial_ig)\Big)
\end{align}
and the weighted Laplace operator $\Delta_\mu$
\begin{align*}
    \Delta_\mu := \mathrm{div}_\mu \circ \nabla = \frac{1}{h^2} \mathrm{div}(h^2\nabla) = \Delta + \frac{2\langle \nabla h, \nabla \rangle}{h}.
\end{align*}
The divergence operator satisfies
\begin{align}\label{divergence_int_by_parts}
    \int_M f \mathrm{div}_\mu(X) \, \di\mu = -\int_M X(f) \, \di\mu,
\end{align}
for any function $f \in C_c^\infty(M)$ and any smooth vector field $X$.
It then follows that for $u, v \in C^\infty(M)$ such that $u$ or $v$ are compactly supported, we have that 
\begin{align*}
    \int_M u \Delta_\mu v \, \di\mu = -\int_M \langle \nabla u, \nabla v \rangle_g \, \di\mu = \int_M  \Delta_\mu u v \, \di\mu.
\end{align*}
Sometimes it is notationally convenient to write $h^2 = e^{-V}$ for some smooth function $V$.
\begin{definition}\label{smooth_generalized_ricci}
 Let $N \geq n$ and $V, h$ as above. The generalised Bakry-Émery $N$-Ricci tensor is defined by 
 \begin{align*}
     \Ric_{\mu, N} := \Ric + \nabla^2 V - \frac{1}{N-n}\nabla V \otimes \nabla V.
 \end{align*}
In the case $n = N$, we use the convention that the only admissible $V$ is constant. 
 For $N = \infty$, we get 
 \begin{align*}
     \Ric_{\mu, \infty} := \Ric + \nabla^2 V = \Ric + \Hess[-2\log h].
 \end{align*}
\end{definition}
We will now generalise the Hodge star operator and the codifferential to this setting. 
Recall the Hodge star operator $*: \Omega^p(M) \to \Omega^{n-p}(M)$ in a Riemannian manifold without weight, where $*\beta$ is defined via $\alpha \wedge * \beta = \langle \alpha, \beta\rangle \vol_g$ for all $\alpha \in \Omega^p(M)$. This shall motivate our next definition.
\begin{definition}
    The weighted Hodge star operator $*_\mu:\Omega^p(M) \to \Omega^{n-p}(M)$ is defined as $h^2*$, where $*$ denotes the classical Hodge star operator. It satisfies 
    \begin{align*}
         \alpha \wedge *_\mu \beta = \langle \alpha, \beta \rangle_g \mu,
    \end{align*}
    where we interpret the $n$-form $\alpha \wedge *_\mu \beta$ as a measure on $M$. 
\end{definition}
It directly follows that for all $p \in \{0, \ldots,  n\}$ and $\beta \in \Omega^p(M)$, it holds
\begin{align*}
    *_\mu^{-1} \beta = *^{-1} \left(\frac{1}{h^2}\beta\right).
\end{align*}
As in the unweighted case, we get that $\delta_\mu := (-1)^p *_\mu^{-1} d *_\mu$ is adjoint to $d$ in the $L^2(M, \mu)$-inner product, i.e.
\begin{align*}
    \int_M \langle \di\alpha,\beta \rangle_g \, \di\mu &= (-1)^p \int_M \langle \alpha, *_\mu^{-1} d *_\mu \beta\rangle_g \, \di\mu,
\end{align*}
for all compactly supported $(p-1)$-form $\alpha$ and  $p$-form $\beta$.
\begin{definition}
    The above operator $\delta_\mu: \Omega^{\bullet}(M) \to \Omega^{\bullet -1}(M)$ is called the weighted codifferential. We denote by $\Delta_{\mu, H}:\Omega^{\bullet}(M) \to \Omega^{\bullet}(M)$ the weighted Laplace-Beltrami operator, which is defined by 
    \begin{align*}
        \Delta_{\mu, H} := \di\delta_\mu + \delta_\mu d.
    \end{align*}
\end{definition}
We now aim to generalise \eqref{general_bochner} to the weighted case.
In local normal coordinates, we can  compute that 
\begin{align*}
    - \frac{1}{2} \Delta_{\mu, H}|\alpha|^2 =  \frac{1}{2} \Delta_{\mu}|\alpha|^2 = \frac{1}{2} \Delta|\alpha|^2 + \frac{1}{h}\langle \nabla h , \nabla |\alpha|^2 \rangle = -\frac{1}{2} \Delta_H|\alpha|^2 + \frac{2}{h}\partial_j h \alpha_k \partial_j \alpha_k
\end{align*}
and 
\begin{align}\label{weightedriccidef}
    \Ric_{\mu, \infty}(\alpha^\sharp, \alpha^\sharp) = \Ric(\alpha^\sharp, \alpha^\sharp) + \Hess[-2\log h](\alpha^\sharp, \alpha^\sharp) = \Ric(\alpha^\sharp, \alpha^\sharp) - 2\alpha_j \alpha_k \Big(\frac{h\partial^2_{jk}h - \partial_jh\partial_kh}{h^2}\Big).
\end{align}
Finally, we get that
\begin{align*}
    -\langle\alpha, \Delta_{\mu, H} \alpha \rangle &= -\langle\alpha, \Delta_{H} \alpha \rangle +\frac{2}{h}( \partial_jh\partial_j \alpha_k \alpha_k - \partial_kh\partial_j \alpha_k \alpha_j) + \frac{2}{h}\partial_j\alpha_k\partial_kh\alpha_j + \partial_j \left(\frac{2\partial_k h}{h} \right)\alpha_k\alpha_j \\
    &= -\langle\alpha, \Delta_{H} \alpha \rangle +\frac{2}{h}\partial_jh\partial_j \alpha_k \alpha_k - \frac{2}{h}\partial_kh\partial_j \alpha_k \alpha_j + \frac{2}{h}\partial_j\alpha_k\partial_kh\alpha_j + \alpha_k\alpha_j\Hess[2\log h]_{kj} \\
    &= -\langle\alpha, \Delta_{H} \alpha \rangle +\frac{2}{h}\partial_jh\partial_j \alpha_k \alpha_k + \alpha_k\alpha_j\Hess[2\log h]_{kj}.
\end{align*}
Recalling \eqref{general_bochner}, this yields
\begin{proposition}\label{weightedbochnersformula}
    Let $(M, g, \mu)$ be a weighted Riemannian manifold with smooth metric and smooth weight. Then, for every smooth $1$-form $\alpha$, it holds
    \begin{align*}
        - \frac{1}{2} \Delta_{\mu, H}|\alpha|^2 = -\langle\alpha, \Delta_{\mu, H} \alpha \rangle + \Ric_{\mu, \infty}(\alpha^\sharp, \alpha^\sharp) + |\nabla \alpha|^2_{HS}.
    \end{align*}
    In other terms, for every smooth vector field $X$, it holds
    \begin{align}\label{weighted_bochner_vectorfield}
       - \frac{1}{2} \Delta_{\mu, H}|X|^2 = -\langle X, (\Delta_{\mu, H} X^\flat)^\sharp \rangle + \Ric_{\mu, \infty}(X, X) + |\nabla X|^2_{HS}. 
    \end{align}
\end{proposition}
\subsection{A Bochner formula for weighted manifolds of lower regularity}\label{SubSec:BochnerWRM}
We can generalise Definition \ref{smooth_generalized_ricci} to the distributional case:
\begin{definition}
Let $M$ be a smooth manifold and $N \in [n, \infty]$. Let $g$ be a $C^{0}$-Riemannian metric with $L^2_{\rm{loc}}$  Christoffel symbols and let $h \in C^{0}(M, (0, \infty))\cap W^{1,2}_{\rm{loc}}(M), V \in C^{0}\cap W^{1,2}_{\rm{loc}}(M)$ such that $h^2 = e^{-V}$. We define the distributional Bakry-Émery $N$-Ricci-curvature tensor of the weighted manifold $(M,g,h^{2} \vol_{g})$ as
    \begin{align}\label{defdistririccifiniteN}
       \Ric_{\mu, N} &=  \Ric + \nabla_g^2 V - \frac{1}{N-n}\nabla V \otimes \nabla V \in {\mathcal{D}'}\mathcal{T}^0_2.
    \end{align}
    The case $N= \infty$ can also be rewritten as:
    \begin{align}\label{defdistributionalweightedricci}
    \Ric_{\mu, \infty} := \Ric -2 \Hess[\log h ] \in {\mathcal{D}'}\mathcal{T}^0_2.
\end{align}
\end{definition}
In the following, we will integrate by parts, to get a version of Bochner's formula that only involves first derivatives of $g$ and $h$.  
First we make the following useful observation: Assume $g$ to be a smooth metric on the manifold $M$. 
Then the Christoffel symbols (of the second kind) are given by 
\begin{align*}
    \Gamma^k_{ij}= \frac{1}{2}g^{kl}(\partial_ig_{jl} + \partial_jg_{il} - \partial_lg_{ij}).
\end{align*}
We can recover the Christoffel symbols of the first kind by multiplication with $g$, i.e.
\begin{align}\label{christoffel_first_kind_formula}
    \Gamma_{ij,l}=\frac{1}{2}(\partial_ig_{jl} + \partial_jg_{il} - \partial_lg_{ij})= g_{lk}\Gamma^k_{ij} = g(\nabla_{\partial_i}\partial_j, \partial_l).
\end{align}
Now note that 
\begin{align}\label{partial_derivative_in_terms_of_christoffel}
    (\Gamma_{il,j}+\Gamma_{jl,i})= \partial_lg_{ij}. 
\end{align}
This follows directly from the definition of the Levi-Civita connection, hence the last identity holds in a weak sense for the distributional Levi-Civita connection for $g$ of lower regularity. More precisely:
\begin{lemma}\label{christoffel_l2_implies_dg_l2}
    Let $g$ be a $C^0$-Riemannian metric such that $g$ admits $L^2_{\rm{loc}}$-Christoffel symbols (of the second kind). Then $g\in W^{1,2}_{\rm{loc}}(M)$.  
\end{lemma}
\begin{proof}
    As $g$ is in $C^0$, it follows from \eqref{christoffel_first_kind_formula}, that the Christoffel symbols of the first kind are in $L^2_{\rm{loc}}$ as well. Then using \eqref{partial_derivative_in_terms_of_christoffel} integration by parts and the linearity of the Christoffel symbols of the first kind in terms of $g$, we get that all partial derivatives of $g$ are in $L^2_{\rm{loc}}$. 
\end{proof}

Recall that \eqref{weighted_bochner_vectorfield} holds on a smooth manifold. Using that for functions $\Delta_\mu = - \Delta_{\mu, H}$ and testing with a function $\phi \in C_c^\infty(M)$, we obtain  
\begin{align*}
    \int_M \left( {-\frac{1}{2} \Delta_\mu |X|^2}  -{\langle (\Delta_{\mu, H} X^\flat)^\sharp, X \rangle} + {|\nabla X|^2} + {\Ric_{\mu, \infty}(X, X)}\right)\phi\, \di\mu = 0.
\end{align*}
We will again integrate by parts, where necessary, to spell this out in local coordinates. By potentially using a partition of unity, we can assume that $\phi$ is supported in one coordinate patch. We have that
\begin{align*}
    \int \frac{1}{2}\langle \nabla|X|^2, \nabla \phi \rangle_g \, \di\mu = \int -\frac{1}{2} \Delta_\mu |X|^2\phi\, \di\mu
\end{align*}
and
\begin{align*}
     \int_M \Ric_{\mu, \infty}(X,X)\phi \, \di\mu =& \int X^jX^k(\partial_p\Gamma^p_{jk}-\partial_j\Gamma^p_{pk}+\Gamma^s_{kj}\Gamma^p_{ps}-\Gamma^s_{kp}\Gamma^p_{js}) \phi h^2\sqrt{|g|}\, dx^1\ldots dx^n \\
     &+ 2\int X^jX^k(\partial_jh\partial_kh - h\partial^2_{jk}h+h\partial_sh\Gamma^s_{kj})\phi\sqrt{|g|}dx^1\ldots dx^n \\
     =& \int  X^jX^k (\Gamma^s_{kj}\Gamma^p_{ps}-\Gamma^s_{kp}\Gamma^p_{js}) \phi h^2 \sqrt{|g|}\, dx^1\ldots dx^n \\
     &- \int \Gamma^p_{jk}\partial_p(X^jX^k \phi h^2 \sqrt{|g|})\, dx^1\ldots dx^n +\int \Gamma^p_{pk}\partial_j(X^jX^k \phi h^2\sqrt{|g|})\, dx^1\ldots dx^n \\
     &+ 2\int X^jX^k(\partial_jh\partial_kh+h\partial_sh\Gamma^s_{kj})\phi\sqrt{|g|}dx^1\ldots dx^n \\
     &+ 2\int \partial_jh\partial_k(X^jX^kh\phi\sqrt{|g|})dx^1\ldots dx^n.
\end{align*}
The above holds for any smooth $g$ and $h$.  Using that each function $h \in C^0\cap W^{1,2}_{\rm{loc}}$ and each metric $g \in C^{0}$ with $L^2_{\rm{loc}}$-Christoffel symbols can be approximated  {in the $C^0\cap W^{1,2}_{\rm{loc}}$ topology} by smooth objects, we get:

\begin{proposition}
Let $M$ be a smooth manifold,  $g$  a $C^{0}$-Riemannian metric with $L^2_{\rm{loc}}$  Christoffel symbols and  $h \in C^{0}(M, (0, \infty))\cap W^{1,2}_{\rm{loc}}(M), V \in C^{0}\cap W^{1,2}_{\rm{loc}}(M)$ such that $h^2 = e^{-V}$. 
\\Then, for every smooth vector field $X$, it holds:
\begin{align}
   -\frac{1}{2} \Delta_\mu |X|^2 -\langle (\Delta_H X^\flat)^\sharp, X \rangle + |\nabla X|^2 + \Ric_{\mu, \infty}(X, X)= 0 \quad  \mathrm{in}\ \mathcal{D}'(M). 
\end{align}
\end{proposition}

For what will come next, it is convenient to investigate the other terms of the formula more closely:
\begin{align*}
    \langle (\Delta_H X^\flat)^\sharp, X \rangle = |dX^\flat|^2+|\delta_\mu X^\flat|^2 = |dX^\flat|^2 + |\mathrm{div}_\mu(X)|^2\ \mathrm{in} \ \mathcal{D}'(M).
\end{align*}
Note that if locally $X = X^{i}\partial_{x^{i}}$ then $X^\flat = g_{ij}X^{i}dx^{j}$ and hence
\begin{align*}
    dX^\flat = d(g_{ij}X^{i}dx^{j}) = \partial_k(g_{ij}X^{i})dx^k\wedge dx^j = \sum_{j < k} (\partial_k(g_{ij}X^{i})-\partial_j(g_{ik}X^{i}))dx^k\wedge dx^j. 
\end{align*}
It follows that 
\begin{align*}
    | dX^\flat|^2 = \sum_{j < k}\sum_{l < m} g^{jl}g^{km}(\partial_k(g_{ij}X^{i})-\partial_j(g_{ik}X^{i}))(\partial_m(g_{pl}X^{p})-\partial_l(g_{pm}X^{p})). 
\end{align*}
We also know that 
\begin{align*}
    |\mathrm{div}_\mu(X)|^2 = \frac{1}{h^4}\Big(h^2\partial_iX^{i} + 2h\partial_i h X^{i} + \frac{1}{2}h^2 X^{i}\mathrm{tr}(g^{-1}\partial_ig)\Big)^2.
\end{align*}
Recalling Lemma \ref{christoffel_l2_implies_dg_l2} and arguing by density, we get that the following formula holds for all $g \in C^{0}$ with $L^2_{\rm{loc}}$ Christoffel symbols, all $\phi \in C^{0}_c\cap W^{1,2}_{\rm{loc}}(M)$ and all positive $h \in C^{0}\cap W^{1,2}_{\rm{loc}}(M)$:
\begin{align}\label{districcilocalweighted}
    \int&  X^jX^k (\Gamma^s_{kj}\Gamma^p_{ps}-\Gamma^s_{kp}\Gamma^p_{js}) \phi h^2\sqrt{|g|}\, dx^1\ldots dx^n - \int \Gamma^p_{jk}\partial_p(X^jX^k \phi h^2 \sqrt{|g|})\, dx^1\ldots dx^n \nonumber  \\
    &+\int \Gamma^p_{pk}\partial_j(X^jX^k \phi h^2 \sqrt{|g|})\, dx^1\ldots dx^n \nonumber \\
    &+ 2\int X^jX^k(\partial_jh\partial_kh+h\partial_sh\Gamma^s_{kj})\phi\sqrt{|g|}dx^1\ldots dx^n + 2\int \partial_jh\partial_k(X^jX^kh\phi\sqrt{|g|})dx^1\ldots dx^n \nonumber \\
     =& -\frac{1}{2}\int \langle \nabla|X|^2, \nabla \phi \rangle_g h^2\sqrt{|g|}\, dx^1\ldots dx^n + \int \Big(\partial_iX_i + X_i\frac{1}{2}\tr(g^{-1}\partial_ig ) + \frac{2}{h}\partial_ih X^{i} \Big)^2h^2\phi \sqrt{|g|}\, dx^1\ldots dx^n \nonumber \\
     &+ \int \sum_{j < k}\sum_{l < m} g^{jl}g^{km}(\partial_k(g_{ij}X^{i})-\partial_j(g_{ik}X^{i}))(\partial_m(g_{pl}X^{p})-\partial_l(g_{pm}X^{p})) \phi h^2\sqrt{|g|}\, dx^1\ldots dx^n \nonumber \\
     & - \int  \Big(\sum_{i,j,r,s}(\partial_i X^s + \Gamma^r_{ip}X^p)(\partial_j X^r + \Gamma^s_{jq}X^q) g^{ij}g_{sr} \Big) \phi h^2\sqrt{|g|}\, dx^1\ldots dx^n.
\end{align}

\section{{First order calculus}}\label{Sec:Four}

In this section we recall the notion of Cheeger energy and the associated Sobolev space $W^{1,2}_{w}$, which provide a well settled first order calculus for general metric measure spaces. On a smooth manifold endowed with a continuous metric and a continuous weight on the measure, one can define the first order quadratic Sobolev space $H^2_1$ also using local charts and distributional weak derivatives. After recalling the two approaches, we will show that they actually coincide, in the latter framework. The results will be useful in the establishing the equivalence of distributional and synthetic Ricci curvature lower bounds.

\subsection{The Cheeger energy}
In this subsection we recall a generalised notion of modulus of gradients and Sobolev functions in metric measure spaces introduced by Cheeger in \cite{cheeger1999differentiability} and further analysed by Ambrosio, Gigli and Savaré  \cite{ambrosio2014inventio}. 

Throughout the section, $(X, \sfd)$ is a complete and separable metric space. The slope (or local Lipschitz constant) of a real valued function $F: X \to \R$ is defined by
\begin{align*}
    |DF|(x) := \limsup_{y \to x}\frac{|F(x)-F(y)|}{\sfd(x, y)},
\end{align*}
if $\{x\}$ is not isolated, and $0$ otherwise.

We endow $(X,\sfd)$ with a non-negative $\sigma$-finite Borel  measure, obtaining the metric measure space $(X, \sfd, \mathfrak{m})$.
Throughout the rest of this work, we assume that there exists a bounded Borel Lipschitz map $V: X \to [0, \infty)$ such that  (cf. \cite[Ch.\,4]{ambrosio2014inventio})
\begin{equation}\label{condition_on_measure}
\begin{split}
    V\ \mathrm{is\ bounded\ on\ each\ compact\ set\ } K \subset X \ \mathrm{and}  \\
    \int_X e^{-V^2} \, \di\mathfrak{m} \leq 1. 
    \end{split}
\end{equation}
The metric speed of a curve $\gamma:[0,1] \to X$ is given by  $|\dot{\gamma}_t|:= \lim_{h \to 0}\frac{\sfd(\gamma_{t+h}, \gamma_t)}{|h|}$ {which exists for almost every $t \in [0,1]$}. The curve $\gamma$ is said to be absolutely continuous if  $t \to |\dot{\gamma}_t| \in L^1((0, 1))$. We next recall the notion of test plan and weak upper gradient. We will use the conventions of  \cite{gigli2018nonsmooth}, after \cite{ambrosio2014inventio}.
\begin{definition}
    Let $\boldsymbol{\pi} \in \mathcal{P}(C([0, 1], X))$. We say that $\boldsymbol{\pi}$ is a \emph{test plan} if there exists a constant $C(\boldsymbol{\pi})>0$ such that 
    \begin{align*}
        (e_t)_\#\boldsymbol{\pi} \leq C(\boldsymbol{\pi})\mathfrak{m}\ \mathrm{for \ all \ } t \in [0,1]
    \end{align*}
    and 
    \begin{align*}
        \int \int_0^1 |\dot{\gamma}_t|^2\,\di t\di\boldsymbol{\pi}(\gamma) < \infty. 
    \end{align*}
    We use the convention that if $\gamma$ is not absolutely continuous, then $\int_0^1 |\dot{\gamma}_t|^2\,\di t =\infty$.
\end{definition} 
\smallskip

\begin{definition}
   Given { $f: X \to \R$ } a $\mathfrak{m}$-measurable function, then a $\mathfrak{m}$-measurable function $G: X \to [0, \infty]$ is called a weak upper gradient of $f$, if 
   \begin{align}\label{weakgradient}
       \int|f(\gamma_1)-f(\gamma_0)|\di\boldsymbol{\pi}(\gamma)\,  \leq \int\int_0^1G(\gamma_t)|\dot{\gamma}_t|\, \di t\di\boldsymbol{\pi}(\gamma) < \infty, \quad \mathrm{for \ all \ test \ plans } \ \boldsymbol{\pi}. 
   \end{align}
   We say that $f$ is in the Sobolev class $S^2(X, \sfd,\mathfrak{m})$, if there exists $G \in L^2(\mathfrak{m})$ such that \eqref{weakgradient} holds.
\end{definition}
The discussion in \cite[Prop. 5.9 and Def. 5.11]{ambrosio2014inventio}, shows the existence of a weak upper gradient $|Df|_w$ such that $|Df|_w \leq G$ {$\mathfrak{m}$-a.e.} for all other weak upper gradients $G$. We will call it the \textit{minimal weak upper gradient} of $f$.
\begin{definition}
    The Sobolev space $W^{1,2}_w(X, \sfd, \mathfrak{m})$ is defined as $L^2 \cap S^2(X)$ and becomes a Banach space with the norm
    \begin{align*}
        \norm{f}_{W^{1,2}_w(X)}^2 = \norm{f}_{L^2}^2 + \norm{|Df|_w}_{L^2}^2. 
    \end{align*}
\end{definition}
{ For simplicity, we will from now on write $W^{1,2}_w(X)$ instead of $W^{1,2}_w(X, \sfd, \mathfrak{m})$.}
\begin{definition}
    The Cheeger energy is defined in the class of $\mathfrak{m}$-measurable functions by
    \begin{align*}
        \Ch(f):= \left\{ \begin{array}{ll}
           \frac{1}{2}\int_X |Df|^2_w\, \di\mathfrak{m}  & \mathrm{if}\ f\ \mathrm{has\ a\ weak\ upper\ gradient\ in }\ L^2(X, \mathfrak{m}),  \\
           \infty  & \mathrm{otherwise},
        \end{array}\right.
    \end{align*}
    with proper domain $D(\Ch)= \{f:X \to [0, \infty], \mathfrak{m}\mathrm{-measurable}, \ \Ch(f) < \infty\} = S^2(X, \sfd, \mathfrak{m})$.
\end{definition}
A metric measure space is { called} \emph{infinitesimally Hilbertian} if  $\Ch$ is a quadratic form or, equivalently, if $W^{1,2}_w(X)$ is a Hilbert space \cite{ambrosio2014duke, Gigli15}.
 In this case, we can define the associated Dirichlet form $\mathcal{E}: D(\Ch)^2 \to \R$, via $\mathcal{E}(f, f)= 2\Ch(f)$ \cite[Ch.\;4.3]{ambrosio2014duke}.
\begin{proposition}[\cite{ambrosio2014duke}, Def. 4.12 and Prop. 4.14]\thlabel{carreduchamp}
    For any $f, g \in D(\Ch)$, it holds 
    \begin{align*}
        \mathcal{E}(f, g)= \int_X G(f, g)\,\di\mathfrak{m},
    \end{align*}
    where
    \begin{align*}
        G(f, g)= \lim_{\e \searrow 0} \frac{|D(f+\e g)|_w^2-|Df|_w^2}{2\e} \in L^1(X, \mathfrak{m}).
    \end{align*}
\end{proposition} 
By the theory of Dirichlet forms (cf. \cite{bouleau2010dirichlet}), there exists an associated Laplace operator $\Delta: D(\Delta) \to L^2(X, \mathfrak{m})$, given by
\begin{align*}
    -\int_X \Delta f h \, \di\mathfrak{m} = \int_X G(f, h) \, \di\mathfrak{m}. 
\end{align*}
Here $D(\Delta) = \{f \in L^2(\mathfrak{m}): \Delta f\ \mathrm{exists},\ \Delta f \in L^2\}$ is dense in $L^2$. Moreover, it induces a linear semi-group $(H_t)_{t \geq 0}$, $H_t: L^2(X) \to L^2(X)$ such that $H_{t+s} = H_sH_t$, for all $f \in L^2$, we have that $\lim_{t \to 0} H_t f = f$, and for all $t > 0$, we have that  $H_tf \in D(\Delta)$ and $\frac{\di}{\di t} H_tf = \Delta H_tf$.
$(H_t)_{t \geq 0}$ is called \emph{heat flow} semi-group. Write $f_t$ for $H_t(f_0)$. By \cite[(4.26)]{ambrosio2014inventio}, we have that 
\begin{align*}
    \Ch(f_t) \leq \inf\left\{ \Ch(g) + \frac{1}{2t}\int_X|g-f_0|^2\,\di\mathfrak{m}:\ g \in D(\Ch) \right\}.
\end{align*}
Another useful fact is the maximum principle for the heat flow: \cite[Prop. 2.14]{ambrosio2014duke}: if $f\in L^2$, $f \leq C$, ($f \geq C$) $\mathfrak{m}$-a.e. then $H_tf \leq C$ ($H_tf \geq C$) $\mathfrak{m}$-a.e. in $X$.

\subsection{Optimal transport and curvature dimension conditions}
In this subsection, we recall the notion of an $\rcd(K, \infty)$-space, starting by some elements of optimal transport theory.  Throughout the section, $(X,\sfd)$ is a complete and separable metric space.
 Denote by $\mathcal{P}(X)$ the set of Borel probability measures on $X$ and  $\mathcal{P}_2(X)$ the  space of probability measures with finite second moment:
 \begin{align*}
    \mathcal{P}_2(X):= \left\{ \mu \in \mathcal{P}(X): \ \int d^2(x, x_0)\, \di\mu(x) < \infty\ \mathrm{for \ some\ } x_0 \in X\right\},
\end{align*}
{Let $(Y, \sfd_Y)$ be another complete and separable metric space.} 
Given a Borel map $T:X \to Y$, $\mu \in \mathcal{P}(X)$, the measure $T_\#\mu$, called the push-forward of $\mu$ by $T$, is defined by 
\begin{align*}
    T_\#\mu(E) = \mu(T^{-1}(E)).
\end{align*}
For $\mu \in \mathcal{P}(X)$ and $\nu \in \mathcal{P}(Y)$, we define the set $\Gamma(\mu, \nu)$ as the set of all transport plans $\gamma \in \mathcal{P}(X \times Y)$, i.e. the set of Borel probability measures on $X \times Y$ such that $\pi^X_\#\gamma= \mu$ and $\pi^Y_\#\gamma= \nu$. 
Here $\pi^X$ and $\pi^Y$ denote the natural projections from $X \times Y$ onto $X$ and $Y$ respectively.

For $\mu, \nu \in \mathcal{P}_2(X)$, the quadratic Wasserstein distance $W_2$ is defined as
    \begin{align*}
        W_2(\mu, \nu):= \sqrt{\inf_{\gamma \in \Gamma(\mu, \nu)}\int d^2(x, y)\,\di\gamma(x, y)}.
    \end{align*}
The Wasserstein distance $W_2$ is a distance on $\mathcal{P}_2(X)$ and turns it into a complete { and separable} metric space. Moreover, if $(X, \sfd)$ is a geodesic space, then so is $(\mathcal{P}_2(X), W_2)$. \\
Given a metric measure space $(X, \sfd, \mathfrak{m})$, we denote by $\mathcal{P}^a_2(X) \subset \mathcal{P}_2(X)$ the set of measures with finite second moment that are absolutely continuous with respect to $\mathfrak{m}$. \\
In order to introduce curvature dimension conditions, we will need the following distortion coefficients. For  $\kappa \in \R$ and $\theta \geq 0$,  let
\begin{align*}
    \mathfrak{s}_\kappa(\theta):= \left\{ \begin{array}{ll}
           \frac{1}{\sqrt{\kappa}}\sin(\sqrt{\kappa} \theta) &\mathrm{if}\  \kappa > 0,  \\
           0 &\mathrm{if}\ \kappa = 0, \\
           \frac{1}{\sqrt{-\kappa}}\sinh(\sqrt{-\kappa} \theta) &\mathrm{if}\ \kappa < 0.
        \end{array}\right.
\end{align*}
Moreover, for $t \in [0,1]$, set
\begin{align*}
    \sigma^{(t)}_\kappa(\theta) := \left\{ \begin{array}{ll}
           \frac{\mathfrak{s}_\kappa(t\theta)}{\mathfrak{s}_\kappa(\theta)} &\mathrm{if}\ \kappa\theta^2 \neq 0\ \mathrm{and}\ \kappa\theta^2 < \pi^2,  \\
           t  &\mathrm{if}\ \kappa\theta^2 = 0, \\
           +\infty &\mathrm{if}\ \kappa\theta^2 \geq \pi^2 .
           \end{array}\right.
\end{align*}
Define the convex and continuous function $ U_\infty$ on $[0, \infty)$ as 
\begin{align*}
    U_\infty(z):= z \log(z). 
\end{align*}
The Boltzmann-Shannon entropy functional $\mathcal{E}_\infty: \mathcal{P}(X)\to \R \cup \{+\infty\}$ is defined by
\begin{align*}
    \mathcal{E}_\infty(\mu) := \left\{ \begin{array}{ll}
           \int_X U_\infty(\rho)\,\di\mathfrak{m} , &\mathrm{if\ } \mu \ll  \mathfrak{m}, \; \mathrm{with}\; \mu = \rho \mathfrak{m}  {\; \mathrm{and} \; U_\infty(\rho)\in L^1(X,  \mathfrak{m})}   \\
           \infty, \quad \mathrm{else}.
           \end{array}\right.
\end{align*}

Next we recall the synthetic notion of Ricci curvature bounded below for a metric measure space, pioneered by Sturm \cite{sturm2006geometryI, sturm2006geometryII} and Lott-Villani \cite{lott2009ricci}. For the finite dimensional case, we will use the (reduced) curvature dimension condition of Bacher-Sturm \cite{bacher2010localization}.
\begin{definition}
    We say that a metric measure space $(X, \sfd, \mathfrak{m})$ has Ricci curvature bounded from below by $K \in \R$ provided the functional $\mathcal{E}_\infty: \mathcal{P}(X)\to \R \cup \{+\infty\}$ is (weakly) $K$-geodesically convex on $(\mathcal{P}^a_2(X), W_2)$, that means that for any $\mu_0, \mu_1 \in \mathcal{P}^a_2(X)$, there exists a Wasserstein geodesic $(\mu_t)_{t \in [0, 1]}$ such that for all $t \in [0,1]$,
    \begin{align*}
        \mathcal{E}_\infty(\mu_t) \leq (1-t)\mathcal{E}_\infty(\mu_0) + t\mathcal{E}_\infty(\mu_1) - \frac{K}{2} t(1-t)W_2^2(\mu_0, \mu_1).
    \end{align*}
    In this case we say that $(X, \sfd, \mathfrak{m})$ is a $\mathsf{CD}(K, \infty)$-space.
\end{definition}
\smallskip

\begin{definition}
    Given two numbers $K \in \R$ and $N \in [1, \infty)$, we say that a metric measure space $(X,\sfd,  \mathfrak{m})$ satisfies the (reduced) curvature-dimension condition $\mathsf{CD}^{*}(K, N)$ if for each pair $\mu_0, \mu_1 \in \mathcal{P}^a_2(X)$ with bounded supports there exists an optimal transport coupling $\boldsymbol{\gamma}$ and a Wasserstein geodesic $(\mu_t = \rho_t \mathfrak{m})_{t \in [0,1]} \subset \mathcal{P}^a_2(X)$ connecting $\mu_0 = \rho_0 \mathfrak{m}$ and $\mu_1=\rho_1 \mathfrak{m}$ such that for all $t \in [0,1]$, $N' \geq N$, it holds 
    \begin{align*}
        \int_{X} \rho_t^{-\frac{1}{N'}} \, \di\mu_t \geq \int_{X \times X} \, \Big[ \sigma^{(1-t)}_{K/N'}(\sfd(x_0, x_1))\rho_0(x_0)^{-\frac{1}{N'}} +\sigma^{(t)}_{K/N'}(\sfd(x_0, x_1))\rho_1(x_1)^{-\frac{1}{N'}} \Big] \di\boldsymbol{\gamma}(x_0, x_1).
    \end{align*}
    \end{definition}
  In order to single out the ``Riemannian'' like structures out of the ``possibly Finslerian'' $\mathsf{CD}$ spaces, Ambrosio-Gigli-Savar\'e \cite{ambrosio2014duke} (see also \cite{AGMR}) in the $N=\infty$ case, and Gigli  \cite{Gigli15} in the $N<\infty$ case  introduced the Riemannian curvature-dimension condition. 
 
\begin{definition}
      We say that $(X, \sfd, \mathfrak{m})$ satisfies the Riemannian curvature-dimension condition $\rcd(K, \infty)$ if it is infinitesimally Hilbertian and satisfies the $\mathsf{CD}(K, \infty)$-condition. \\
    For $N \in [1, \infty)$, we say that $(X, \sfd, \mathfrak{m})$ satisfies the Riemannian curvature-dimension condition $\rcd^{*}(K, N)$ if it is infinitesimally Hilbertian and satisfies the $\mathsf{CD}^{*}(K, N)$-condition. 
\end{definition}

{
\begin{remark}\label{rem:RCD*=RCD}
    In the case of weighted Riemannian manifolds with continuous metric and continuous weight, the resulting measure will always be $\sigma$-finite in which case the $\rcd^*$ and $\rcd$ condition coincide, see \cite{CavallettiMilman21, Li-AMPA}.
\end{remark}
}

Let us recall the following useful results:
\begin{lemma}[\cite{erbar2015equivalence}, Lem.\;3.2 and Thm.\;3.17]\label{implications_of_cd_conditions}
    If $(X, \sfd, \mathfrak{m})$ satisfies the $\mathrm{RCD}^{*}(K, N)$-condition then it also satisfies the $\mathrm{RCD}^{*}(K', N')$-condition for any $K' \leq K$ and $N'\geq N$. Moreover, it satisfies the $\mathrm{RCD}(K, \infty)$-condition. 
\end{lemma}

\begin{theorem}[\cite{ambrosio2014duke}, Thm. 6.2]\thlabel{linftyboundonderivative}
   Let $(X, \sfd,\mathfrak{m})$ be a $\rcd(K, \infty)$ space and $f \in W^{1,2}_w(X)$. Then
   \begin{align*}
       |D(H_tf)|_w^2 \leq e^{-2Kt}(H_t|Df|_w^2).
   \end{align*}
\end{theorem}

\begin{remark}\label{rem:CDVol}
By  \cite[Thm.\;4.24]{sturm2006geometryI}, the condition \eqref{condition_on_measure} is automatically satisfied in every $\cd(K, \infty)$ space, by choosing the function  $V:X \to \R$, $V(x) :=C \sfd(x, x_0)$ for a suitable $C >0$ and any  base point $x_0 \in X$.
\end{remark}

\subsection{\texorpdfstring{$L^p$}{TEXT} and Sobolev spaces on weighted Riemannian manifolds of low regularity}

In this subsection, we recall the language of Sobolev spaces on manifolds. We will summarise facts from \cite[Ch.\;2]{hebey1996sobolev}, adapting to the case of weighted manifolds of low regularity. 
Let $(M,g)$ be a Riemannian manifold (with $g$ a smooth metric for now). Let $h:M \to (0, \infty)$ be a $C^{0}$-function and define the measure $\mu$ via $\di\mu := h^2\di\vol_g$ . 
For an integer $k \geq 0$ and $u \in C^\infty(M)$ define $\nabla^k_cu$ to be the $k$-th covariant derivative ($\nabla_c^0u:= u$). Notice that $\nabla^k_cu$ is a $(0, k)$-tensor. We will only be interested in the cases $k \leq 2$, so we write down the explicit formulas for the covariant derivatives in local coordinates:
\begin{align*}
    (\nabla^1_c u)_i &= (du)_i = \partial_i u \quad \mathrm{and}  \quad  (\nabla^2_c u)_{ij} = \partial_{ij}u - \Gamma^k_{ij}\partial_ku.
\end{align*}
By definition, we have that 
\begin{align*}
    |\nabla^k_cu|^2= g^{i_1j_1}\ldots g^{i_kj_k} (\nabla^k_cu)_{i_1\ldots i_k}(\nabla^k_cu)_{j_1\ldots j_k},
\end{align*}
where $i_1 < \ldots < i_k$ and $j_1< \ldots < j_k$.
Let $p \in [1, \infty)$. Denote by $C^p_k(M, \mu)$ the space of smooth functions $u \in C^\infty(M)$ such that  $|\nabla^j_cu| \in L^p(M, \mu)$ for any $j=0, \ldots k$.
\begin{definition}
    The Sobolev space $H^p_k(M, \mu)$ is the completion of $C^p_k(M, \mu)$ with respect to the norm
    \begin{align*}
        \norm{u}_{H^p_k(M, \mu)}= \sum_{j=0}^k\left(\int_M |\nabla^j_cu|^p \, \di\mu\right)^{\frac{1}{p}}.
    \end{align*}
\end{definition}
\begin{proposition}
    $H^2_k(M, \mu)$ is a Hilbert space with the inner product
    \begin{align*}
        \langle u, v\rangle = \sum_{m=0}^k \int_M\left( g^{i_1j_1}\ldots g^{i_mj_m} (\nabla^m_cu)_{i_1\ldots i_m}(\nabla^m_cv)_{j_1\ldots j_m}\right)h^2\,\di\vol_g.
    \end{align*}
\end{proposition}
Note that the inner product induces an equivalent norm on $H^2_k(M, \mu)$. In the case $p=2$, we will work with the norm induced by the inner product. 
We can generalise these definitions to the following:
\begin{definition}
    Let $s, r \geq 0$ and $p \geq 1$. Define $C^p_k(T^r_s(M), \mu)$ the space of smooth sections $u \in \mathcal{T}^r_s(M)$ such that $|\nabla_c^ju|_g \in L^p(M, \mu)$ for $j=0, \ldots, k$. The Sobolev space $H^p_k(T^r_s(M), \mu)$ is defined to be the closure of  $C^p_k(T^r_s(M), \mu)$ under the norm
    \begin{align*}
        \norm{u}_{H^p_k(T^r_s(M), \mu)} := \sum_{j=0}^k \left(\int_M |\nabla^j_cu|^p\, \di\mu\right)^{\frac{1}{p}}.
    \end{align*}
    Moreover, for all such $s,r$, we define  the space $\mathring{H}^p_k(T^r_s(M), \mu)$ as the closure of the set of smooth compactly supported $(r,s)$-tensors under the $H^p_k(T^r_s(M), \mu)$-norm.
\end{definition}
We notice that these definitions make sense for $k \leq 2$ {and $p=2$} if $g \in C^{0}(M)$ with $L^2_{\rm{loc}}$-Christoffel symbols.
The following criterion for weak convergence will be useful later.
We { will repeatedly make use of} the following elementary result.
\begin{lemma}\label{ccinfty_dense_in_lp}
   Let $M$ be a smooth manifold, $g$ be a continuous Riemannian metric and $h$ a continuous positive function. Moreover let $p \in [1, \infty)$. Define the measure $\mu$ by $\mu := h^2\di\vol_g$. Then $C_c^\infty(M)$ is dense in $L^p(M, \mu)$.
\end{lemma}
\begin{lemma}\thlabel{weakconv}
   Let $(M, g, \mu)$ be a weighted Riemannian manifold with a $C^{0}$-Riemannian metric $g$ with $L^2_{\rm{loc}}$-Christoffel symbols and endowed with a measure $\mu$ defined via $\di\mu = h^2\di\vol_g$ for a function $h \in C^{0}(M, (0, \infty))$. Let $(V_\alpha, \psi_\alpha)_\alpha$ be an atlas for $M$ such that $\overline{V_\alpha}$ is compact, for all $\alpha$. 
 \\  Let $f_i$ be a sequence in $L^2(M, \mu)$ and  let $f \in L^2(M, \mu)$.
    \\Then $f_i  \rightharpoonup f$  in $L^2(M, \mu)$ if and only if $(\psi_\alpha)_*f_i \rightharpoonup (\psi_\alpha)_*f$ in $L^2(\psi_\alpha(V_\alpha), (\psi_\alpha)_\#\mu)$, for all $\alpha$. \\
    If $f, f_i \in H^2_k(M, \mu)$ {($k \leq 1$)} and $\partial^\beta f_i \rightharpoonup \partial^\beta f$ in $L^2(\psi_\alpha(V_\alpha), (\psi_\alpha)_\#\mu)$ for all $\alpha$ and all $\beta \in \N^n$, $|\beta| \leq k$,  then $f_i \rightharpoonup f$ in $H^2_k(M, \mu)$. 
\end{lemma}
\begin{proof}
    For the first part, note that it suffices to prove convergence of the inner product on a dense subset of $L^2(M, \mu)$. Using Lemma \ref{ccinfty_dense_in_lp}, we can pick a function $\varphi \in C_c(M)$. Let $(V_j)_{j=1}^l$ be a finite cover of $\supp\, \varphi$ and pick a partition of unity $(\rho_j)$ subordinate to that cover. Now
    \begin{align*}
        \int_{M} \varphi f_i h^2\, \di\vol_g = \int_{M} \sum_{j=1}^l \rho_j \varphi f_ih^2\, \di\vol_g = \sum_{j=1}^l \int_{\psi_j(V_j)} (\psi_j)_* (\rho_j \varphi f_i) h^2\sqrt{|g|} \, \di\mathcal{L}^n.
    \end{align*}
    If $(\psi_j)_*f_i \rightharpoonup (\psi_j)_*f$, it is easy to see that the above term converges. For the other direction it suffices to note that every function $\varphi \in L^2(\psi_j(V_j), (\psi_j)_*\mu)$ for some $j$ can be extended to a function $\Phi \in L^2(M, \mu)$ via pullback and extension by $0$. \\
    For the second part, we simply use the first part and the definition of the $H^2_k(M, \mu)$-inner product.  
\end{proof}
Now, we can generalise the validity of the identity \eqref{districcilocalweighted} to Sobolev vector fields:
\begin{lemma}[Distributional Ricci curvature for $H^2_1(TM, \mu)$-fields]\thlabel{districcisobolev}
    Let $g$ be a $C^{0}$-Riemannian metric whose Christoffel symbols are in $L^2_{\rm{loc}}$ and $\mu$ be the measure defined by $\di\mu = h^2\di\vol_g$ for a positive $h \in C^{0}(M, (0, \infty))\cap W^{1,2}_{\rm{loc}}(M)$. Let {$X \in H^2_1\cap L^\infty_{loc}(TM, \mu)$}. Then \eqref{districcilocalweighted} holds. 
\end{lemma}
\begin{proof}
    This follows from the density of $C^2_1(TM, \mu)$ in $H^2_1(TM, \mu)$ and the fact that $\phi, |\nabla \phi|_g \in L^1 \cap L^\infty(M)$ have compact support.
\end{proof}
\subsection{Cheeger energy on a manifold with continuous Riemannian metric and weight}
In this subsection, we examine the Cheeger energy in the setting of a smooth manifold $M$ endowed with a continuous Riemannian metric $g$ and a continuous weight $h^2$ on the volume form. The {distance} induced by $g$ is given by 
\begin{align}\label{def:dg}
    \sfd_g(x, y) = \inf\Big\{ \int_{0}^1 |\dot{\gamma_t}|_g\, \di t: \, \gamma\ \mathrm{piecewise \ }C^\infty, \gamma_0=x, \gamma_1 = y \Big\}.
\end{align}
We also consider the Riemannian volume form given locally by $\sqrt{|g|}dx^1\wedge \ldots \wedge dx^n$ and the associated volume measure $\di\vol_g= \sqrt{|g|}\di\mathcal{L}^n$. We will consider the metric measure space $(M, \sfd_g, \mu)$, where $\di\mu = h^2\di\vol_g$ for a positive and continuous $h$. It is easily checked that  $\mu$ is a $\sigma$-finite Borel measure. We assume that \eqref{condition_on_measure} holds, as we will work with $\cd(K, \infty)$-manifolds (cf.\;Rem.\,\ref{rem:CDVol}).

 To begin with, we will collect some properties of the metric space arising from a manifold with a continuous Riemannian metric $g$. 
The even more general setting of Lipschitz manifolds, with emphasis on potential theory for uniformly elliptic operators,  has previously been studied by Norris \cite{norris1997heat}, Saloff-Coste  \cite{saloff1992uniformly}, Sturm  \cite{sturm1995analysis}, De Cecco and Palmieri  \cite{de1991integral}. 
In that case, the term Lipschitz manifold refers to a topological manifold $M$ with Lipschitz charts, endowed with a Riemannian metric tensor $g$, such that the coefficients of $g$ and $g^{-1}$  are in $L^\infty_{\rm{loc}}$. This shall not be confused with metric tensors whose coefficients are locally Lipschitz continuous. 
More in the spirit of the present section is the work by Burtscher \cite{burtscher2012length}. For the reader's convenience,  in the following we will focus on our original results, while trying to be as self-contained as possible.
\begin{proposition}[cf.\;\cite{petersen2006riemannian} Thm.\;5.3.8,  and \cite{burtscher2012length} Thm.\;4.5]\thlabel{lipschitz}
    Let $f$ be a Lipschitz function on $(M, \sfd_g)$ and $(V_\alpha, \psi_\alpha)_\alpha$ be an atlas for $M$. Then for each $\alpha$, $f\circ (\psi_\alpha)^{-1}$ is locally Lipschitz.
\end{proposition}
\begin{proof}
   We show that $|\cdot|_{euc}$ and $\sfd_g$ are locally equivalent. Pick a point $p \in M$, $\alpha$ such that $p \in V_\alpha$ and fix an $\e > 0$ such that $\overline{B^{euc}_{\e}(\psi_\alpha(p))} \subset \psi_\alpha(V_\alpha)$. From now on, we will drop the index $\alpha$. 
   Take two arbitrary $x, y$ such that $\psi(x), \psi(y) \in B^{euc}_{\e/2}(\psi(p)) $ and let $c:[0, 1] \to M$ be a piecewise $C^\infty$-curve from $x$ to $y$.  We note that we can find a $\lambda \geq 1$ such that $\frac{1}{\lambda} \norm{v}_{euc} \leq g(v, v) \leq \lambda \norm{v}_{euc}$ for all $q \in \psi^{-1}(B^{euc}_{\e}(\psi(p)))$ and all $v \in T_qM$.
   We now distinguish three cases:
   \begin{itemize}
       \item[1.] $c$ is a straight line in the Euclidean metric. Then 
       \begin{align*}
           |x-y|_{euc} &= L_{euc}(c)= \int_{0}^1 |\dot{c}|_{euc}\, \geq \frac{1}{\lambda} \int_{0}^1 |\dot{c}|_{g}\, \di t \geq \frac{1}{\lambda}\sfd_g(x,y).
       \end{align*}
       \item[2.] If $c$ is a general curve that lies in $\psi^{-1}(B^{euc}_{\e}(\psi(p)))$ then
       \begin{align*}
           L_g(c) &= \int_0^1 |\dot{c}|_g\, \di t \geq \lambda \int_0^1 |\dot{c}|_{euc}\, \di t \geq \lambda|x-y|_{euc}.
       \end{align*}
       \item[3.]If $c$ leaves $\psi^{-1}(B^{euc}_{\e}(\psi(p)))$ then there is a minimal $t_0$ such that $c(t_0) \notin \psi^{-1}(B^{euc}_{\e}(\psi(p)))$ and hence 
       \begin{align*}
            L_g(c) &\geq \int_0^{t_0} |\dot{c}|_g\, \di t \geq \lambda \int_0^{t_0}|\dot{c}|_{euc}\, \di t \geq \lambda \frac{\e}{2} \geq \frac{\lambda}{2} |x-y|_{euc}.
       \end{align*}
   \end{itemize}
   This yields that the two metrics are equivalent in $\psi^{-1}(B^{euc}_{\e/2}(\psi(p)))=:U_p$. Hence, for a $\sfd_g$-Lipschitz function $f$ and any $p \in M$, there exists a constant $L_{U_p}$ such that $f$ is locally $L_{U_p}$-Lipschitz with respect to the induced Euclidean metric via $\psi$ on $U_p$. 
\end{proof}
 { In \cite{burtscher2012length}, a mollification argument yields:}
\begin{proposition}[{ \cite{burtscher2012length} Prop.\;4.1 and Thm\;3.15}]\label{length_space}
    The metric \eqref{def:dg} turns $M$ into a length space.
\end{proposition} 
The Hopf-Rinow theorem directly implies:
\begin{corollary}
    If $M$ is compact, $(M, \sfd_g)$ is a geodesic space.
\end{corollary}
{ A standard fact on the slope of functions is that if $f \in C^1(M)$, then  $|Df|(x) = |\nabla f|_g(x)$ for every $x$. The next proposition is also classical, so the proof is omitted.}
\begin{proposition}\label{metric_speed_explicit}
    Let $\gamma:[0, 1]\to M$ be an absolutely continuous curve. Then the metric speed $|\dot{\gamma_t}|$ coincides a.e.\;with $|\dot{\gamma}_t|_g = \sqrt{\langle  \dot{\gamma}_t, \dot{\gamma}_t\rangle_g}$, where $\dot{\gamma}_t$ denotes the (a.e.\;existing) derivative. 
\end{proposition}
 
 We will show that if $f \in C^1$, then $|\nabla f|_g$ is the minimal weak upper gradient. 
{The next lemma will be used in the proof of \thref{weakgradientforc1}.}
\begin{lemma}\thlabel{Lebesguepointsadvanced} 
Let $M$ be a smooth manifold equipped with a $C^0$-Riemannian metric $g$ and a measure $\mu$ defined via $\di\mu = h^2\di\vol_g$, for some $h \in C^0(M, (0, \infty))$.
Let $k \in L^1_{\rm{loc}}(M, \mu)$ and fix some coordinate patch $(V_m, \psi_m)$. Let $x \in \psi_m(V_m)$  and $\theta > 0$, such that $B_{4\theta}^{euc}(x) \subset \psi_m(V_m)$.  
For $\delta \in (0, \theta)\cap \Q$, $\dot{\gamma} \in \Q^n\cap B_\theta^{euc}(0)$,  we define 
    \begin{align*}
        F_{\dot{\gamma}, \delta, x}: t \to \frac{1}{\mathcal{L}^n(B_\delta(x))}\int_{B_\delta(x)} |\dot{\gamma}|_{g(y + t\dot{\gamma})}k(y + t\dot{\gamma})\,\di\mathcal{L}^n(y), \quad \forall t, \;|t|< 1.
    \end{align*}
    Then for $\mathcal{L}^n$-a.e.\;$x \in \psi_m(V_m)$ and all {$\delta \in (0, \theta)\cap \Q$}, $\dot{\gamma} \in \Q^n\cap B_\theta^{euc}(0)$, we have that $t=0$ is a Lebesgue point of $ F_{\dot{\gamma}, \delta, x}$.
\end{lemma}
 \begin{proof}
     As $g$ and $h$ are continuous, we get that $(\psi_m)_*k \in L^1_{\rm{loc}}(\psi_m(V_m), \mathcal{L}^n)$.
     Fix $w \in \psi_m(V_m)$, and choose $\theta > 0$ accordingly. Moreover, fix $\delta$ and $\dot{\gamma}$. For any $z \in B^{euc}_\theta(w)$, the Lebesgue differentiation theorem yields that $\mathcal{L}^1$-almost every $t \in (-1, 1)$ is a Lebesgue point of $ F_{\dot{\gamma}, \delta, z}$. 
     In other words that means that for $\mathcal{L}^1$-almost every $y$ on the line $\{z+t\dot{\gamma}, |t| <1\}$, $t=0$ is a Lebesgue point of $ F_{\dot{\gamma}, \delta, y}$. The set of non-Lebesgue points on that line shall be denoted by $N_{\dot{\gamma}, \delta, z}$, and it holds that $\mathcal{L}^1(N_{\dot{\gamma}, \delta, z})=0$. Let $H_{\dot{\gamma}, w}$ be the $(n-1)$-dimensional hyperplane through $w$ that is orthogonal to $\dot{\gamma}$,  and intersected with $B^{euc}_{\theta}(w)$. Then, by Fubini's Theorem, it holds
     \begin{align*}
         \mathcal{L}^n\Big(\bigcup_{z \in H_{\dot{\gamma}, w}}N_{\dot{\gamma}, \delta, z}\Big) = \int_{ H_{\dot{\gamma}, w}} \mathcal{L}^1(N_{\dot{\gamma}, \delta, z}) \di\mathcal{L}^{n-1}(z) = 0. 
     \end{align*}
     Hence for $\mathcal{L}^n$-almost every $y \in B^{euc}_\theta(w)$, $t=0$ is a Lebesgue point of $ F_{\dot{\gamma}, \delta, y}$. Note that, by construction,   $\dot{\gamma}$ and $\delta$  vary in a countable set. This shows the claim in an open neighbourhood around $w$. As $w$ was arbitrary, the proof is complete. 
 \end{proof}
\begin{proposition}\thlabel{weakgradientforc1}
    Let $M$ be a smooth manifold, $g$ a $C^0$-Riemannian metric on $M$ and $\mu$ a measure on $M$ defined by $\di\mu = h^2 \di\vol_g$, where $h \in C^0(M, (0, \infty))$. 
    Let $f$ be a $C^1$-function. Then 
    \begin{align*}
        |Df|_w(x) = |\nabla f|_g(x), \quad {\text{for $\mu$-a.e. x.}}
    \end{align*}
\end{proposition}

\begin{proof}
    { First, we observe that $|\nabla f|_g$ is a weak upper gradient because for any absolutely continuous curve $\gamma$, we have
    \begin{align*}
    |f(\gamma_1)-f(\gamma_0)|= \left|\int_0^1\langle \nabla f, \dot{\gamma_t}\rangle_g\big|_{\gamma_t}\, \di t \right| \leq \int_0^1 |\nabla f|_g(\gamma_t)|\dot{\gamma_t}|_g\, \di t.
    \end{align*}
    To see that it is minimal,} we argue by contradiction. Suppose that $|\nabla f|_g$ is not minimal. Then there exists an $\e \in (0, \frac{1}{4})$ such that $\mu(\{|\nabla f|_g - |Df|_w > \e\}) > 0$.
    We can then find an open set $V$ such that $\overline{V}$ is compact, is contained in one coordinate patch and $\mu(\{|\nabla f|_g - |Df|_w > \e\}\cap V)>0$. We denote by $\psi: V \to \R^n$ the associated coordinate chart.
    By the continuity of $g$ and $h$, there exists a constant $C= C(g, h, V)$ such that for each Borel set $E \subset V$, $ \frac{1}{C}\mathcal{L}^n(\psi(E)) \leq \mu(E) \leq C \mathcal{L}^n(\psi(E))$.
    For almost all $x \in \psi(V)$ take $\theta, \dot{\gamma}, \delta$ as in \thref{Lebesguepointsadvanced}. It still holds $\mathcal{L}^n(\{|\nabla f|_g - |Df|_w > \e\}\cap V)>0$. 
    By the Lebesgue differentiation theorem and \thref{Lebesguepointsadvanced} with $k= |Df|_w$, we get that there exists an $x \in \{|\nabla f|_g - |Df|_w > \e\}\cap V$ such that $x$ is a Lebesgue point of $|Df|_w$, and for all $\dot{\gamma}, \delta$ as above, $t = 0$ is a Lebesgue point of $ F_{\dot{\gamma}, \delta, x}$. 
     Set
     \begin{align*}
         \eta = \frac{1}{101}\min\Bigg( 1, \frac{1}{\norm{\nabla f}_{L^\infty}}, \frac{1}{2n^3\norm{g}_{L^\infty}}\Bigg)^2.
     \end{align*}
     Choose $\dot{\gamma} \in \Q^n$ such that  and
     {
        \begin{equation}\label{choice_of_gamma_dot_for_cs}
        \frac{99}{100}\leq |\dot{\gamma}|_g(x)\leq \frac{101}{100} \quad \mathrm{and} \quad
            |\langle\nabla f, \dot{\gamma}\rangle_g(x)-|\nabla f|_g(x)|\dot{\gamma}|_g(x)| \leq \frac{1}{4}\eta \e.
        \end{equation}
     }
     Note that if $t=0$ is a Lebesgue-point of $F_{x, \dot{\gamma}, \delta}$ then it is a Lebesgue-point of $F_{x, \alpha\dot{\gamma}, \delta}$ for all $\alpha \in \Q$.
     
     { Now choose } $ q \in \Q\cap(0, \theta)$ such that for all $\delta \in (0, q)$, it holds 
     \begin{align}
         \frac{1}{\mathcal{L}^n(B_\delta(x))} \int_{B_\delta(x)} (|Df|_w(y)-|Df|_w(x)) &< \eta \e, \label{x_lebesgue_point}\\
         ||\nabla f|_g(y)-|\nabla f|_g(z)| &\leq \frac{1}{2}\eta \e, \ \forall y, z \in B_{3\delta}(x), \label{gradients_close_continuity} \\
         |g_{ij}(y)-g_{ij}(z)| &\leq \frac{1}{2}\eta \e, \ \forall 1\leq i, j\leq n, \   y, z \in B_{3\delta}(x). \label{metric_g_close_continuity}
     \end{align}
     { The first inequality can be achieved because $x$ is a Lebesgue point of $|Df|_w$ and the remaining inequalities follow from the continuity of $g$ and $\nabla f$.}
We can now compute that for $s \leq q$ and $y \in B_q(x)$,
\begin{align*}
    |f(y-s\dot{\gamma})-f(y+s\dot{\gamma})| = \int_{-s}^s\langle\nabla f, \dot{\gamma} \rangle_g(y+t\dot{\gamma})\,\di t.
\end{align*}
{ Using \eqref{choice_of_gamma_dot_for_cs}, \eqref{gradients_close_continuity}, and \eqref{metric_g_close_continuity}, we get that for $t \in [-s, s]$ it holds}
\begin{align*}
    |\langle\nabla f, \dot{\gamma} \rangle_g(y+t\dot{\gamma})-|\nabla f|_g|\dot{\gamma}|_g(x)| \leq& |\langle\nabla f(y+t\dot{\gamma}) ,\dot{\gamma}\rangle_{g(y+t\dot{\gamma})-g(x)}| +|\langle \nabla f(y+t\dot{\gamma})-\nabla f(x), \dot{\gamma} \rangle_{g(x)}| \\
    &+ |\langle\nabla f, \dot{\gamma}\rangle_g(x)-|\nabla f|_g(x)|\dot{\gamma}|_g(x)| \leq \frac{3\e}{100}.
\end{align*}
Hence, for all {$\delta \in (0,q),$} $y \in B_\delta(x),$
\begin{align}\label{differenceversusgradient}
    |(|f(y-s\dot{\gamma})-f(y+s\dot{\gamma})|)-(2s|\nabla f|_g(x)|\dot{\gamma}|_g(x))| \leq \frac{6s\e}{100}.
\end{align}
Moreover, {again by \eqref{choice_of_gamma_dot_for_cs}, \eqref{gradients_close_continuity}, and \eqref{metric_g_close_continuity},} we get that 
\begin{align}
    \Big|\frac{1}{|B_q(x)|}\int_{B_q(x)} |\nabla f|_g (y)\di\mathcal{L}^n(y) -|\nabla f|_g (x) \Big| &\leq \frac{1}{100} \e \label{lebesguepointgradienta}, \\
    \Big|\frac{1}{|B_q(x)|}\int_{B_q(x)} |\nabla f|_g|\dot{\gamma}|_g(y)\di\mathcal{L}^n(y) -|\nabla f|_g|\dot{\gamma}|_g (x) \Big| &\leq \frac{1}{100} \e, \label{lebesguepointgradientb}
\end{align}
and { \eqref{x_lebesgue_point} yields that}
\begin{align}\label{lebesguepointweakgradient}
     \Big|\frac{1}{|B_q(x)|}\int_{B_q(x)} |Df|_w(y)\di\mathcal{L}^n(y)  -|Df|_w(x) \Big| \leq \frac{1}{100} \e,
\end{align}
thus,
\begin{align}\label{positivegapgradients}
    \frac{1}{|B_q(x)|}\int_{B_q(x)} (|\nabla f|_g -|Df|_w)(y)\di\mathcal{L}^n(y) > \frac{9\e}{10}. 
\end{align}
Using that $t=0$ is a Lebesgue point of $F_{\dot{\gamma},q, x}$, and the fact that $g$ and $|\nabla f|_g$ are continuous, we can choose $s \in (0, q)$ small enough such that 
\begin{align}\label{complicatedlebesguepointgradient}
    \Bigg|\frac{1}{|B_q(x)|}&\int_{-s}^s\int_{B_q(x)}|\nabla f|_g(y +t\dot{\gamma})|\dot{\gamma}|_{g(y +t\dot{\gamma})}\,\di\mathcal{L}^n(y)\di t - 2s  \frac{1}{|B_q(x)|}\int_{B_q(x)} |\nabla f|_g (y)|\dot{\gamma}|_g(y)\di\mathcal{L}^n(y)\Bigg| \leq \frac{2s\e}{100}, 
\end{align}
and
\begin{align}\label{complicatedlebesguepointweakgradient}
    \Bigg|\frac{1}{|B_q(x)|}&\int_{-s}^s\int_{B_q(x)}|Df|_w(y +t\dot{\gamma})|\dot{\gamma}|_{g(y +t\dot{\gamma})}\,\di\mathcal{L}^n(y)\di t- 2s  \frac{1}{|B_q(x)|}\int_{B_q(x)} |Df|_w(y)|\dot{\gamma}|_g(y)\di\mathcal{L}^n(y)\Bigg| \leq \frac{2s\e}{100}.
\end{align}
Now, we have all the ingredients to construct a test plan that yields a contradiction. For $y \in B_q(x)$ define $\gamma_y: [0, 1]\to \psi_m(V_m)$, $t\mapsto y+2(t-\frac{1}{2}) s\dot{\gamma}$ and define a test plan $\boldsymbol{\pi}$ as
\begin{align*}
    \di\boldsymbol{\pi}(\gamma):= \left\{ \begin{array}{ll}
           \frac{1}{|B_{q}(x)|}\di\mathcal{L}^n(y) & \mathrm{if}\ \gamma = \gamma_{y}, \ \mathrm{for \ some} \ y \in  B_{q}(x), \\
           0\  & \mathrm{otherwise}.
        \end{array}\right.
\end{align*}
Now, we can directly compute that 
\begin{align}\label{testingdifferenceeasy}
    \int |f(\gamma_1)-f(\gamma_0)| \di\boldsymbol{\pi}(\gamma) = \frac{1}{|B_{q}(x)|}\int_{B_q(x)}|f(y-s\dot{\gamma})-f(y+s\dot{\gamma})|\, \di\mathcal{L}^n(y).
\end{align}
Moreover, 
\begin{align}
    &\iint_0^1 |\nabla f|_g(\gamma_t)|\dot{\gamma}_t|_{g(\gamma_t)}\, \di t\, \di\boldsymbol{\pi}(\gamma) \nonumber \\
    &\quad = \frac{1}{|B_q(x)|}\int_{B_q(x)}\int_0^1 |\nabla f|_g\Big(y +2\Big(t-\frac{1}{2}\Big)s\dot{\gamma}\Big)2s|\dot{\gamma}|_{g((\gamma_y)_t)}\,\di t \, \di\mathcal{L}^n(y) \nonumber  \\
    &\quad = \frac{1}{|B_q(x)|}\int_{B_q(x)}\int_{-s}^s |\nabla f|_g(y+t\dot{\gamma})|\dot{\gamma}|_{g((\gamma_y)_t)}\,\di t\, \di\mathcal{L}^n(y), \label{transformationtestplangradient}
\end{align}
and similarly
\begin{align*}
    &\iint_0^1 |Df|_w(\gamma_t)|\dot{\gamma}_t|_{g(\gamma_t)}\, \di t\, \di\boldsymbol{\pi}(\gamma)  \\
    &\quad = \frac{1}{|B_q(x)|}\int_{B_q(x)}\int_{-s}^s |Df|_w(y+t\dot{\gamma})|\dot{\gamma}|_{g((\gamma_y)_t)}\,\di t \, \di\mathcal{L}^n(y).
\end{align*}
We then have that
\begin{align}\label{differencetesting}
    &\iint_0^1 |\nabla f|_g(\gamma_t)|\dot{\gamma}_t|_{g(\gamma_t)}\, \di t\,\di\boldsymbol{\pi}(\gamma)-\iint_0^1 |Df|_w(\gamma_t)|\dot{\gamma}_t|_{g(\gamma_t)}\, \di t\,\di\boldsymbol{\pi}(\gamma) \nonumber \\
    &\quad = \frac{1}{|B_q(x)|}\int_{B_q(x)}\int_{-s}^s (|\nabla f|_g-|Df|_w)(y+t\dot{\gamma})|\dot{\gamma}|_{g(y+t\dot{\gamma})}\,\di t\di\mathcal{L}^n(y) \nonumber\\
    &\overset{\eqref{complicatedlebesguepointweakgradient}, \eqref{complicatedlebesguepointgradient}}{\geq} 2s\frac{1}{|B_q(x)|}\int_{B_q(x)} (|\nabla f|_g-|Df|_w)|\dot{\gamma}|_g(y) \, \di\mathcal{L}^n(y) - \frac{4s\e}{100} \nonumber\\
    &\quad \overset{\eqref{positivegapgradients}}{>} \frac{99}{100}\frac{9s\e}{10} - \frac{4s\e}{100} \geq \frac{3s\e}{2} > 0.
\end{align}
The combination of the above estimates gives: 
\begin{align*}
    \int |f(\gamma_1)-f(\gamma_0)| \di\boldsymbol{\pi}(\gamma)  &\overset{\eqref{testingdifferenceeasy},\eqref{differenceversusgradient}}{\geq} 2s|\nabla f|_g(x)|\dot{\gamma}|_g(x) - \frac{6s\e}{100} \\
    &\overset{\eqref{lebesguepointgradientb}}{\geq} \frac{2s}{|B_q(x)|}\int_{B_q(x)} |\nabla f|_g|\dot{\gamma}|_g(y) \di\mathcal{L}^n(y) - \frac{s\e}{10}\\
    &\overset{\eqref{complicatedlebesguepointgradient}}{\geq} \frac{1}{|B_q(x)|}\int_{-s}^s\int_{B_q(x)}|\nabla f|_g(y +t\dot{\gamma})|\dot{\gamma}|_{g(y +t\dot{\gamma})}\,\di\mathcal{L}^n(y)\di t - \frac{3s\e}{20}\\
    &\overset{\eqref{transformationtestplangradient}}{=} \iint_0^1 |\nabla f|_g(\gamma_t)|\dot{\gamma}_t|_{g(\gamma_t)}\, \di t\di\boldsymbol{\pi}(\gamma) - \frac{3s\e}{20}\\
&\overset{\eqref{differencetesting}}{>}\iint_0^1|Df|_w(\gamma_t)|\dot{\gamma}_t|_{g(\gamma_t)} \,\di t\di\boldsymbol{\pi}(\gamma) -\frac{3s\e}{20}+\frac{3s\e}{2}.
\end{align*}
This yields a contradiction, as $|Df|_w$ fails to satisfy the condition of a weak upper gradient for the chosen test plan $\boldsymbol{\pi}$. 
\end{proof}
In the following, we will investigate $W^{1,2}_w(X)$  for $(X, \sfd, \mathfrak{m}) = (M, \sfd_g, \mu)$, with $\di\mu = h^2\di\vol_g$. \\
A mollification argument together with {Rademacher's theorem and \thref{weakgradientforc1}} shows that: 
\begin{proposition}\label{Lipschitz_weak_upper_gradient} 
    Let $(M, \sfd_g, \mu)$ be the metric measure space arising from a smooth manifold $M$ a continuous Riemannian metric $g$ and a measure given by $\di \mu = h^2\di \vol_g$ for a function $h \in C^0(M, (0, \infty))$.
    Let $f$ be a locally $\sfd_g$-Lipschitz function. Then $f$ is differentiable $\mu$-almost everywhere. If in addition $f, |\nabla f|_g \in L^2(M, \mu)$ then $|Df|_w = |\nabla f|_g$ almost everywhere, moreover $f \in W^{1,2}_w(M) \cap H^2_1(M, \mu)$, and $\norm{f}_{W^{1,2}_w(M)} = \norm{f}_{H^2_1(M, \mu)}$.
\end{proposition}
{
\begin{proof}
    Throughout this proof, we will fix a locally finite cover $(B_i, \psi_i)_{i=1}^\infty$ of $M$ such that each $B_i$ is a regular coordinate ball.
    As $f$ is locally $\sfd_g$-Lipschitz, we can apply \thref{lipschitz} (or Theorem 4.5 in \cite{burtscher2012length}), to see that for each $i\geq 1$, $(\psi_i)_*f$ is locally Lipschitz continuous with respect to the Euclidean metric on $\psi_i(B_i)$. 
    By Rademacher's theorem, we have that for each $i$,  $(\psi_i)_*f$ is differentiable $\mathcal{L}^n$-almost everywhere in $\psi_i(B_i)$. As $(\psi_i)_\#\mu$ is absolutely continuous with respect to $\mathcal{L}^n$ on $\psi_i(B_i)$, we get that $(\psi_i)_*f$ is differentiable $\mu$-almost everywhere in $\psi_i(B_i)$. As $M$ is the countable union of the $B_i$, we get that $f$ is differentiable $\mu$-almost everywhere in $M$, hence the function $|\nabla f|_g$ is well-defined in $L^\infty_{loc}(M, \mu)$.
    
    Assume $f, |\nabla f|_g \in L^2(M, \mu)$.
    We will prove that $|\nabla f|_g$ is a weak upper gradient for $f$. Let $\e > 0$.
    Fix a partition of unity $\zeta_i$ subordinate to the cover $B_i$ and proceed as follows. For each $i=1, 2, \ldots,$ choose $0 < \delta_i \leq \dist^{euc}(\supp\, \zeta_i, \partial \psi_i(B_i))$ such that
    \begin{align*}
        &\norm{\rho_{\delta_i}*(\psi_i)_*\zeta_if-(\psi_i)_*\zeta_if}_{C^0(\psi_i(B_i))} \leq \frac{\e}{2^{i}} \quad \mathrm{and}\\
        &\norm{\rho_{\delta_i}*(\psi_i)_*\zeta_if-(\psi_i)_*\zeta_if}_{W^{1,2}(\psi_i(B_i), (\psi_i)_\#\mu)} \leq \frac{\e}{2^{i}(1+ n^2(\norm{(\psi_i)_*g}_{L^\infty(\psi_i(B_i))}+\norm{(\psi_i)_*g^{-1}}_{L^\infty(\psi_i(B_i))}))}.
    \end{align*}
   This is possible by classical results about $L^p$-spaces. 
    Define $f_\e = \sum_{i=1}^\infty \rho_{\delta_i}*(\psi_i)_*\zeta_if$. As the cover $B_i$ is locally finite, $f_\e$ is well-defined.
    That produces a sequence $(f_{\e})_{\e > 0}$. 
    Note that by construction, we have that for every point $p \in M$, it holds 
    \begin{align*}
        |f_\e(p)-f(p)| \leq \e.
    \end{align*}
    Moreover,
    \begin{align*}
        \norm{f_\e-f}_{H^2_1(M, \mu)} &\leq n^2\sum_{i=1}^\infty (\norm{g}_{L^\infty(\psi_i(B_i))}+\norm{g^{-1}}_{L^\infty(\psi_i(B_i))})\norm{\rho_{\delta_i}*(\psi_i)_*\zeta_if-(\psi_i)_*\zeta_if}_{W^{1,2}(\psi_i(B_i), (\psi_i)_\#\mu)} \leq \e.
    \end{align*}
    Fix a test plan $\boldsymbol{\pi}$. For all $\e \in (0, 1]$, it holds that 
\begin{align}\label{mollified_weak_upper_gradients}
   \int  |f_\e(\gamma_0)-f_\e(\gamma_1)| \, \di\boldsymbol{\pi}(\gamma)= \int \Bigg| \int_0^1 \langle \nabla f_\e(\gamma_t), \dot{\gamma}_t\rangle_{g(\gamma_t)}\,\di t \Bigg|\di\boldsymbol{\pi}(\gamma) \leq \int \int_0^1 |\nabla f_\e(\gamma_t)|_{g(\gamma_t)}|\dot{\gamma}_t|_{g(\gamma_t)}\,\di t\, \di\boldsymbol{\pi}(\gamma).
\end{align}
As $\e \leq 1$, we have that $|f_\e(x)-f_\e(x')| \leq 2+ |f(x)-f(x')|$, for any $x, x' \in M$.
Now we can use that $\boldsymbol{\pi}$ is a probability measure together with the Hölder inequality, to get that
\begin{align*}
    \left(\int  |f(\gamma_0)-f(\gamma_1)|+2 \, \di\boldsymbol{\pi}(\gamma)\right)^2 &\leq 8 + 2\int  |f(\gamma_0)-f(\gamma_1)|^2 \, \di\boldsymbol{\pi}(\gamma) \leq 8+ 4\int |f(\gamma_0)|^2+|f(\gamma_1)|^2 \, \di\boldsymbol{\pi}(\gamma) \\
    &= 8+4 \int |f|^2 \, \di (e_0)_\#\boldsymbol{\pi}(\gamma) +4 \int |f|^2 \, \di (e_1)_\#\boldsymbol{\pi}(\gamma) \leq 8+8C(\boldsymbol{\pi})\norm{f}_{L^2}^2.
\end{align*}
Hence, we can use dominated convergence theorem to infer that  
\begin{align}\label{lhs_convergence_of_ints}
    \int  |f_\e(\gamma_0)-f_\e(\gamma_1)| \, \di\boldsymbol{\pi}(\gamma) \to \int  |f(\gamma_0)-f(\gamma_1)| \, \di\boldsymbol{\pi}(\gamma),
\end{align}
as $\e \to 0$. Moreover, we have that for all $k \in L^2(M, \mu)$, it holds
\begin{align*}
    \int \int_0^1  |k|(\gamma_t)|\dot{\gamma}_t|_{g(\gamma_t)}\,\di t\,\di\boldsymbol{\pi}(\gamma) &\leq \left(\int \int_0^1  |k|^2(\gamma_t)\,\di t\,\di\boldsymbol{\pi}(\gamma)\right)^{\frac{1}{2}}\left(\int \int_0^1 |\dot{\gamma}_t|^2_{g(\gamma_t)}\,\di t\,\di\boldsymbol{\pi}(\gamma)\right)^{\frac{1}{2}}.
\end{align*}
 By the definition of test plans, we have that 
 \begin{align*}
     \int \int_0^1 |\dot{\gamma}_t|^2_{g(\gamma_t)}\,\di t\,\di\boldsymbol{\pi}(\gamma) < \infty,
 \end{align*}
 and, using Fubini's theorem,
 \begin{align*}
     \int \int_0^1  |k|^2(\gamma_t)|\,\di t\di\boldsymbol{\pi}(\gamma) = \int_0^1 \int |k|^2 \, \di(e_t)_{\#}\boldsymbol{\pi} \di t \leq C(\boldsymbol{\pi})\norm{k}^2_{L^2}. 
 \end{align*}
 Now we can use the $L^2$-convergence of $\nabla f_\e \to \nabla f$ and choose $k = |\nabla f_\e-\nabla f|_g$ to see that 
 \begin{align}\label{rhs_convergence_of_ints}
     \int \int_0^1 |\nabla f_\e(\gamma_t)|_{g(\gamma_t)}|\dot{\gamma}_t|_{g(\gamma_t)}\,\di t\di\boldsymbol{\pi}(\gamma) \to \int \int_0^1 |\nabla f(\gamma_t)|_{g(\gamma_t)}|\dot{\gamma}_t|_{g(\gamma_t)}\,\di t\, \di\boldsymbol{\pi}(\gamma),
 \end{align}
 as $\e \to 0$. The combination of \eqref{lhs_convergence_of_ints} and \eqref{rhs_convergence_of_ints} allows passing to the limit as $\e \to 0$ in \eqref{mollified_weak_upper_gradients}, and obtain
 \begin{align*}
   \int  |f(\gamma_0)-f(\gamma_1)| \, \di\boldsymbol{\pi}(\gamma) \leq \int \int_0^1 |\nabla f(\gamma_t)|_{g(\gamma_t)}|\dot{\gamma}_t|_{g(\gamma_t)}\,\di t\di\boldsymbol{\pi}(\gamma),
 \end{align*}
 hence, $|\nabla f|_g$ is a weak upper gradient. In particular, we get that for all Lipschitz functions $f$, it holds $\norm{f}_{W^{1, 2}_w(M)} \leq \norm{f}_{H^2_1(M, \mu)}$.
 As $f \in H^2_1(M, \mu)$, we can find a sequence $f_p \in H^2_1(M, \mu)\cap C^1(M)$ such that $f_p \to f$ in $H^2_1(M, \mu)$ as $p \to \infty$.
 This implies that $f_p \to f$ in $W^{1, 2}_w(M)$ and by \thref{weakgradientforc1}, $|\nabla f_p|_g=|Df_p|_w \to |\nabla f|_g$ in $L^2(M, \mu)$, hence $|\nabla f|_g= |Df|_w$ almost everywhere, and $\norm{f}_{W^{1, 2}_w(M)} = \norm{f}_{H^2_1(M, \mu)}$.
\end{proof}
}
By \cite[Rem.\;5.5]{ambrosio2014inventio}, the slope $|Df|$ of a $\sfd_g$-Lipschitz function $f$ is a weak upper gradient. Thus, if $f \in L^2(M, \mu)$ is a $\sfd_g$-Lipschitz function such that $|Df| \in L^2(M, \mu)$, then Proposition \ref{Lipschitz_weak_upper_gradient} implies that:
\begin{align}\label{slopegradientestimate}
   |Df|_w = |\nabla f|_g \leq |Df|.
\end{align}
This observation is crucial for the following lemma. 

\begin{lemma}\thlabel{cheegerdomain}
Let $M$ be a smooth manifold, $g$ a $C^0$-Riemannian metric on $M$, $\sfd_g$ the distance induced by $g$, $\vol_g$ the volume form, and $h$ be a positive continuous function. Define the measure $\mu$ by $\di\mu = h^2\di\vol_g$. Then $W^{1,2}_w(M) \subset H^2_1(M, \mu)$ and for each $f \in W^{1,2}_w(M)$, it holds $|Df|_w = |\nabla f|_g$ { a.e.}. 
\end{lemma}
\begin{proof}
    Let $f \in W^{1,2}_w(M)$. By \cite[(2.22)]{ambrosio2014duke}, we can find a sequence of Lipschitz functions $(f_l)_{l=1}^\infty$ such that $f_l \to f$ in $L^2$ and $|Df_l| \to |Df|_w$ in $L^2$.
    Using \eqref{slopegradientestimate}, we get that $(f_l)_{l=1}^\infty$ is a bounded sequence in $H^2_1(M, \mu)$. As this space is reflexive, we can find a subsequence that converges weakly in $H^2_1(M, \mu)$. As $f_l \to f$ strongly in $L^2$, we get that $f_l \rightharpoonup f$ in $H^2_1(M, \mu)$.
As we now know that $f \in H^2_1(M, \mu)$, we can approximate it with a different sequence $f_k \in H^2_1(M, \mu)\cap C^2_1(M, \mu)$, i.e. $\lim_{l\to \infty}\norm{f_k-f}_{H^2_1(M, \mu)} = 0$. Then $f_k$ is a Cauchy sequence in $H^2_1(M, \mu)$. As for $C^2_1(M, \mu)$-functions $\norm{\cdot}_{W^{1,2}_w(M)} = \norm{\cdot}_{H^2_1(M, \mu)}$ (by \thref{weakgradientforc1}), it is a Cauchy sequence in $W^{1,2}_w$. Hence $|\nabla f_k|_g = |Df_k|_w\to |Df|_w$ in $L^2$, which in particular shows that $|Df|_w = |\nabla f|_g$ almost everywhere. 
\end{proof}
We can now finally identify the space $W^{1,2}_w$ with the classical Sobolev space $H^2_1$:
\begin{corollary}\thlabel{summary_chapter_4}
    Let $M$ be a smooth manifold with a $C^0$-Riemannian metric $g$. Consider the metric measure space $(M, \sfd_g, \mu)$, where $\di\mu = h^2\di\vol_g$ for a continuous positive function $h$. Then $W^{1,2}_w(M) = H^2_1(M, \mu)$ and the Dirichlet form associated to the Cheeger energy on $(M, \sfd_g, \mu)$ is a quadratic form given by the $L^2(M, \mu)$-inner product of the gradients, { more precisely for all functions $f, h \in W^{1,2}_w(M) = H^2_1(M, \mu)$ it holds
    \begin{itemize}
        \item[(i)] $|Df|_w = |\nabla f|_g$. 
        \item[(ii)] $\mathcal{E}(f,h) = \int_M \langle \nabla f, \nabla h \rangle_g \di \mu$
    \end{itemize}
    }
     Moreover,  every $f\in D(\Ch)$ with $|Df|_{w}\leq 1$ $\mu$-a.e. admits a $1$-Lipschitz representative $\mu$-a.e. with respect to $\sfd_{g}$.
\end{corollary}
\begin{proof}
  We  prove that $H^2_1(M, \mu) \subset W^{1,2}_w(M)$. At first we notice that by \thref{weakgradientforc1}, $C^2_1(M, \mu) \subset W^{1,2}_w(M)$ and for each $f \in C^2_1(M, \mu)$, we have that $\norm{f}_{H^2_1(M, \mu)}=\norm{f}_{W^{1,2}_w(M)}$. Now, fix $f \in H^2_1(M, \mu)$ and approximate it with a sequence $f_k \in C^2_1(M, \mu)$. 
    This is then a Cauchy sequence in $H^2_1(M, \mu)$ and hence in $W^{1,2}_w(M)$, by the norm equality we proved in \thref{cheegerdomain}. Hence, $f_k$ converges in $W^{1,2}_w(M)$ to some $\Tilde{f}$. It follows that $f_k \to \Tilde{f}$ in $L^2(M, \mu)$, hence, $\Tilde{f} = f \in W^{1,2}_w(M)$. We get $W^{1,2}_w(M) = H^2_1(M, \mu)$ from \thref{cheegerdomain}. 
    
    For the second part, we use that by \thref{carreduchamp}, the associated Dirichlet form $\mathcal{E}$ is defined on $(H^2_1(M, \mu))^2$ and is of the form
\begin{align*}
    \mathcal{E}(f, k) = \int_M G(f, k)\,\di\mu,
\end{align*}
where 
\begin{align}\label{explicitcarreduchamp}
        G(f, k)= \lim_{\e \searrow 0} \frac{|D(f+\e k)|_w^2-|Df|_w^2}{2\e} = \langle \nabla f, \nabla k\rangle_g \in L^1(X, \mathfrak{m}).
\end{align}

{ For the last part, without loss of generality by localising to a suitable coordinate patch and by using a partition of unity, we can assume that $f$ has compact support in an open set $U\subset \R^{n}$ endowed with a continuous Riemannian metric $g$ such that for every $x\in U$ it holds that $C^{-1} \,{\rm Id}_{n}\leq g \leq C\, {\rm Id}_{n}$, for some constant $C\geq 1$ uniform on $U$.  It follows that $f$ lies in $W^{1,\infty}(\R^{n})$ and thus, by the classical result in $\R^{n}$, $f$ has a Lipschitz representative $\mathcal{L}^{n}$-a.e. that we identify with $f$. Hence, $f$ is Lipschitz also in $(M,g)$ and, by Proposition \ref{Lipschitz_weak_upper_gradient} and the assumption on $f$, it holds that  $|\nabla f|_{g}=|Df|_{w}\leq 1$ $\mu$-a.e. Let $\rho_\delta(x):= \frac{1}{\delta^n}\rho(\frac{x}{\delta})$, $\delta > 0$, be standard mollifiers in $\R^{n}$ and consider $f_{\delta}:=f\star \rho_{\delta}$. It is easily seen that 
\begin{equation}\label{eq:fdeltatof}
 |\nabla f_{\delta}|_{g}\leq 1 +\theta(\delta)\quad  \text{and} \quad f_{\delta}\to f \quad \text{uniformly as }\delta\to 0.
 \end{equation}
 Here $\lim_{\delta \to 0} \theta(\delta) = 0$. 
Let $x,y\in M, \; x\neq y$. By Proposition \ref{length_space}, for every $\e\in (0, \sfd_{g}(x,y)/2)$ there exists a rectifiable curve $\gamma:[0,1]\to M$ parametrised by arc-length such that 
\begin{equation}\label{eq:arclengthxy}
{\rm{length}}_{g}(\gamma) \leq \sfd_{g}(x,y)+\e, \quad |\dot\gamma|= {\rm{length}}_{g}(\gamma), \; {\mathcal{L}}^{1}\text{-a.e. on }[0,1]. 
\end{equation}
Then 
\begin{align*}
|f_{\delta}(x)- f_{\delta}(y)| &= \left|\int_{0}^{1} \frac{\di}{\di t} f_{\delta}(\gamma_{t}) \, \di t \right| \leq   \int_{0}^{1} |\dot{\gamma}_{t}|_g \, |\nabla f_{\delta}|_{g}  \, \di t  \\
&\overset{\eqref{eq:fdeltatof}, \eqref{eq:arclengthxy}}{\leq}  (1+\theta(\delta))(\sfd_{g}(x,y)+\e).
\end{align*}
Passing to the limit as $\e\to 0$, we get that $f_{\delta}$ is $1+\theta(\delta)$-Lipschitz with respect to $\sfd_{g}$. Passing further to the limit as $\delta\to 0$ and recalling the second in \eqref{eq:fdeltatof} we conclude that $f$ is 1-Lipschitz with respect to $\sfd_{g}$.
}
\end{proof} 

\section{Second order calculus}
In this section, we will specialise the second order calculus developed in \cite[Ch.\;3]{gigli2018nonsmooth}  to the setting of a smooth manifold with a $C^{0}$-Riemannian metric with $L^2_{\rm{loc}}$-Christoffel symbols and $C^{0}\cap W^{1,2}_{\rm{loc}}$-weight on the measure. 

\subsection{{Some elements of the theory for general \texorpdfstring{$\rcd(K, \infty)$}{TEXT}-spaces}}
Let $(M, \sfd, \mathfrak{m})$ be a metric measure space that satisfies the $\rcd(K, \infty)$-condition.
Recall that 
$$D(\Delta): = \{f \in L^2(\mathfrak{m}): \Delta f \in L^2( \mathfrak{m})\} \subset W^{1,2}_w(M).$$
{ Denote by $\mathrm{Const}(M)$ the space of constant functions on $M$.}
Following the notation of \cite[Ch.\;3]{gigli2018nonsmooth},  we define the space of test functions
\begin{align*}
  \mathrm{TestF}(M):= \big\{ f \in D(\Delta)\cap L^\infty(M, \mathfrak{m}): \ |\nabla f| \in L^\infty(M, \mathfrak{m})\ \mathrm{and}\ \Delta f \in W^{1,2}_w(M)\big\}, 
\end{align*}
 test vector fields
\begin{align*}
  \mathrm{TestV}(M):= \Big\{ \sum_{i=1}^n h_i\nabla f_i \ : \ n \in \N \ f_i \in  \mathrm{TestF}(M), h_i \in  \mathrm{TestF}(M)\cup \mathrm{Const}(M)\} \ i=1, \ldots, n \Big\},
\end{align*}
and test forms
\begin{align*}
     \mathrm{TestForm}_k(M):= \big\{\mathrm{linear \ combinations\ of\ forms\ of \ the \ kind} \\
    f_0df_1 \wedge \ldots \wedge df_k:
     &\ f_1,\ldots f_k \in  \mathrm{TestF}(M), f_0 \in \mathrm{TestF} \cup \mathrm{Const}(M) \big\},
\end{align*}
for $1 \leq k \leq n$. { Set $\mathrm{TestForm}_0(M) = \mathrm{TestF}(M)$}. 
For the notion of gradients and differentials in metric measure spaces, we refer to  \cite[Ch.\;2]{gigli2018nonsmooth}. We will apply this theory to weighted manifolds with continuous metrics and weights. In that case, the non-smooth notions of \cite{gigli2018nonsmooth} coincide with  the classical gradients and differentials on manifolds. 

All expressions are well-defined, because $f \in \mathrm{TestF}(M)$ implies that $f\in W^{1,2}_w$, so the differential and the gradient are well-defined. 
{ A regularization result due to Savaré \cite{savare2013self} (see also \cite{gigli2018nonsmooth}, (3.1.5))}, yields that if $f \in L^2\cap L^{\infty}(M,\mathfrak{m})$ then for all $t>0$, $f_t:= H_tf \in  \mathrm{TestF}(M)$.   In the next proposition we recall two important density results proved in  \cite[Prop.\;2.2.5]{gigli2018nonsmooth} for general metric measure spaces.  In the setting of a smooth manifold with a continuous metric and a continuous weight, we will prove an even stronger result in \thref{testvectorfieldsdense}.
\begin{proposition}
    $\mathrm{TestForm}_1(M)$ is dense in $L^2(T^*M)$ and $\mathrm{TestV}(M)$ is dense in $L^2(TM)$. 
\end{proposition}

In \cite{gigli2018nonsmooth} it is shown that for $f, g \in \mathrm{TestF}(M)$, then $\langle \nabla f, \nabla g \rangle \in W^{1,2}_w(M)$, which allows  to define a notion of Hessian $H[f]: [\mathrm{TestF}(M)]^2 \to L^2(M)$ of a function $f \in \mathrm{TestF}(M)$ as
\begin{align*}
    H[f](g,h) = \frac{1}{2}\big( \langle \nabla \langle \nabla f, \nabla g\rangle, \nabla h\rangle + \langle \nabla \langle \nabla f, \nabla h\rangle, \nabla g\rangle - \langle \nabla \langle \nabla g, \nabla h\rangle, \nabla f\rangle\big).
\end{align*}
\begin{definition}
    The space $D(\boldsymbol{\Delta}) \subset W^{1,2}_w(M)$ is the space of functions $f \in W^{1,2}_w(M)$ such that there exists a measure $\nu \in \mathrm{Meas}(M)$ satisfying
    \begin{align*}
        \int h\, \di\nu = - \int \langle \nabla h, \nabla f \rangle \, \di\mathfrak{m},
    \end{align*}
    for all $h: M \to \R$ Lipschitz with bounded support. In this case the measure $\nu$ is unique and denoted by $\boldsymbol{\Delta}f$. 
\end{definition}
By \cite[Lemma 3.2.6]{gigli2018nonsmooth}, we have that if $X, Y \in \mathrm{TestV}(M)$, then $\langle X, Y \rangle \in D(\boldsymbol{\Delta})$.
Define 
$$D_{W^{1,2}_w}(\Delta) = \{f \in W^{1,2}_w(M),\ \Delta f \in W^{1,2}_w(M)\}.$$ 
We are now able to define the measure valued operator $\boldsymbol{\Gamma}_2: {[\mathrm{TestF}(M)]^2} \to \mathrm{Meas}(M)$ given by
\begin{align*}
    \boldsymbol{\Gamma}_2(f,g):= \frac{1}{2}\boldsymbol{\Delta}\langle \nabla f , \nabla g \rangle - \frac{1}{2}\big( \langle \nabla \Delta f , \nabla g \rangle + \langle \nabla f , \nabla \Delta g \rangle \big).
\end{align*}
In the next definition we recall the Bakry-\'Emery condition BE$(K, N)$, the reader is referred to \cite{ambrosio2015bakry, ambrosio2016loctoglob, erbar2015equivalence} for more details.
\begin{definition}
    Let $K \in \R$ and $N \in [1, \infty]$. We say that $(M, \sfd, \mathfrak{m})$ satisfies the BE$(K, N)$ condition if for every $f \in \mathrm{TestF}(M)$, it holds 
    \begin{align*}
       \boldsymbol{\Gamma}_2(f,f) \geq \Big(K|\nabla f|^2 + \frac{1}{N}(\Delta f)^2\Big)\mathfrak{m}.
    \end{align*}
\end{definition}
The following important fact was proved for $N=\infty$ in  \cite{ambrosio2014duke, ambrosio2015bakry} (see also \cite{GigliKuwadaOhtaCPAM, AGMR}) and for $N\in [1,\infty)$   in \cite[Sec.\;4]{erbar2015equivalence} and \cite[Sec.\;12]{ambrosio2019nonlinear}: 
\begin{theorem}\label{cd_implies_be}
    Let $(X, \sfd, \mathfrak{m})$ be a metric measure space, let $K \in \R$ and $N \in [1, \infty]$. The following are equivalent:
    \begin{itemize}
        \item[$\mathrm{(i)}$] $(X, \sfd, \mathfrak{m})$  is a  $\rcd^*(K,N)$-space  or, in case $N=\infty$, $(X, \sfd, \mathfrak{m})$  is a  $\rcd(K,\infty)$-space.
        \item[$\mathrm{(ii)}$] $(X, \sfd, \mathfrak{m})$ satisfies  \eqref{condition_on_measure}, the $\mathsf{BE}(K, N)$-condition, { and it  is infinitesimally Hilbertian}.  Moreover,  every $f\in D(\Ch)$ with $|Df|_{w}\leq 1$ $\mathfrak{m}$-a.e. admits a $1$-Lipschitz representative $\mathfrak{m}$-a.e. with respect to $\sfd$.  
    \end{itemize}
\end{theorem}

\subsection{Application to smooth manifolds with lower regularity Riemannian metrics and weights}
Let $M$ be a smooth manifold. Take a Riemannian metric $g \in C^{0}$ that admits $L^2_{\rm{loc}}$-Christoffel symbols and induces the distance $\sfd_g$. Define the measure $\mu$ via $\di\mu := h^2 \di\vol_g$ for a $h \in C^{0}(M; (0, \infty))\cap W^{1,2}_{\rm{loc}}(M, \mu)$. 
Moreover, we assume that  $(M, \sfd_g, \mu)$ is a $\cd(K, \infty)$-space. Thus \eqref{condition_on_measure} is  satisfied (see Remark \ref{rem:CDVol})  and  by \thref{summary_chapter_4},  $(M, \sfd_g, \mu)$ is an $\rcd(K, \infty)$-space.  

We will now specialise some concepts from \cite{gigli2018nonsmooth} to this particular case. Using integration by parts, \cite[Prop.\;2.14 (iv)]{ambrosio2014duke} and \thref{weakconv} yields:
\begin{proposition}\label{heat_flow_weak_conv_first}
    Let $f \in W^{1,2}_w(M)$ and $H_t$ denote the heat flow of the Cheeger energy. Then $f_t :=H_tf \to f$ in $W^{1,2}_w(M)$ as $t \to 0$. 
\end{proposition}
\begin{proposition}
    The space $L^2(T^*M)$ (as defined in \cite[Def.\;2.2.1]{gigli2018nonsmooth}) is isometric to $H^2_0(T^*M, \mu)$ and the space  $L^2(TM)$ (as defined in \cite[Def.\;2.3.1]{gigli2018nonsmooth}) is isometric to $H^2_0(TM, \mu)$.
\end{proposition}
 We will not give a proof as it directly follows from the (quite technical) definitions and from \cite[Prop.\;2.2.5]{gigli2018nonsmooth}; in the following we will identify the isomorphic spaces with each other. \\ 
We will now turn to a generalisation of the \textbf{divergence}. 
For a compactly supported $X \in H^2_1(TM, \mu)$ we have that $ \mathrm{div}(X)=\big(\partial_iX^i +X^{i}\frac{2}{h}\partial_ih +  X^i\frac{1}{2}\tr(g^{-1}\partial_ig)\big) \in L^2(M, \mu)$ satisfies \eqref{divergence_int_by_parts}.  
Hence, any compactly supported and smooth vector field $X$ is in $D(\mathrm{div})$, according to \cite[Def.\;2.3.11]{gigli2018nonsmooth}.

Next, we look at a generalisation of the \textbf{Hessian}. 
The identity \eqref{hessian_for_gradients} motivates the following definition of the space $W^{2, 2}(M)$. 
\begin{definition}[\cite{gigli2018nonsmooth} Def. 3.3.1]\label{synthetic_w22_def}
    The space $W^{2, 2}(M) \subset { W^{1,2}_w(M)}$ is the space of all functions $f \in  W^{1,2}_w(M)$ with the following property: There exists $A \in L^2((T^*)^{{\otimes}2}M)$ such that for any $h_1, h_2, \phi \in \mathrm{TestF}(M)$ it holds
    \begin{align}\label{equation_for_synthetic_w22}
        2\int \phi A &(\nabla h_1, \nabla h_2)\, \di\mu \nonumber \\
        &= - \int \langle \nabla f, \nabla h_1 \rangle\mathrm{div}(\phi \nabla h_2) +\langle \nabla f, \nabla h_2 \rangle\mathrm{div}(\phi \nabla h_1) + \phi\langle \nabla f, \nabla \langle \nabla h_1, \nabla h_2 \rangle \rangle \, \di\mu.
    \end{align}
We will call $A$ the Hessian of $f$ and denote it as $\Hess(f)$. The space $W^{2,2}(M)$ is endowed with the norm $\norm{\cdot}_{W^{2, 2}(M)}$ defined via 
\begin{align*}
    \norm{f}_{W^{2, 2}(M)}^2 = \norm{f}_{L^2(M)}^2+ \norm{df}^2_{L^2(T^*M)} + \norm{\Hess f}_{L^2((T^*)^{{\otimes}2}M)}^2.
\end{align*}
\end{definition}
We next investigate such a space $W^{2,2}(M)$ in case of a smooth manifold $M$ with a $C^{0}$-Riemannian metric $g$ with $L^2_{\rm{loc}}$-Christoffel symbols and a $C^0$ weighted measure $\mu$.

{ Recall that, for every $\varphi \in C^\infty_c(M) \subset L^2 \cap L^\infty(M, \mu) \cap D(\Delta)$, it holds $H_t \varphi \in \mathrm{TestF}(M)$.  Moreover, for each function $\varphi \in D(\Delta)$, we know that $\Delta H_t \varphi = H_t \Delta \varphi \to \Delta \varphi$ in $L^2$. Then, Proposition \ref{heat_flow_weak_conv_first} yields that  $H_t \varphi \to \varphi$ in $H^2_1(M, \mu)$ and \thref{linftyboundonderivative} gives that $\norm{|\nabla H_t \varphi|}_{L^\infty} \leq e^{-2Kt}\norm{|\nabla\varphi|}_{L^\infty}$. 
Hence, it follows that for each $f \in W^{2,2}(M)$, \eqref{equation_for_synthetic_w22} also holds with $\phi, h_1, h_2 \in C_c^\infty(M)$, which together with $A \in L^2$ implies that $f \in H^2_2(M, \mu)$. This fact is summarised in the following proposition.}
\begin{proposition}\label{norm_equality_h22}
    { $H^2_2(M, \mu) \supset W^{2, 2}(M)$ and for all $f \in W^{2,2}(M)$, it holds $\norm{f}_{H^2_2(M, \mu)} =\norm{f}_{W^{2, 2}(M)}$. }
\end{proposition}
Moreover, it holds that $\mathrm{TestF}(M) \subset W^{2, 2}(M)$, see  \cite[Thm.\;3.3.8]{gigli2018nonsmooth}.

A consequence of Proposition \ref{heat_flow_weak_conv_first}, \thref{weakconv}, and Mazur's lemma is:
\begin{proposition}\thlabel{normdensityw22}
Let $M$ be a smooth manifold, $g$ a $C^{0}$-Riemannian metric that admits $L^2_{\rm{loc}}$-Christoffel symbols and $\mu$ a measure on $M$ defined by $\di\mu = h^2\di\vol_g$ for a $h \in C^{0}(M, (0, \infty))\cap W^{1,2}_{\rm{loc}}(M)$. Then {$\mathrm{TestF}(M)$} is $H^2_2(M, \mu)$-dense in {$L^\infty \cap H^2_2(M, \mu)$} .  
\end{proposition}
\begin{definition}[\cite{gigli2018nonsmooth}, Def. 3.3.17]
    The space $H^{2,2}(M)$ is defined as the  $W^{2,2}$-closure of $\mathrm{TestF}(M)$.
\end{definition}
By Proposition 3.3.18 in \cite{gigli2018nonsmooth}, we have that $H^{2,2}(M)$ coincides with the $W^{2,2}$-closure of $D(\Delta_\mu)$, so we get: 
\begin{proposition}\label{w22closure}
    {$H^{2,2}(M) = \overline{H^{2}_2(M, \mu)\cap L^\infty(M, \mu)}^{H^2_2(M, \mu)}$} and it coincides with the $W^{2,2}$-closure of $D(\Delta_\mu)$.
\end{proposition}
Now, we can turn to the abstract definition of the \textbf{covariant derivative}:
\begin{definition}[\cite{gigli2018nonsmooth} Def.\;3.4.1]
    The Sobolev space $W^{1, 2}_C(TM) \subset L^2(TM)$ is defined as the space of all $X \in L^2(TM)$ such that there exists $T \in L^2(T^{\otimes2}M)$ such that for every $h_1, h_2, \p \in \mathrm{TestF}(M)$ it holds
    \begin{align*}
        \int \p T:(\nabla h_1 \otimes \nabla h_2) \, \di\mu= -\int \langle X, \nabla h_2 \rangle \mathrm{div}(\p \nabla h_1) - \p \Hess(h_2)(X, \nabla h_1)\, \di\mu.
    \end{align*}
    In this case we call the tensor $T$ the covariant derivative of $X$ and denote it by $\nabla X$. We endow $W^{1, 2}_C(TM)$ with the norm $\norm{\cdot}_{W^{1, 2}_C(TM)}$ defined by 
    \begin{align*}
        \norm{X}^2_{W^{1, 2}_C(TM)}:= \norm{X}^2_{L^2(TM)} +\norm{\nabla X}^2_{L^2(T^{\otimes2}M)}.
    \end{align*}
    We denote by $H^{1,2}_C(TM)$ the closure of $\mathrm{TestV}(M)$ in $W^{1,2}_C$. 
\end{definition}
This definition makes sense, as in \cite{gigli2018nonsmooth} it is proved that test vector fields are indeed in $W^{1,2}_C(TM)$.

\begin{proposition}
    {$\mathrm{TestV}(M)\subset H^2_1\cap L^\infty(TM, \mu)$}, and hence $H^{1,2}_C(TM)$ is a subspace of $H^2_1(TM, \mu)$.
\end{proposition}
We note that by our previous computations { we get that for $X \in H^{2,1}_C(TM)$, it holds $\nabla X = (\nabla_cX)^\sharp$ in the classical $H^2_1(TM, \mu)$-sense.} 
For $X \in W^{1,2}_C(TM)$ and $Z \in L^2(TM)$, we define $\nabla_Z X$ via
\begin{align*}
    \langle \nabla_Z X, Y \rangle_g = \nabla X : (Z \otimes Y).
\end{align*}
A computation shows that in the case of smooth vector fields, this coincides with the smooth covariant derivative. 
As for $C^1$-vector fields, we have that $\nabla_ZX=Z^{i}\partial_i X^s + \Gamma^s_{ij}X_jZ_i$, this holds by density for all vector fields {$X \in H^{2}_1\cap L^\infty(TM, \mu)$} and motivates the following definition of \textbf{Lie bracket}:
\begin{definition}[\cite{gigli2018nonsmooth}, Def.\;3.4.8]
    For $X, Y \in H^{1,2}_C(TM)$, we define $[X,Y] \in L^1(TM)$ via
    \begin{align*}
       [X,Y]:= \nabla_XY-\nabla_YX.
    \end{align*}
\end{definition}
Note that this again coincides with the Lie bracket in the smooth case and the local expressions carry over by density. The generalised \textbf{differential} is defined as follows:
\begin{definition}[\cite{gigli2018nonsmooth}, Def.\;3.5.1, 3.5.5]
    The space $W^{1,2}_d(\Lambda^kT^*M) \subset L^2(\Lambda^kT^*M)$ is the space of $k$-forms $\omega \in L^2(\Lambda^kT^*M)$ such that there exists a $(k+1)$-form $\eta \in L^2(\Lambda^{k+1}T^*M)$ for which the identity 
    \begin{align*}
        \int \eta (X_0, \ldots, X_k)\, \di\mu = & \int \sum_i (-1)^{i+1}\omega(X_0, \ldots, \hat{X}_i. \ldots, X_k)\, \di\mu \\
        &+ \int \sum_{i<j} (-1)^{i+j}\omega([X_i, X_j], X_0, \ldots, \hat{X}_i, \ldots, \hat{X}_j, \ldots, X_k)\, \di\mu,
    \end{align*}
    holds for any $X_0, \ldots, X_k \in \mathrm{TestV}(M)$. In this case we call $\eta $ the exterior differential of $\omega$ and we will denote it as $\di\omega$. We endow $W^{1,2}_d(\Lambda^kT^*M)$ with the norm $\norm{\cdot}_{W^{1,2}_d(\Lambda^kT^*M)}$ given by
    \begin{align*}
        \norm{\omega}^2_{W^{1,2}_d(\Lambda^kT^*M)} := \norm{\omega}_{L^2(\Lambda^kT^*M)}^2 +  \norm{\di\omega}_{L^2(\Lambda^{k+1}T^*M)}^2.
    \end{align*}
    Moreover, we define $H^{1,2}_d(\Lambda^kT^*M)$ as the $W^{1,2}_d$-closure of $\mathrm{TestForm}_k(M)$. 
\end{definition}
Again, this definition makes sense as in \cite{gigli2018nonsmooth} it is proved that test forms are contained in $W^{1,2}_d(\Lambda^kT^*M)$.
As $\mathrm{TestV}(M) \subset H^2_1(TM, \mu)$, we get that $H^2_1(M, \mu) \subset W^{1,2}_d(\Lambda^0T^*M)$, on the intersection the coefficients need to coincide with the ones from the classical differential,  and the norms are equivalent on the intersection. 
The cases of $k=0, 1$ will be of particular interest.
\\ Next, we recall the definition of the \textbf{codifferential}:
\begin{definition}[\cite{gigli2018nonsmooth}, Def.\;3.5.11]
    The space $D(\delta_k) \subset L^2(\Lambda^kT^*M)$ is the space of $k$-forms $\omega$ for which there exists a form $\delta\omega \in  L^2(\Lambda^{k-1}T^*M)$ called the codifferential of $\omega$, such that 
    \begin{align*}
        \int \langle \delta\omega, \eta  \rangle \, \di\mu = \int \langle \omega, \di\eta \rangle \, \di\mu
    \end{align*}
    for all $\eta \in \mathrm{TestForm}_{k-1}(M)$. In the case $k=0$, we set $D(\delta_0)= L^2(\mu)$ and define $\delta$ to be identically zero. 
\end{definition}
For $1$-forms, note that $ \omega \in D(\delta_1)$ if and only if $\omega^{\sharp} \in D(\mathrm{div})$ and, in this case,  $\delta \omega = -\mathrm{div}(\omega^{\sharp})$.
In \cite{gigli2018nonsmooth}, it is shown that for each $k$, $\mathrm{TestForm}_k(M) \subset D(\delta_k)$. Hence, the following definition makes sense:
\begin{definition}[\cite{gigli2018nonsmooth}, Def.\;3.5.13]
    The space $W^{1,2}_H(\Lambda^kT^*M)$ is defined as $W^{1,2}_d(\Lambda^kT^*M) \cap D(\delta_k)$ endowed with the norm $\norm{\cdot}_{W^{1,2}_H(\Lambda^kT^*M)}$ defined by
    \begin{align*}
        \norm{\omega}_{W^{1,2}_H(\Lambda^kT^*M)}^2 := \norm{\omega}_{W^{1,2}_d(\Lambda^kT^*M)}^2 + \norm{\delta\omega}^2_{L^2(\Lambda^{k-1}T^*M)}. 
    \end{align*}
    The space $H^{1,2}_H(\Lambda^kT^*M)$ is defined as the $W^{1,2}_H$-closure of $\mathrm{TestForm}_k(M)$. 
\end{definition}
It is not hard to check that $\mathrm{TestForm}_k(M)\subset W^{1,2}_H(\Lambda^kT^*M)$.
\begin{definition}[\cite{gigli2018nonsmooth}, Def.\;3.6.3]
    The space $H^{1,2}_H(TM)\subset L^2(TM, \mu)$ is the space of vector fields $X \in L^2(TM)$ such that $X^\flat \in H^{1,2}_H(T^*M)$ equipped with the norm $\norm{X}_{H^{1,2}_H(TM)} := \norm{X^\flat}_{H^{1,2}_H(T^*M)}$. 
\end{definition}
\begin{definition}[\cite{gigli2018nonsmooth}, Def.\;3.5.14]
    Given $k \in \N$ the domain $D(\Delta_{H, k}) \subset H^{1,2}_H(\Lambda^kT^*M)$ of the Hodge Laplacian is defined as the set of $\omega \in H^{1,2}_H(\Lambda^kT^*M)$ for which there exists an $\alpha \in L^2((\Lambda^kT^*M)$ such that 
    \begin{align*}
        \int \langle \alpha, \eta \rangle \, \di\mu = \int \langle d\omega, \di\eta \rangle \, \di\mu+\int \langle \delta\omega, \delta\eta \rangle \, \di\mu \quad \forall \eta \in H^{1,2}_H(\Lambda^kT^*M).
    \end{align*}
    In this case $\alpha$ is unique and we denote it by $\Delta_{H, k} \omega$. 
\end{definition}
Note that
\begin{align*}
    \int \langle \Delta_{H, k} \omega, \omega \rangle \, \di\mu = \int (\langle d\omega, d\omega \rangle+ \langle \delta\omega, \delta\omega \rangle) \, \di\mu.
\end{align*}
In \cite[Prop.\;3.6.1]{gigli2018nonsmooth}, it is shown that $\mathrm{TestForm}_1(M) \subset D(\Delta_{H,1})$. Moreover, $\Delta f = - \Delta_{H, 0}f$ for all $f \in \mathrm{TestF}(M)$. Finally, we recall the generalised \textbf{Ricci curvature tensor}: 
\begin{theorem}[\cite{gigli2018nonsmooth}, Thm.\;3.6.7]\label{defsyntheticricci}
    There is a unique continuous map $\mathrm{\textbf{Ric}}: [H^{1,2}_H(TM)]^2 \to \mathrm{Meas}(M)$ such that, for every $X, Y \in \mathrm{TestV}(M)$, it holds:
    \begin{align*}
        \mathrm{\textbf{Ric}}(X,Y) = \boldsymbol{\Delta}\frac{\langle X, Y \rangle}{2} + \Big( \frac{1}{2}\langle X, (\Delta_HY^\flat)^\sharp\rangle + \frac{1}{2}\langle Y, (\Delta_HX^\flat)^\sharp\rangle - \nabla X : \nabla Y\Big)\mu.
    \end{align*}
    This map is bilinear, symmetric and satisfies: 
    \begin{align*}
        \mathrm{\textbf{Ric}}(X,X) \geq Kg(X, X)\mu.
    \end{align*}
    Moreover, setting $X = Y$, we get that
    \begin{align*}
       \mathrm{\textbf{Ric}}(X,X) + (|\nabla X|^2_{HS} -  \langle X, (\Delta_HX^\flat)^\sharp\rangle)\mu=  \boldsymbol{\Delta} \frac{|X|^2}{2}.
    \end{align*}
\end{theorem}
We will now connect the measure valued Ricci tensor in the sense of Theorem \ref{defsyntheticricci} and the distributional Ricci tensor in the sense of Subsection \ref{SubSec:BochnerWRM}. 
\smallskip

\begin{proposition}\thlabel{equalityriccis}
Let $M$ be a smooth manifold with a $C^0$-Riemannian metric $g$ with $L^2_{\rm{loc}}$ Christoffel symbols. Consider $\di\mu = h^2\di\vol_g$ for a  positive function $h\in C^{0}\cap W^{1,2}_{\rm{loc}}$. Assume that $(M, \sfd_g, \mu)$ is a $\mathsf{CD}(K, \infty)$-space.
    For {$X \in \mathrm{TestV}(M)$} we have that $\mathrm{\textbf{Ric}}(X,X) = \Ric_{\mu, \infty}(X,X)$ in the sense of \eqref{districcilocalweighted}.
\end{proposition}
\begin{proof}
     By the previous observations we have that for $\p \in C_c\cap W^{1,2}_{\rm{loc}}(M)$ (which we assume to be supported in one coordinate patch), 
\begin{align*}
    \int_M \p \,\di\mathrm{\textbf{Ric}}(X,X) &= \int_M \p \, d \boldsymbol{\Delta} \frac{|X|^2}{2} + \int_M (\langle X, (\Delta_HX^\flat)^\sharp\rangle -|\nabla X|^2_{HS})\p \, \di\mu \\ 
    &= -\frac{1}{2}\int_M \langle \nabla |X|^2, \nabla \p \rangle  \, \di\mu + \int_M (\langle X, (\Delta_{\mu, H}X^\flat)^\sharp\rangle -|\nabla X|^2_{HS})\p \, \di\mu \\
    &=  -\frac{1}{2}\int_M \langle \nabla |X|^2, \nabla \p \rangle  \, \di\mu + \int_M (|dX^\flat|^2+|\delta_\mu X^\flat|^2 -|\nabla X|^2_{HS})\p \, \di\mu. 
\end{align*}
Now the last line equals the right hand side of \eqref{districcilocalweighted} as it appears in \thref{districcisobolev} so we get that locally
\begin{align*}
    &\int_M \p\, d \mathrm{\textbf{Ric}}(X,X) \\
   =&\int  X^jX^k (\Gamma^s_{kj}\Gamma^p_{ps}-\Gamma^s_{kp}\Gamma^p_{js}) \p h^2\sqrt{|g|}\, dx^1\ldots dx^n - \int \Gamma^p_{jk}\partial_p(X^jX^k \p h^2 \sqrt{|g|})\, dx^1\ldots dx^n \nonumber  \\
    &+\int \Gamma^p_{pk}\partial_j(X^jX^k \p h^2 \sqrt{|g|})\, dx^1\ldots dx^n \nonumber \\
    &+ 2\int X^jX^k(\partial_jh\partial_kh+h\partial_sh\Gamma^s_{kj})\varphi\sqrt{|g|}dx^1\ldots dx^n + 2\int \partial_jh\partial_k(X^jX^kh\varphi\sqrt{|g|})dx^1\ldots dx^n.
\end{align*}
\end{proof}
As in \cite{gigli2018nonsmooth}, we have that for all $f \in \mathrm{TestF}(M)$ it holds
\begin{align}\label{interplay_ric_gamma2}
    \mathrm{\textbf{Ric}}(\nabla f,\nabla f) + |\nabla^2 f|^2_{HS}\mu =  \boldsymbol{\Delta} \frac{|\nabla f|^2}{2} - \langle \nabla f, \nabla \Delta f\rangle\mu = \boldsymbol{\Gamma}_2(f,f).
\end{align}
\begin{proposition}\thlabel{strongheatflowconvergencew22}
     Let $M$ be a smooth manifold and $g$ a $C^{0}$-Riemannian metric with $L^2_{\rm{loc}}$-Christoffel symbols and $h \in C^{0}(M, (0, \infty))\cap W^{1,2}_{\rm{loc}}$ such that $(M, \sfd_g, \mu)$ is a $\mathsf{CD}(K, \infty)$-space, where $\di\mu= h^2\di\vol_g$. 
     \\Then,  for all $f \in H^{2,2}(M)$, it holds that 
     \begin{align*}
         \int_M |\nabla^2 f|^2h^2\,  \di\vol_g\leq \int_M |\Delta f|^2 h^2\,\di\vol_g - K\int_M |\nabla f|^2 h^2\, \di\vol_g.
     \end{align*}
    If $f \in D(\Delta)$, we have that $H_t f \to f$ strongly in $H^{2,2}$. 
\end{proposition}
\begin{proof} 
    By \cite[Cor.\;3.3.9]{gigli2018nonsmooth}, we have that
    \begin{align*}
        \int_M|\nabla^2u|^2_{HS}\,\di\mu \leq \int_M ((\Delta u)^2- K |\nabla u|^2) \, \di\mu
    \end{align*}
    for all $u \in D(\Delta)$. 
    If $f \in H^{2,2}\cap D(\Delta)$, we have that $\Delta f \in L^2$ and hence $H_t \Delta f \to \Delta f$ strongly in $L^2$ as $t \to 0$. As the semi-group $H_t$ is generated by $\Delta$, we have that $ \Delta H_t f = H_t \Delta f$ for all $t \geq 0$. 
    It follows that $\Delta H_t f \to \Delta f$ strongly in $L^2$ as $t \to 0$. Setting $u = f-H_tf$, we infer
    \begin{align*}
        \norm{\nabla^2(f-H_tf)}^2_{L^2((T^*)^{\otimes 2}M)} &\leq  - K\int_M |\nabla (f-H_tf)|^2 \, h^2\di\vol_g +  \int_M |\Delta (f-H_tf)|^2 \, h^2\di\vol_g \\
        & \leq (|K|+1)(\norm{f-H_tf}_{H^2_1(M, \mu)} + \norm{\Delta (f-H_tf)}_{L^2(M, \mu)}) \to 0
    \end{align*}
    as $t \to 0$. 
\end{proof}

\section{$\rcd$ implies distributional  Ricci curvature lower bounds}
\subsection{Some useful approximation results}
We start the section with some local considerations, so we work in $\R^n$ for simplicity of presentation. Recall the following consequence of the Poincar\'e inequality:
\begin{lemma}\thlabel{scalingpoincare}
    Let $R > 0$, $1 \leq p < \infty$ and $u \in W^{1,p}(B_R(0))$. Let $m := \frac{1}{|B_R(0)|}\int_{B_R(0)} u\, \di\mathcal{L}^n$. Then there exists a constant $C=C(p)>0$ independent of $R$ such that 
    \begin{align*}
        \norm{u-m}_{L^p(B_R(0))} \leq CR\norm{\nabla u}_{L^p(B_R(0))}.
    \end{align*}
\end{lemma}
\begin{lemma}\label{weak_star_convergence_criterion}
    Let $\Omega \subset \R^n$ be open and $\nu$ be a $\sigma$-finite measure on $\Omega$ such that $C_c^\infty(\Omega)$ is dense in $L^p(\nu)$ for each $p \in [1, \infty)$. Let $f, (f_k)_{k \geq 1}$ be measurable functions on $\Omega$ such that $\norm{f_k}_{L^\infty}$ and $\norm{f}_{L^\infty}$ are bounded and such that $f_k  \rightharpoonup f$ weakly in $L^p(\nu)$ for some $p \in [1, \infty)$. Then $f_k  \overset{*}{\rightharpoonup} f$ in the weak* topology of $(L^1)^* = L^\infty$.
\end{lemma}
\begin{proof}
    As $\nu$ is $\sigma$-finite, we have that indeed $(L^1)^* = L^\infty$. By the Banach-Alaoglu theorem, it follows that there is a  weakly* convergent subsequence of $f_k$, which we will still denote by $f_k$. Hence, $f_k$ converges weakly* to some $\Tilde{f} \in L^\infty$. Suppose $\Tilde{f} \neq f$. Then there exists an $L^1$-function $\phi$ such that 
    \begin{align*}
        \int_\Omega f\phi \, \di\nu \neq  \int_\Omega \Tilde{f}\phi \, \di\nu.  
    \end{align*}
    As $C_c^\infty(\Omega)$ is dense in $L^1$, we can assume that $\phi \in C_c^\infty \subset L^q$, where $q = \frac{p}{p-1}$. This contradicts the weak convergence in $L^p$. 
\end{proof}
\begin{lemma}\thlabel{locallycontroledcover}
    Let $K \subset B_1(0) \subset \R^n$ be compact and denote $d = \dist (K, \partial B_1(0)) > 0$. Then there exists a $\delta_0 > 0$ and a constant $C = C(n)$ such that for all $\delta \in (0, \delta_0)$, there exist an integer $m \leq C\delta^{-n}$ and points $y_1, \ldots, y_m \in B_1(0)$ such that $B_{3\delta(y_i)} \subset B_{2\delta+\frac{d}{12}}(y_i) \subset B_1(0)$ for all $1 \leq i \leq m$, $K \subset \bigcup_{i=1}^m B_\delta(y_i)$ and for all $x \in B_1(0)$, $|\{i: x \in B_{2\delta}(y_i)\}| \leq C$. 
\end{lemma}
\begin{proof}
    Let $\delta_0= \frac{d}{12}$ and fix $\delta \in (0,\delta_0)$. Now let $\{y_1, \ldots, y_m\}:= \frac{\delta}{2\sqrt{n}}\Z^n \cap B_{1-\frac{3d}{4}}(0)$. 
    It follows that $K \subset B_{1-\frac{3d}{4}-\frac{\delta}{2}}(0) \subset \bigcup_{i=1}^m B_\delta (y_i) \subset \bigcup_{i=1}^m B_{2\delta+\frac{d}{12}} (y_i) \subset B_{1-\frac{d}{4}}(0)$. 
    Moreover, there exists a constant $C_1(n)$ such that for each $x \in \R^n$, $|\frac{1}{2\sqrt{n}}\Z^n \cap B_2(x)| \leq C_1(n)$ so for each $x \in B_1(0)$, there are at most $C_1(n)$ points in $\frac{\delta}{2\sqrt{n}}\Z^n \cap B_{2\delta}(x)$.
    Finally, there exists a constant $C_2(n)$ such that for each $0 <\lambda \leq 1$, $|\frac{\lambda}{2\sqrt{n}}\Z^n \cap B_1(0)| \leq C_2(n) \lambda^{-n}$. 
    Taking $C(n) = \max(C_1(n), C_2(n))$ finishes the proof. 
\end{proof}
\begin{lemma}\thlabel{nicepartition}
    Let $R > 0$, $\varphi \in C_c^\infty(B_R(0))$. Then there exists a constant $C=C(n,R)>0$ and $\delta_0 > 0$ such that for each $\delta \in (0, \delta_0)$, there exists a set $\{\chi_1, \ldots, \chi_m\} \subset C_c^\infty(B_R(0), [0,1])$ and a family $\{y_1, \ldots, y_m\} \subset B_{R-3\delta}(0)$ such that the following holds: 
    \begin{itemize}
        \item [$\mathrm{(i)}$]  $m \leq C\delta^{-n}$,  
        \item[$\mathrm{(ii)}$] $\sum_{i=1}^m \chi_i(x) \leq 1$ for all $x \in B_R(0)$ and $\sum_{i=1}^m \chi_i(x) = 1$ for all $x \in \supp\ \p$,
        \item[$\mathrm{(iii)}$] For each $1 \leq i \leq m$, $\norm{\nabla \chi_i}_{L^{\infty}} \leq C \delta^{-1}$.
         \item[$\mathrm{(iv)}$] For all $i$, it holds $\chi_i \in C_c^\infty(B_{2\delta}(y_i))$.
         \item[$\mathrm{(v)}$] For all $x \in B_R(0)$, it holds $|\{i: x \in B_{2\delta}(y_i)\}| \leq C$. 
    \end{itemize}
\end{lemma}
\begin{proof}
    By rescaling, we can assume $R = 1$. Then take $K = \supp \, \p$ and apply \thref{locallycontroledcover} to get $\delta_0$. Choose $0 <\delta \leq \delta_0$ and take $m$, $\{y_1, \ldots, y_m\}$ as in \thref{locallycontroledcover}. It then follows that for all $i=1, \ldots, m$, $B_{3\delta}(y_i)\subset B_R(0)$, hence $y_i \in B_{R-3\delta}(0)$.
    Now, for each $i \in \{1, \ldots, m\}$ take a function $\eta_i \in C_c^\infty(B_{2\delta}(y_i), [0,1])$ such that $\eta_i(x) = 1$ for all $x \in B_{\delta}(y_i)$ and $|\nabla \eta_i| \leq C \frac{1}{\delta}$. 
    It follows that $\sum_{i=1}^m \eta_i \in C_c^\infty(B_1(0))$ and $\sum_{i=1}^m \eta_i \geq 1$ on $\supp\, \p$. 
    Now take a function $h \in C^\infty(\R)$ such that $|h'| \leq 1$, $ h(x) \geq \max(|x|, \frac{1}{4})$ and $h(x) = x$ for $x \geq 1$. Define
    \begin{align*}
        \chi_j(x) = \frac{\eta_j(x)}{h(\sum_{i=1}^m \eta_i(x))}.
    \end{align*}
    It follows that $\sum_{i=1}^m \chi_i(x) \leq 1$ for all $x \in B_1(0)$ and $\sum_{i=1}^m \chi_i(x) = 1$ for all $x \in \supp\, \p$. To see (iii), note that for all $x \in B_1(0)$
    \begin{align*}
        |\nabla \chi_j(x)| \leq \frac{|\nabla \eta_j|(x) +|h'(\sum_{i=1}^m \eta_i(x))|\sum_{i=1}^m |\nabla \eta_i|(x)|}{h(\sum_{i=1}^m \eta_i(x))^2} \leq 16C\Big(\delta^{-1} +\sum_{i: x \in \supp\,{\eta _i}}^m \delta^{-1}\Big) \leq 16(1+C(n))\delta^{-1}. 
    \end{align*}
\end{proof}
\begin{definition}
Let $\Omega \subset \R^n$ be an open set. 
    We denote by $\mathrm{CompV}(\Omega)$ the linear span of functions of the form $\rho \nabla h$, where $\rho, h \in C_c^\infty(\Omega)$. For a smooth manifold $M$ with a $C^0$-Riemannian metric $g$, we define  $\mathrm{CompV}(M)$ as the linear span of functions of the form $\rho \nabla h$, where $\rho, h \in C_c^\infty(M)$.
\end{definition}
\smallskip

\begin{lemma}\label{main_construction}
    Let $R > 0$, $B_R(0)\subset \R^n$. Let $X \in C_c^\infty(B_R(0); \R^n)$. Then there exists a constant $C= C(n,R)>0$ and a constant $C_2=C_2(n)>0$ such that for all $\e > 0$ there is a  $\Tilde{X}= \sum_{p=1}^q h_p\nabla f_p \in  \mathrm{CompV}(B_R(0))$ satisfying the following
   \begin{itemize}
       \item[$\mathrm{(i)}$] $\norm{X-\Tilde{X}}_{W^{1,1}(B_R(0))} \leq C \e$ and $\norm{X-\Tilde{X}}_{L^\infty(B_R(0))} \leq C \e$,
       \item[$\mathrm{(ii)}$] $\norm{\Tilde{X}}_{W^{1,\infty}(B_R(0))} \leq C (1+\norm{X}_{W^{2,\infty}(B_R(0))})$,
       \item[$\mathrm{(iii)}$] $q \leq C_2$,
       \item[$\mathrm{(iv)}$] for all $p$, we have that $\norm{h_p}_{L^\infty(B_R(0))} \leq C$ and $\norm{f_p}_{L^\infty(B_R(0))} \leq C\norm{X}_{C^1(B_R(0))}\e$,  
        \item[$\mathrm{(v)}$] for all $p$, we have that $\norm{\nabla h_p}_{L^\infty(B_R(0))} \leq C(1+\norm{X}_{C^2(B_R(0))})\e^{-1}$ and $\norm{\nabla f_p}_{L^\infty(B_R(0))} \leq C(1+\norm{X}_{C^1(B_R(0))})$. 
   \end{itemize}  
\end{lemma}
\begin{proof}[Proof of (i)]
Assume (by potentially readjusting) that $\e \leq \min(\frac{1}{2},\delta_0)$, where $\delta_0>0$ is as in \thref{nicepartition}.
 We will denote by $DX$ the Jacobi matrix of $X$. 
We choose
\begin{align}\label{defdeltaeps}
    \delta = \frac{\e}{8n^2(1+\norm{DX}_{L^\infty}+ \norm{D^2X}_{L^\infty})},
\end{align}
so by the mean value theorem we get that for all $y, z \in B_R(0)$ with $|y-z|\leq 4\delta$, it holds $|X(y)-X(z)|< \e$ and $|DX(y)-DX(z)|< \e$. Note that we also have $\delta < \delta_0$. 
    For this chosen $\delta$, take $m$, $\{\chi_1, \ldots \chi_m\}$ and $\{y_1, \ldots, y_m\}$ as in \thref{nicepartition}. 
    Let $\psi \in C_c^\infty(B_{3\delta}(0))$ such that $\psi(B_{2\delta}(0))=\{1\}$ and $|\nabla \psi| \leq 2 \delta^{-1}$.
    Fix an $i \in \{1, \ldots, m\}$. Let 
    \begin{align*}
        A_i= \frac{1}{|B_{2\delta}(y_i)|} \int_{B_{2\delta}(y_i)} X \di\mathcal{L}^n \in \R^n,
    \end{align*}
    and 
    \begin{align*}
        \alpha_i: \R^n \to \R, x \mapsto \psi(x-y_i)\langle A_i, (x-y_i) \rangle.
    \end{align*}
    Then $\nabla \alpha_i = A_i$ and $DA_i = 0$ on $B_{2\delta}(y_i)$. 
    For $1 \leq l,k \leq n$, let $B_i^{lk} = \partial_kX^l(y_i) \in \R$. 
    Moreover, define 
    \begin{align*}
        \rho_i^k(x) &:= \psi(x-y_i)(x_k-(y_i)_k) \in C_c^\infty(B_R(0)), \\
        \eta_i^{lk}(x)&:= \psi(x-y_i)B^{lk}_i (x_l-(y_i)_l) \in C_c^\infty(B_R(0))   \ \mathrm{and} \\
        \beta_i^{lk}(x) &:=(\rho_i^k \nabla  \eta_i^{lk})(x) =  \psi(x-y_i)(x_k-(y_i)_k) \cdot \nabla (\psi(x-y_i)B^{lk}_i (x_l-(y_i)_l) \in \mathrm{CompV}(B_R(0)).
    \end{align*}
    Then, for all $x \in B_{2\delta}(y_i)$, we have that
    \begin{align*}
        \beta_i^{lk} (x) &= (x_k-(y_i)_k)B^{lk}_ie_l \in \R^n, \quad \mathrm{and} \\
        D\beta_i^{lk} (x) &= (\delta_m^l\delta_p^k B^{lk}_i)_{m,p=1, \ldots,n } \in \R^{n \times n}.
    \end{align*}  
    Notice that 
    \begin{align*}
        \int_{B_{2\delta}(y_i)} \beta^{lk}_i d \mathcal{L}^n= 0 \in \R^n,
    \end{align*}
    as $\beta^{lk}_i$ is linear on $B_{2\delta}(y_i)$. 
    Now define $\Tilde{X} := \sum_{i=1}^m\chi_i \cdot (\nabla \alpha_i + \sum_{k,l=1}^n \beta_i^{lk}) \in \mathrm{CompV}(B_R(0))$. As $X, \Tilde{X} \in W^{1,1}_0(B_R(0))$, the Poincaré inequality implies that
    \begin{align*}
        \norm{X-\Tilde{X}}_{W^{1,1}(B_R(0))} \leq C\cdot(1+R)\norm{DX-D\Tilde{X}}_{L^{1}(B_R(0))},
    \end{align*}
    for some $C=C(n)>0$. 
    Now, as $X = \sum_{i=1}^m \chi_i X$, we get that 
\begin{align*}
    &\norm{DX-D\Tilde{X}}_{L^{1}(B_R(0))} = \int_{B_R(0)} |\sum_{i=1}^m D(\chi_iX)-D(\chi_i (\nabla \alpha_i + \sum_{k,l=1}^n \beta_i^{lk})) | \, \di\mathcal{L}^n \\
    &\qquad  \leq \int_{B_R(0)} \sum_{i=1}^m |\nabla \chi_i||X-(\nabla \alpha_i+\sum_{k,l=1}^n \beta_i^{lk})| +|\sum_{i=1}^m \chi_i(DX-D\sum_{k,l=1}^n \beta_i^{kl})| \, \di\mathcal{L}^n.
\end{align*}
For each $i$, we get that 
\begin{align}\label{linftygradient}
    \norm{DX-D\Big(\sum_{k,l=1}^n \beta_i^{kl}\Big)}_{L^\infty(B_{2\delta}(y_i))}= \norm{DX-DX(y_i)}_{L^\infty(B_{2\delta}(y_i))} \leq n^2\e.
\end{align}
Thus, using that $\chi_i \geq 0$, we get that
\begin{align*}
    \int_{B_R(0)}\Big|\sum_{i=1}^m \chi_i(DX-D\sum_{k,l=1}^n \beta_i^{kl})\Big| \, \di\mathcal{L}^n \leq  \int_{B_R(0)}\sum_{i=1}^m \chi_i \Big|(DX-D\sum_{k,l=1}^n \beta_i^{kl})\Big| \, \di\mathcal{L}^n \leq n^2|B_R(0)|\e. 
\end{align*}
For each $i$, we have $|\nabla \chi_i| \leq C(n,R) \delta^{-1}$ and 
\begin{align*}
    \int_{B_R(0)} |\nabla \chi_i| \Big|X-\nabla \alpha_i+\sum_{k,l=1}^n \beta_i^{kl}\Big| \, \di\mathcal{L}^n &= \int_{B_{2\delta}(y_i)}|\nabla \chi_i| \Big|X-\nabla \alpha_i+\sum_{k,l=1}^n \beta_i^{kl}\Big| \, \di\mathcal{L}^n \\
    &\leq C(n,R) \delta^{-1} \norm{X-\nabla \alpha_i+\sum_{k,l=1}^n \beta_i^{kl}}_{L^1(B_{2\delta}(y_i))}.
\end{align*}
By \thref{scalingpoincare} and \eqref{linftygradient}, we have that 
\begin{align*}
  \norm{X-\nabla \alpha_i+\sum_{k,l=1}^n \beta_i^{kl}}_{L^1(B_{2\delta}(y_i))} &\leq  2C \delta \norm{DX-D\Big(\nabla \alpha_i+\sum_{k,l=1}^n \beta_i^{kl}\Big)}_{L^1(B_{2\delta}(y_i))} \\
  &=2C \delta \norm{DX-D\Big(\sum_{k,l=1}^n \beta_i^{kl}\Big)}_{L^1(B_{2\delta}(y_i))} \\
  &\leq C\delta^{n+1} \e.
\end{align*}
Thus,
\begin{align*}
    \int_{B_R(0)} |\nabla \chi_i||X-\nabla \alpha_i+\sum_{k,l=1}^n \beta_i^{kl}| \, \di\mathcal{L}^n  \leq  C(n, R) \e \delta^n.
\end{align*}
Now, as $m \leq C(n,R)\delta^{-n}$, we get that 
\begin{align*}
    \int_{B_R(0)} \sum_{i=1}^m |\nabla \chi_i||X-\Big(\nabla \alpha_i+\sum_{k,l=1}^n \beta_i^{lk}\Big)| \leq C(n,R)^2\e. 
\end{align*}
All together this gives
\begin{align*}
    \norm{X-\Tilde{X}}_{W^{1,1}(B_R(0))} \leq C(1+R)\norm{DX-D\Tilde{X}}_{L^{1}(B_R(0))} \leq C(n,R)\e.
\end{align*}
Moreover, note that for each $x \in B_R(0)$,
\begin{align*}
    |X-\Tilde{X}|(x) \leq  \Big|\sum_{i=1}^m \chi_i(X- \nabla \alpha_i)\Big|(x) + \sum_{i=1}^m\sum_{l,k} |\chi_i\beta^{kl}_i|(x) \leq C(n)\e. 
\end{align*}
This concludes the proof of (i). 

\smallskip 
\textit{Proof of (ii).} Next, we want to investigate the $L^\infty$-norm of $D\Tilde{X}$. We get that for a point $y \in B_R(0)$
\begin{align*}
      |DX-D\Tilde{X}|(y) &= \Big|\sum_{i=1}^m D(\chi_iX)-D\Big(\chi_i (\nabla \alpha_i + \sum_{k,l=1}^n \beta_i^{lk})\Big) \Big|(y) \\
     &\leq \Big(\sum_{i=1}^m |\nabla \chi_i|| \Big|X-(\nabla \alpha_i+\sum_{k,l=1}^n \beta_i^{lk})\Big|\Big)(y)+ \Big|\sum_{i=1}^m \chi_i\Big(DX-D\sum_{k,l=1}^n \beta_i^{kl}\Big)\Big|(y).
\end{align*}
We will investigate the two sums separately. For a fixed $i$, we know that $|\nabla \chi_i|(y) \leq C(n,R)\delta^{-1}$ and $|\chi_i|(y) \leq 1$. It suffices to only consider $y \in B_{2\delta}(y_i)$, in which case we have that $B_{2\delta}(y_i) \subset B_{4\delta}(y)$. By our choice of $\delta$, we have that for all $z \in B_{2\delta}(y_i)$ it holds $|X(z)-X(y)| \leq \e$.
Hence, 
\begin{align*}
    & |X(y)-\nabla \alpha_i| \\
    &= \left| X(y)- \frac{1}{|B_{2\delta}(y_i)|}\int_{B_{2\delta}(y_i)} X(z)\, \di\mathcal{L}^n(z) \right| \\
    & \leq \frac{1}{|B_{2\delta}(y_i)|}\int_{B_{2\delta}(y_i)} |X(z)-X(y)|\, \di\mathcal{L}^n(z) \leq \e.
\end{align*}
Moreover,
\begin{align*}
   \Big|\sum_{k,l=1}^n \beta_i^{lk}\Big|(y) \leq n^2 \delta (1+|DX|_{L^\infty}). 
\end{align*}
This gives that 
\begin{align*}
    |\nabla \chi_i| \Big|X-\Big(\nabla \alpha_i+\sum_{k,l=1}^n \beta_i^{lk}\Big)\Big|(y) &\leq C(n)\delta^{-1} \Big(|X(y)-\nabla \alpha_i| + \Big|\sum_{k,l=1}^n \beta_i^{lk}\Big|(y)\Big)\\
    &\leq C(n)\delta^{-1}(\e +\delta |DX|_{L^\infty}) \\
    & \leq C(n)(1+|DX|_{L^\infty}+ |D^2X|_{L^\infty}), 
\end{align*}
where in the last estimate we used the precise dependence of $\e$ in terms of $\delta$, \eqref{defdeltaeps}.
The choice of $\delta$ also gives 
\begin{align*}
    \Big|DX-D\sum_{k,l=1}^n \beta_i^{kl}\Big|(y) \leq n^2 \e.
\end{align*}
Noting that $|i: y \in \supp \chi_i| \leq C(n)$, we get that
\begin{align*}
    |DX-D\Tilde{X}|(y) &\leq \sum_{i=1}^m |\nabla \chi_i| \Big|X-\Big(\nabla \alpha_i+\sum_{k,l=1}^n \beta_i^{lk}\Big)\Big|(y)+\Big|\sum_{i=1}^m \chi_i\Big(DX-D\sum_{k,l=1}^n \beta_i^{kl}\Big)\Big|(y) \\
    & \leq C(n)\Big( \big(1+\norm{DX}_{L^\infty}+ \norm{D^2X}_{L^\infty}\big) + n^2 \e\Big).
\end{align*}
It follows that there exists a constant $C(n, R)$ such that
\begin{align*}
    \norm{D\Tilde{X}}_{L^\infty} \leq \norm{DX-D\Tilde{X}}_{L^\infty} + \norm{DX}_{L^\infty} \leq C(n, R)(1+\norm{X}_{C^2(B_R(0))}).
\end{align*} 
As $\Tilde{X} \in C_c^\infty(B_R(0))$, we get that also $\norm{\Tilde{X}}_{L^\infty} \leq CR (1+\norm{X}_{C^2(B_R(0))})$, which proves (ii). 

\smallskip 
\textit{Proof of (iii)}
For the last three statements, let us explicitly spell out how we can build $\Tilde{X}$ out of finitely many (bounded by $C_2= C_2(n)$ independent of $\delta$) $C^\infty_c$-functions and their gradients. 
Recall that by the proof of \thref{locallycontroledcover}, we have chosen the $y_i$ to be in $\frac{\delta}{2\sqrt{n}}\Z^n \cap B_{R}(0)$.
Note that $\Big|\Big(\faktor{\Z}{12n\Z}\Big)^n\Big| = (12n)^n$. 
For $\xi \in \Big(\faktor{\Z}{12n\Z}\Big)^n$, let 
$$Z_\xi := \{\zeta \in \Z^n: \zeta_s \equiv \xi_s \mod 12n, \forall 1 \leq s \leq n\}.$$
 Define $Y_\xi :=\{i: y_i \in \{y_1, \ldots, y_m\} \cap \frac{\delta}{2\sqrt{n}}Z_\xi\}$. 
If $Y_\xi = \emptyset$, define  $\chi_\xi=\alpha_\xi=\rho_\xi^k=\eta^{lk}_\xi=0 \in C_c^\infty(B_R(0))$, where $k,l \in \{1, \ldots, n\}$. Otherwise, define 
\begin{align*}
    \chi_\xi &:=  \sum_{i \in Y_\xi} \chi_i \in C_c^\infty(B_R(0)), \\
    \alpha_\xi &:= \sum_{i \in Y_\xi} \alpha_i \in C_c^\infty(B_R(0)), \\
    \rho_\xi^k &:= \sum_{i \in Y_\xi} \rho^k_i \in C_c^\infty(B_R(0)), \ \mathrm{for\ } 1 \leq k \leq n \ \mathrm{and} \\
    \eta^{lk}_\xi &:= \sum_{i \in Y_\xi} \eta^{lk}_i \in C_c^\infty(B_R(0)), \ \mathrm{for\ } 1 \leq k,l \leq n.
\end{align*}
We note that 
\begin{align}\label{condition_of_xi}
    B_{3\delta}(y_i)\cap B_{3\delta}(y_{i'}) = \emptyset, \ \mathrm{if\ } i, {i'} \in Y_\xi \  \mathrm{and} \ i\neq i', 
\end{align}
as then $|y_i-y_{i'}| \geq 6 \sqrt{n} \delta$. 
 Now, for all $i$, we have that 
 \begin{align*}
     \supp\, \alpha_i \cup \supp\, \chi_i \cup \bigcup_{k=1}^n\supp\, \rho_i^k \cup \bigcup_{k,l=1}^n \supp\ \eta^{lk}_i \subset B_{3\delta}(y_i),
 \end{align*}
so 
 \begin{align*}
     \chi_\xi \nabla \alpha_\xi + \sum_{k,l=1}^n \chi_\xi \rho_\xi^k \nabla \eta^{lk}_\xi &= \left(\sum_{i \in Y_\xi} \chi_i\right)\left(\sum_{i \in Y_\xi} \nabla \alpha_i\right) + \sum_{k,l=1}^n\left(\sum_{i \in Y_\xi} \chi_i\right)\left(\sum_{i \in Y_\xi} \rho^k_i \right)\left(\sum_{i \in Y_\xi}\nabla  \eta^{lk}_i \right) \\
     &= \sum_{i \in Y_\xi} \chi_i\nabla \alpha_i + \sum_{i \in Y_\xi} \chi_i \sum_{k,l=1}^n \rho_i^k \nabla  \eta^{lk}_i.
 \end{align*}
 Hence,
\begin{align*}
    \Tilde{X} = \sum_{\xi \in \Big(\faktor{\Z}{12n\Z}\Big)^n} \left(  \chi_\xi \nabla \alpha_\xi + \sum_{k,l=1}^n \chi_\xi \rho_\xi^k \nabla \eta^{lk}_\xi \right).
\end{align*}
These are at most $2\cdot (12n)^n(n^2+1) =: C_2(n)$ functions, proving (iii). 

\smallskip \textit{Proof of (iv) and (v).}
To prove the last two statements, we may again fix a $\xi \in \Big(\faktor{\Z}{12n\Z}\Big)^n$ as above. Take an $x \in B_R(0)$. 
If $x \notin B_{3\delta}(y_i)$ for some $i \in Y_\xi$, then all of the above functions and their gradients are zero when evaluated at $x$.
Otherwise, by \eqref{condition_of_xi}, there exists exactly one $i \in Y_\xi$ such that $x \in B_{3\delta}(y_i)$. Then we have,
\begin{align*}
    &|\chi_\xi(x)| = |\chi_i(x)| \leq 1, \\
   &| \nabla \chi_\xi(x)| = | \nabla \chi_i(x)| \leq C\delta^{-1} \leq C(1+\norm{X}_{C^2})\e^{-1}, \\ 
   &|\alpha_\xi(x)| = |\alpha_i(x)| \leq C\norm{X}_{C^0}\delta \leq C (1+\norm{DX}_{C^1})\delta \leq C\e,
\end{align*}
and
\begin{align*}
     |\nabla \alpha_\xi(x)| &= |\nabla \alpha_i(x)| \leq |\nabla \psi(x-y_i)\langle A_i, (x-y_i)\rangle|(x)+|\psi(x-y_i)A_i|(x) \leq \delta^{-1} \cdot C\norm{X}_{C^0}\delta + C\norm{X}_{C^0} \\
     &\leq C(1+ \norm{X}_{C^0}).
\end{align*}
Moreover,
\begin{align*}
     &|\chi_\xi\rho_\xi^k(x)| = |\chi_i\rho_i^k(x)| \leq C\delta \leq C\e  
\end{align*}
and 
\begin{align*}
     |\nabla (\chi_\xi\rho_\xi^k)|(x) &= |\nabla( \chi_i\rho_i^k)|(x) \leq |(\nabla \chi_i)|(x)|\rho_i^k|(x) + |\chi_i|(x)(|\nabla \psi(x-y_i)||(x_k-(y_i)_k)| +|\psi(x-y_i)|) \\
     &\leq C(\delta^{-1}\cdot \delta) + C(\delta^{-1} \delta +1) \leq C.
\end{align*}
Finally,
\begin{align*}
    |\eta_\xi^{lk}|(x) = |\eta_i^{lk}|(x)=|\psi(x-y_i)||B^{lk}_i||(x_l-(y_i)_l)| \leq C \norm{X}_{C^1} \delta 
\end{align*}
and 
\begin{align*}
    |\nabla \eta_\xi^{lk}|(x) &= |\nabla \eta_i^{lk}|(x) \leq |\nabla \psi(x-y_i)||B^{lk}_i||x_l-(y_i)_l| + |\psi(x-y_i)||B^{lk}_i| \\
    &\leq C\delta^{-1}\norm{X}_{C^1}\delta + \norm{X}_{C^1} \leq C\norm{X}_{C^1}. 
\end{align*}
As $x$ was arbitrary in $B_R(0)$, this holds for all $x$, and as $\xi$ was arbitrary, this holds for all $\xi$, which proves the lemma. 
\end{proof}
We now deduce a corollary about Sobolev spaces in $\R^n$. This is not related to our further study but it is worth mentioning.
\smallskip

\begin{corollary}
Let $\Omega \subset \R^n$ be an open subset. Then $\mathrm{CompV}(\Omega)$ is dense in $W^{1,p}_0(\Omega)$ for all $p \in [1,\infty)$.
\end{corollary}
We are now going to apply the previous results to manifolds. The next lemma follows from partitioning the manifold in balls and patching the approximations from Lemma \ref{main_construction} together.
\smallskip

\begin{lemma}\label{patch_together}
    Let $M$ be a smooth manifold with a $C^{0}$-Riemannian metric $g$ that admits $L^2_{\rm{loc}}$-Christoffel symbols and a measure $\mu$ defined via $\di\mu = h^2 \di\vol_g$, where $h \in C^{0}(M, (0,\infty))$. Let $X$ be a compactly supported smooth vector field on $M$ and let $\e >0$. 
    Then there exist a constant $C= C(M, g, h, X)$ and a vector field $\Tilde{X} = \sum_{p=1}^q h_p \nabla f_p \in \mathrm{CompV}(M)$ such that 
 \begin{itemize}
       \item[$\mathrm{(i)}$] $\norm{X-\Tilde{X}}_{H^1_1(TM, \mu)} \leq C \e$,
       \item[$\mathrm{(ii)}$] $\sup_M(|\Tilde{X}|_g+ |\nabla \Tilde{X}|_g)  \leq C$,
       \item[$\mathrm{(iii)}$] $q \leq C$,
       \item[$\mathrm{(vi)}$] for all $p$, we have that $\norm{h_p}_{L^\infty} \leq C$ and $\norm{f_p}_{L^\infty} \leq C\e$,  
        \item[$\mathrm{(v)}$] for all $p$, we have that $\norm{|\nabla h_p|_g}_{L^\infty} \leq C\e^{-1}$ and $\norm{ |\nabla f_p|_g}_{L^\infty} \leq C$. 
   \end{itemize}  
\end{lemma}

 In the next proposition we invoke the parabolic De Giorgi-Nash-Moser theory \cite{ladyzhenskaia1968linear} to obtain $C^{0,\alpha}$ convergence of the heat flow to the initial datum. The non-triviality of the statement relies on the fact that the Riemannian metric is merely $C^0$ with $L^2_{\rm{loc}}$ Christoffel symbols, and the weight on the measure is merely $C^0{\cap W^{1,2}_{\rm{loc}}(M)}$.

\smallskip

\begin{proposition}\label{c0_convergence_heat_flow}
    Let $M$ be a smooth manifold and $g$ a $C^0$-Riemannian metric with $L^2_{\rm{loc}}$ Christoffel symbols. Let moreover $h \in C^0(M, (0, \infty)){\cap W^{1,2}_{\rm{loc}}(M)}$. Define the measure $\mu$ on $M$ via $\di\mu = h^2\di\vol_g$. Let $K \in \R$ and assume that $(M, \sfd_g, \mu)$ is an $\rcd(K, \infty)$-space. Denote by $H_t$ the heat flow on $M$. Let $\Omega \subset M$ be open such that $\Omega$ is contained in one coordinate patch $(U, \psi)$ and $\overline{\Omega}$ is compact. Then there exists an $\alpha \in (0,1)$ such that for each $\rho \in C_c^\infty(M)$,  $T > 0$, and for each open $\Tilde{\Omega} \subset \subset \Omega$, it holds 
    $({\rm id}_{[0,1]}\times \psi)_*H_t\rho|_{\Tilde{\Omega}} \in C^{0,\alpha}([0,T), \Tilde{\Omega})$.
\end{proposition}
\begin{proof}
    Throughout this proof we write $\rho_t$ for $\psi_*H_t\rho|_{{\Omega}}$. As $\overline{\Omega}$ is compact, we get that there exists a constant $\lambda>0$ such that  $ \frac{1}{\lambda} \leq h^2\sqrt{|g|} \leq \lambda$ and $\frac{1}{\lambda}{\rm Id}_{n} \leq g^{-1} \leq \lambda {\rm Id}_{n}$ on $\Omega$. 
    We note that $\rho_t$ weakly solves
    \begin{align*}
        \frac{\di}{\di t} \rho_t = \Delta_\mu \rho_t
    \end{align*}
    on $\Omega$. 
    Observing  that for any two functions $f, \phi \in C_c^\infty(\Omega)$ it holds 
    \begin{align*}
        \int_M \Delta_\mu f \cdot \phi \, \di\mu = - \int_M \langle \nabla_g f, \nabla_g \phi \rangle_g \, \di\mu = - \int_M h^2\sqrt{|g|}g^{ij}\partial_i f\partial_j \phi \, \di\mathcal{L}^n,
    \end{align*}
    we get that $\rho_t$  weakly solves
    \begin{align}
        \partial_t \rho_t -\mathrm{div} (A D\rho_t) = 0,
    \end{align}
    where 
    \begin{align}
        A_{ij} = h^2\sqrt{|g|}g^{ij}.
    \end{align}
    By the previous observations, $A$ is uniformly elliptic on $\Omega$. Recalling that $\rho\in C_c^\infty(M)$, the result follows from the parabolic DeGiorgi-Nash-Moser theory, see for instance Theorem 1.1 in Chapter V of \cite{ladyzhenskaia1968linear}.
\end{proof}
We are now ready to prove the main approximation result of the section that will be key in the proof of the main theorem of the paper.
\begin{lemma}\thlabel{testvectorfieldsdense}
    Let $M$ be a smooth manifold endowed with a $C^{0}$-Riemannian metric $g$ that admits $L^2_{\rm{loc}}$-Christoffel symbols and a continuous positive weight $h$. Let $\mu$ be the measure defined via $\di\mu=h^2 dvol_g$. 
    Assume that the metric measure space $(M,\sfd_g,\mu)$ satisfies the $\mathsf{CD}(K,\infty)$-condition.
    Then for each relatively compact, open $U \subset M$ and vector field $X \in C^\infty_c(TM)$ there exists a sequence $(W_j)_j \subset \mathrm{TestV}(M)$ such that $W_j \to X$ in $H^2_1(TU, \mu)$ and $(W_j)_j$ is bounded in $L^\infty(U, \mu)$.
\end{lemma}
\begin{proof}
    { As a first observation we note that by Corollary \ref{summary_chapter_4}, the metric measure space $(M, \sfd_g, \mu)$ is infinitesimally Hilbertian, hence an $\rcd(K, \infty)$-space, which in particular enables us to apply all the theory from the previous section.}
    {Let $\e > 0$. We can find a constant $C = C(M, g, h, X)$ and functions $f_p, h_p \in C_c^\infty(M)$ where $p \in \{1, \ldots q\}$ such that, denoting $\Tilde{X} := \sum_{p=1}^q h_p \nabla f_p$ satisfy conditions (i)-(v) from Lemma \ref{patch_together}.}
   We first note that by interpolation, $\norm{X-\Tilde{X}}_{H^2_1(TM, \mu)} \leq C \sqrt{\e}$.
   We have that  $C_c^\infty \subset D(\Delta_\mu) \cap W^{2,2}(M) \cap L^\infty(M)$, { so for all $p$, it holds $f_p,h_p \subset H^{2}_2\cap L^\infty(M, \mu) \subset D(\Delta)$}. Thanks to \thref{strongheatflowconvergencew22} { together with the equality of norms established in Proposition \ref{norm_equality_h22}}, we get that for each $p$, there exists a $t \in (0,\frac{1}{|K|}]$, such that
\begin{align*}
    &\max(\norm{h_p- H_{t}h_p}_{H^2_2(M, \mu)}, \norm{f_p- H_{t}f_p}_{H^2_2(M, \mu)}) < \frac{\e^2}{\max_p \big(\norm{f_p}_{H_2^2(M, \mu)}+\norm{h_p}_{H_2^2(M, \mu)}\big)}.
\end{align*} 
Fix a relatively compact, open set $U \subset M$. Then by  using Proposition \ref{c0_convergence_heat_flow} on a finite, relatively compact cover, we can potentially decrease $t$, to get  
\begin{align*}
     &\max(\norm{h_p- H_{t}h_p}_{L^\infty(U, \mu)}, \norm{f_p- H_{t}f_p}_{L^\infty(U, \mu)}) < \frac{\e^2}{\max_p \big(\norm{f_p}_{H_2^2(M, \mu)}+\norm{h_p}_{H_2^2(M, \mu)}\big)},
\end{align*}
for all $p = 1, \ldots, q$. 
Define $\Tilde{f}_p :=  H_{t}f_p$ and $\Tilde{h}_p :=  H_{t}h_p$ for all $p$.
Moreover, recall that for each $t\in \big(0, \frac{1}{|K|}\big]$ and each $f\in W^{2,2}\cap L^\infty $, we have that 
\begin{align*}
    &\norm{H_tf}_{L^\infty(M, \mu)} \leq \norm{f}_{L^\infty(M, \mu)} \ \mathrm{and }\\
     &\norm{|\nabla H_tf|_g}_{L^\infty(M, \mu)} \leq e^{-Kt}\norm{|\nabla f|_g}_{L^\infty(M, \mu)},
\end{align*}
by \thref{linftyboundonderivative} and the maximum principle of the heat flow. 
Hence,
\begin{align*}
    &\norm{\Tilde{f}_p}_{L^\infty} \leq \norm{f_p}_{L^\infty} \leq C\e, \\
    &\norm{\Tilde{h}_p}_{L^\infty} \leq \norm{h_p}_{L^\infty} \leq C, \\
    & \norm{|\nabla\Tilde{f}_p|_g}_{L^\infty} \leq C, \ \mathrm{and} \\
    & \norm{|\nabla\Tilde{h}_p|_g}_{L^\infty} \leq C \e^{-1}.
\end{align*}
Define $W := \sum_{p=1}^q \Tilde{h}_p\nabla \Tilde{f}_p $. We directly get that 
\begin{align*}
    \norm{|W|_g}_{L^\infty(M, \mu)} = \norm{ \Big|\sum_{p=1}^q \Tilde{h}_p\nabla \Tilde{f}_p\Big|_g}_{L^\infty(M, \mu)} \leq \sum_{p=1}^q \norm{| \Tilde{h}_p\nabla \Tilde{f}_p|_g}_{L^\infty(M, \mu)} \leq qC^2.
\end{align*}
 To conclude, we estimate that 
\begin{align*}
    \norm{W-\Tilde{X}}_{L^2(TM, \mu)} &\leq \sum_{p=1}^q \norm{\Tilde{h}_p\nabla \Tilde{f}_p-h_p \nabla f_p}_{L^2(TM, \mu)} \\
    &\leq  \sum_{p=1}^q \norm{\Tilde{h}_p\nabla \Tilde{f}_p-\Tilde{h}_p \nabla f_p}_{L^2(TM, \mu)} + \norm{\Tilde{h}_p\nabla f_p-h_p \nabla f_p}_{L^2(TM, \mu)}\\
    &\leq \sum_{p=1}^q \norm{\Tilde{h}_p}_{L^\infty(M, \mu)}\norm{\nabla \Tilde{f}_p-\nabla f_p}_{L^2(TM, \mu)} + \norm{\Tilde{h}_p-h_p}_{L^2(M, \mu)}\norm{\nabla \Tilde{f}_p}_{L^\infty(TM, \mu)}\\
    &\leq C(M, g, h, X)\e. 
\end{align*}
For the covariant derivative, we note that 
\begin{align*}
    \nabla W = \left(\nabla_c \sum_{p=1}^q \Tilde{h}_p\nabla \Tilde{f}_p \right)^\sharp = \left(\sum_{p=1}^q \Tilde{h}_p \Hess \Tilde{f}_p + \di\Tilde{h}_p \otimes \nabla \Tilde{f}_p \right)^\sharp = \sum_{p=1}^q \Tilde{h}_p (\Hess \Tilde{f}_p)^\sharp  + \nabla \Tilde{h}_p \otimes \nabla \Tilde{f}_p.
\end{align*}
Then, by similar arguments as the previous ones, we get that 
\begin{align*}
    \norm{\nabla W-\nabla \Tilde{X}}_{L^2(T^{\otimes 2}U, \mu)} \leq& \sum_{p=1}^q \norm{ \nabla(\Tilde{h}_p\nabla \Tilde{f}_p)-\nabla(h_p \nabla f_p)}_{L^2(T^{\otimes 2}U, \mu)} \\
    \leq& \sum_{p=1}^q \norm{\Tilde{h}_p}_{L^\infty(U, \mu)}\norm{\Hess (\Tilde{f}_p- f_p)^\sharp}_{L^2(T^{\otimes2}U, \mu)} + \norm{\Tilde{h}_p-h_p}_{L^\infty(U, \mu)}\norm{(\Hess \Tilde{f}_p)^\sharp}_{L^2(T^{\otimes2}U, \mu)} \\
    &+ \sum_{p=1}^q \norm{\nabla \Tilde{h}_p}_{L^\infty(TU)}\norm{\nabla \Tilde{f}_p-\nabla f_p}_{L^2(TU, \mu)} + \norm{\nabla(\Tilde{h}_p-h_p)}_{L^2(TU, \mu)}\norm{\nabla \Tilde{f}_p}_{L^\infty(TU, \mu)} \\
    \leq & C(M,U,g,h,X)\e. 
\end{align*}
Finally, we estimate that
\begin{align*}
    \norm{X-W}_{H^2_1(TU, \mu)} \leq \norm{X-\Tilde{X}}_{H^2_1(TU, \mu)} + \norm{\Tilde{X}-W}_{H^2_1(TU, \mu)} \leq C(M,U,g,h,X)\sqrt{\e}. 
\end{align*}
\end{proof}

\subsection{The case $N=\infty$}

We are now able to prove the first main result:
\smallskip

\begin{theorem}\label{first_main_theorem}
     Let $M$ be a smooth manifold and $g \in C^{0}(M)$ be a metric with $L^2_{\rm{loc}}$ Christoffel symbols  and let $h \in C^{0}(M, (0, \infty))\cap W^{1,2}_{\rm{loc}}(M)$ be such that $(M, \sfd_g, \mu)$, with $\di\mu = h^2\di\vol_g$ is a $\mathsf{CD}(K, \infty)$-space. Then 
    \begin{align*}
        \Ric_{\mu, \infty} \geq Kg
    \end{align*}
    in the distributional sense. 
\end{theorem}
\begin{proof}
    From \thref{summary_chapter_4}, we get that $(M, \sfd_g, \mu)$ is indeed an $\rcd(K, \infty)$-space, so all the observations from the previous chapter apply.
    Our aim is to show that $\Ric_g \geq Kg$ distributionally, so we fix a smooth vector field $X \in \mathcal{T}^1_0$ and a test volume $\omega = \phi h^2\vol_g$, where $\phi \in C_c\cap W^{1,2}_{\rm{loc}}(M)$. 
    Using a smooth cut-off function, we can assume that $X$ is compactly supported. We can (using a partition of unity) again assume that $\phi$ is supported in one coordinate patch. We will therefore directly work in an open, relatively compact set $U \subset \R^n$.  From \eqref{defdistributionalweightedricci}, we know that 
    \begin{align*}
       \int_U \Ric_{\mu, \infty}(X,X)\omega = & \int_U  X^jX^k (\Gamma^s_{kj}\Gamma^p_{ps}-\Gamma^s_{kp}\Gamma^p_{js}) \phi h^2\sqrt{|g|}\, dx^1\ldots dx^n \\
       &- \int_U \Gamma^p_{jk}\partial_p(X^jX^k \phi h^2\sqrt{|g|})\, dx^1\ldots dx^n +\int_U \Gamma^p_{pk}\partial_j(X^jX^k \phi h^2\sqrt{|g|})\, dx^1\ldots dx^n  \\
    &+ 2\int_U X^jX^k(\partial_jh\partial_kh+h\partial_sh\Gamma^s_{kj})\phi\sqrt{|g|}dx^1\ldots dx^n + 2\int_U \partial_jh\partial_k(X^jX^kh\phi\sqrt{|g|})dx^1\ldots dx^n \\
    &= \int_U  X^jX^k (\Gamma^s_{kj}\Gamma^p_{ps}-\Gamma^s_{kp}\Gamma^p_{js}) \phi h^2 \sqrt{|g|}\, dx^1\ldots dx^n - \int \Gamma^p_{jk}\partial_p(X^j)X^k \phi h^2\sqrt{|g|}\, dx^1\ldots dx^n   \\
    &- \int_U \Gamma^p_{jk}X^j\partial_p(X^k \phi h^2\sqrt{|g|})\, dx^1\ldots dx^n +\int_U \Gamma^p_{pk}\partial_j(X^j)X^k \phi h^2\sqrt{|g|}\, dx^1\ldots dx^n \\
    &+\int_U \Gamma^p_{pk}X^j\partial_j(X^k \phi h^2\sqrt{|g|})\, dx^1\ldots dx^n \\
    &+ 2\int_U X^jX^k(\partial_jh\partial_kh+h\partial_sh\Gamma^s_{kj})\phi\sqrt{|g|}dx^1\ldots dx^n
    + 2\int_U \partial_jh\partial_kX^jX^kh\phi\sqrt{|g|})dx^1\ldots dx^n \\
    &+ 2\int_U \partial_jhX^j\partial_k(X^kh\phi\sqrt{|g|})dx^1\ldots dx^n.
    \end{align*}
    By the assumptions on $g$ and $h$ and Lemma \ref{christoffel_l2_implies_dg_l2}, there exist functions $\Psi_{i,j,k} \in L^2(U, \mathcal{L}^n)$ and $\Phi_{j,k} \in L^1(U, \mathcal{L}^n)$ all compactly supported in $U$ such that \begin{align*}
        \int_U \Ric_{\mu, \infty}(X,X)\omega = \int_U X^jX^k \Phi_{j,k} \, \di\mathcal{L}^n + \int_U \partial_i X^jX^k\Psi_{i,j,k} \, \di\mathcal{L}^n. 
    \end{align*} 
     By \thref{testvectorfieldsdense}, there exists a sequence $V_\delta \in \mathrm{TestV}(M)$ such that $\norm{|V_\delta|_g}_{L^\infty}(U) \leq C$ for some constant $C > 0$ that does not depend on $\delta$ and $\norm{V_\delta - X}_{H^2_1(TU, \mu)} \to 0$ as $\delta \to 0$. Fix an $\e > 0$. 
     Note that $V_\delta \otimes V_\delta$ is bounded in $L^\infty$ and converges to $X \otimes X$ in $L^1$. By Lemma \ref{weak_star_convergence_criterion}, $V_\delta \otimes V_\delta \overset{*}{\rightharpoonup} X \otimes X$ in $L^\infty = (L^1)^*$ and $V_\delta \overset{*}{\rightharpoonup} X$ in $L^\infty = (L^1)^*$.
     Hence, we can find a $\delta_0 > 0$ such that  $\norm{V_\delta - X}_{H^2_1(TU, \mu)} < \e$ and $\norm{V_\delta \otimes V_\delta - X\otimes X}_{L^1(T^{\otimes 2}U, \mu)} < \e$ for all $\delta \in (0, \delta_0)$. 
     As $g$ and $h^2\sqrt{|g|}$ are uniformly bounded from above and from below on $U$, we get that $\norm{V_\delta - X}_{W^{1,2}(\R^n, \mathcal{L}^n)} < C\e$ and $\norm{V_\delta \otimes V_\delta - X\otimes X}_{L^1(\R^{n \times n}, \mathcal{L}^{n \times n})} < C\e$ for all $\delta \in (0, \delta_0)$ and a $C= C(M, g, h, U) > 0$.
     Now we can use the weak* convergence, to choose a $\delta_1 \in (0, \delta_0)$ such that  
     \begin{align*}
         &\left|\int_U X^jX^k \Phi_{j,k} \, \di\mathcal{L}^n - \int_U V_{\delta_1}^jV_{\delta_1}^k \Phi_{j,k} \, \di\mathcal{L}^n\right| \leq \e \ \mathrm{and} \\
         &\left|\int_U \partial_i X^jX^k\Psi_{i,j,k} \, \di\mathcal{L}^n- \int_U \partial_i X^jV_{\delta_1}^k\Psi_{i,j,k} \, \di\mathcal{L}^n\right|\leq \e.
     \end{align*}
    From now on we will denote $V_{\delta_1} = V$ for simplicity. 
    We get that
    \begin{align*}
        \left|\int_U \partial_i X^jV^k\Psi_{i,j,k} \, \di\mathcal{L}^n- \int_U \partial_i V^jV^k\Psi_{i,j,k} \, \di\mathcal{L}^n\right|\leq C(M, g, h, U, \phi, X)\e.
    \end{align*} 
   Now note that 
   \begin{align*}
      \left| \int_M \langle V, V \rangle_g\phi \, h^2\di\vol_g - \int_M \langle X, X \rangle_g\phi \, h^2\di\vol_g \right|  \leq C  \norm{V \otimes V - X\otimes X}_{L^1(\R^{n \times n})} \leq C\e,
   \end{align*}
   where again $C=C(M, g, h, U, \phi, X)$. 
   Finally, we use that 
   \begin{align}
        \int_M \phi h^2\, d \mathrm{\textbf{Ric}}(V, V)= \int_U \Ric_{\mu, \infty}(V,V)\omega = \int_U V^jV^k \Phi_{j,k} \, \di\mathcal{L}^n + \int_U \partial_i V^jV^k\Psi_{i,j,k} \, \di\mathcal{L}^n. 
   \end{align}
   All together, we get that 
   \begin{align*}
       \int_M \Ric(X,X)\omega &\geq \int_M \phi h^2\, d \mathrm{\textbf{Ric}}(V, V) - C\e \\
       &\geq K\int_M \langle V, V \rangle_g\phi h^2 \, \di\vol_g - C\e \\
       &\geq K\int_M \langle X, X \rangle_g\phi h^2 \, \di\vol_g -(2+|K|)C\e \\
       &=  K\int_M g(X,X)\omega - (2+|K|)C\e.
   \end{align*}
   Sending $\e \to 0$ yields the result. 
\end{proof}
\begin{lemma}\label{symmetric_matrix_apprx}
    Let $\Omega \subset \R^n$ be open and $M \in C_c^\infty(\Omega; \R^{n\times n}_{sym})$ be pointwise positive semidefinite. Let $\e > 0$. Then there exist a function $\p \in C_c^\infty(\Omega, [0,1])$ and vector fields $b_k \in C^\infty(\Omega, \R^n)$ for $k=1, \ldots, n$, such that for $M_\e := \sum_k \p b_k \otimes b_k$, it holds
    \begin{align*}
        |M-M_\e|_{C^1(\R^n)} \leq C\e,
    \end{align*}
    where $C= C(\Omega, \supp\, M, n)$ does not depend on $\e$. 
\end{lemma}
\begin{proof}
    First we recall a basic fact about matrices. Let $P \in \R^{n \times n}$ be positive definite. Then there exists a unique positive definite matrix $A \in \R^{n \times n}$ such that $A^TA= P$. Writing $A = (A_{ij})_{ij}$, we get that $P_{ij} = (A^TA)_{ij} = \sum_{k=1}^n A_{ki}A_{kj} = \sum_k (A_k \otimes A_k)_{ij}$, where $A_k$ denotes the $k$-th row vector of $A$ and $\otimes$ the dyadic product. \\
    Now define $\Tilde{M}:= M + \e {\rm Id}_{n} \in C^\infty(\Omega; \R^{n\times n}_{sym})$ and note that $\Tilde{M}$ is constant outside $\supp\, M$. By construction, we have that $\Tilde{M}$ is positive definite everywhere in $\Omega$ and its smallest eigenvalue is bounded below by $\e$. Similarly, its largest eigenvalue is bounded above by $||{M}|_{HS}|_{C^0(\Omega)}+ \e=:m$. 
    Set 
    $$U:= \left\{a+ib: a \in \Bigg(\frac{1}{2}\e, m + 2\Bigg), b \in (-1, 1)\right\}$$ and $K := [\e, m+1]\subset \R_+ \subset \C$.
    For a matrix $B \in \C^{n \times n}$, we denote by $\sigma_{\C^{n \times  n}}(B)$ the set of its complex eigenvalues. Note that for each $p \in \Omega$, $\sigma_{\C^{n \times n}}(\Tilde{M}(p)) \subset K$.
    Let $\gamma$ be a smooth cycle in $U\setminus K$ such that $n(\gamma, w) = 1$ for all $w \in K$, where $n(\gamma, w)$ denotes the winding number of $\gamma$ around $w$. We denote by $s: U \to \C$ the unique holomorphic function which is defined by $s^2(z)=z$ and $s(w) \geq 0$ for $w \in \R \cap U$. Define
    \begin{align*}
        B(p) := \frac{1}{2\pi i}\int_\gamma s(z)(z{\rm Id}_{n}-\Tilde{M}(p))^{-1}\, dz.
    \end{align*}
    Using holomorphic functional calculus we infer that $B(p)$ is the unique positive square root of $\Tilde{M}(p)$ for all $p \in \Omega$. As $\gamma$ is a fixed curve with positive distance to $K$, all functions in the integral are smooth and bounded on the domain of integration. As $\Tilde{M}(p)$ is a smooth function of $p$, so is $B(p)$ and each of its rows, which we will denote by $b_k$ for $k=1, \ldots, n$.
    Let $\p \in C_c^\infty(\Omega, [0,1])$ be a non-negative cut-off function such that $\p \equiv 1$ on $\supp \, M$. We can choose $\p$ such that $|\nabla \p|$ is bounded by $C(n)\cdot \dist(\partial \Omega, \supp\, M)$. Now define 
    \begin{align*}
        M_\e := \p \sum_k b_k \otimes b_k.
    \end{align*}
    By definition, we have that 
    \begin{align*}
    |M_\e - \Tilde{M}|(p) = \left\{ \begin{array}{ll}
           0 & \mathrm{if}\ p \in \supp\, M,  \\
          \p |\e {\rm Id}_{n}| \leq C(n)\e &\mathrm{if}\ p \notin \supp\, M.
        \end{array}\right.
    \end{align*}
    Similarly, 
    \begin{align*}
    |\nabla(M_\e - \Tilde{M})|(p) = \left\{ \begin{array}{ll}
           0 &\mathrm{if}\ p \in \supp\, M,  \\
          \nabla \p |\e {\rm Id}_{n}| \leq C(\Omega, \supp\, M, n)\e &\mathrm{if}\ p \notin \supp\, M.
        \end{array}\right.
    \end{align*}
    By definition, we know that $|M-\Tilde{M}|_{C^1} \leq C(n)\e$, so combining the estimates above, we get that
    \begin{align*}
       |M-M_\e|_{C^1}  \leq |M-\Tilde{M}|_{C^1} + |M_\e - \Tilde{M}|_{C^1} \leq C \e.
    \end{align*}
\end{proof}
\begin{theorem}\label{ricciasmeasure}
    Let $M$ be a smooth manifold endowed with a $C^{0}$-Riemannian metric $g$ that admits $L^2_{\rm{loc}}$-Christoffel symbols and a positive function $h \in C^{0}(M)\cap W^{1,2}_{\rm{loc}}(M)$. Define the measure $\mu$ via $\di\mu = h^2 \di\vol_g$.
    Suppose the distributional Bakry-Émery $\infty$-Ricci curvature as defined in \eqref{defdistributionalweightedricci} is bounded below by $K$ for some $K \in \R$. 
    For each coordinate patch $U$, there exist regular Radon measures $\nu_{ij}$ for $i,j=1, \ldots, n$ such that for a test volume $\omega$ with $\supp\, \omega \subset U$ (locally written as $\omega = \p h^2\sqrt{|g|} dx^1\wedge \ldots \wedge dx^n$) and smooth vector fields $X, Y$, we have that in local coordinates it holds
    \begin{align}\label{ricci_as_measure}
        \Ric_{\mu, \infty}(X, Y)\omega = \sum_{i,j}\int X^{i}Y^j\p h^2 \sqrt{|g|}\di\nu_{ij}(x_1, \ldots, x_n).
    \end{align}
\end{theorem}
\begin{proof}
    Using the local trivialisation $\psi: U \to \Omega \subset \R^n$, we may assume that we are working on the open set $\Omega$ and will express everything in local coordinates. Define 
    \begin{align*}
        D_{ij} := \partial_p \Gamma^p_{ij} - \partial_i \Gamma^p_{pj} + \Gamma^s_{ij}\Gamma^p_{ps} - \Gamma^s_{jp}\Gamma^p_{is} - \frac{2}{h^2}(h\partial_{ij}h -\partial_ih\partial_jh - \Gamma^s_{ij}h\partial_sh) - K g_{ij} \in \mathcal{D}'(\Omega). 
    \end{align*}
    By definition, $D_{ij}$ is a distribution of order at most $1$ and for any smooth vector field $X$ and any test volume $\omega = \p h^2\sqrt{|g|}dx^1\wedge\ldots\wedge dx^n$, for a $\p \in C_c\cap W^{1,2}_{\rm{loc}}(M)$ we get that
    \begin{align*}
        (\Ric(X,X)_{\mu, \infty}-Kg(X,X))\omega = \sum_{i,j} \langle D_{ij}, X^{i}X^j\p h^2\sqrt{|g|}\rangle_{\mathcal{D}', \mathcal{D}}.
    \end{align*}
    Now define $D \in \mathcal{D}'(\Omega, \R^{n \times n}_{sym})$ via $\langle D, M\rangle_{\mathcal{D}', \mathcal{D}} := \sum_{i,j} \langle D_{ij}, M_{ij} \rangle_{\mathcal{D}', \mathcal{D}}$. Denote $B(X, \p):= \p h^2\sqrt{|g|}(X \otimes X)$. Then
    \begin{align*}
        (\Ric(X,X)_{\mu, \infty}-Kg(X,X))\omega = \sum_{ij} \langle D_{ij}, B(X, \p)_{ij}\rangle_{\mathcal{D}', \mathcal{D}} = \langle D, B(X, \p)\rangle_{\mathcal{D}', \mathcal{D}}.
    \end{align*}
    Now let $B \in C_c^\infty(\Omega; \R^{n\times n}_{sym})$ be pointwise positive semidefinite and $\e > 0$.
    Let $B_\e$ be as in Lemma \ref{symmetric_matrix_apprx}. As $\Ric_{\mu, \infty} -Kg \geq 0$, we get that 
    \begin{align*}
        \langle D, B_\e \rangle_{\mathcal{D}', \mathcal{D}} \geq 0.
    \end{align*}
    Moreover, 
    \begin{align*}
        |\langle D, (B_\e-B) \rangle_{\mathcal{D}', \mathcal{D}}| \leq C(\Omega, \supp B,g, h) |B-B_\e|_{C^1(\Omega)} \leq C (\Omega, B, n, g, h) \e.
    \end{align*}
    Hence,
    \begin{align*}
        \langle D, B \rangle_{\mathcal{D}', \mathcal{D}} \geq -C (\Omega, B, n, g, h) \e.
    \end{align*}
    Letting $\e \to 0$, we get that for any such $B$, 
    \begin{align}\label{positive_dist_on_sym}
        \langle D, B \rangle_{\mathcal{D}', \mathcal{D}} \geq 0.
    \end{align}
    Now fix $S \subset \Omega$ compact and let $A \in C_c^\infty(\Omega; \R^{n\times n}_{sym})$ (not necessarily positive semidefinite) such that $\supp \, A \subset S$ and $\sup_S|A|_{HS} \leq 1$. Then $A$ only has eigenvalues in $[-1, 1]$ at each point in $\Omega$. Let $\rho \in C_c^\infty(\Omega, [0,1])$ such that $\rho \equiv 1$ on $S$. It then follows that $\rho {\rm Id}_{n}- A$ is positive semidefinite in $\Omega$. Then by \eqref{positive_dist_on_sym},
    \begin{align*}
        \langle D, A \rangle_{\mathcal{D}', \mathcal{D}} \leq \langle D, \rho {\rm Id}_{n} \rangle_{\mathcal{D}', \mathcal{D}}.
    \end{align*}
    In other words:
    \begin{align*}
        \sup_{A \in C_c^\infty(\Omega; \R^{n\times n}_{sym}), \supp A \subset S, ||A|_{HS}|_{C^0}\leq 1} \langle D, A \rangle_{\mathcal{D}', \mathcal{D}} \leq \langle D, \rho {\rm Id}_{n} \rangle_{\mathcal{D}', \mathcal{D}} < \infty. 
    \end{align*}
    As $S$ is arbitrary, we have proven that $D \in C_c(\Omega; \R^{n\times n}_{sym})' \cong \big(C_c(\Omega, \R)^{\frac{n(n+1)}{2}}\big)' \cong (C_c(\Omega, \R)')^{\frac{n(n+1)}{2}}$. 
    By Riesz' Theorem, $C_c(\Omega, \R)'$ equals the space of regular Radon measures on $\Omega$, so it follows that there exist regular Radon measures $\Tilde{\nu}_{ij}$ on $\Omega$ for $i,j=1, \ldots, n$ such that
    $$\langle D, M \rangle =\sum_{i,j} \int_\Omega M_{ij}\di\Tilde{\nu}_{ij}, \quad \text{for all } M \in C_c(\Omega; \R^{n\times n}_{sym}).$$ 
    Define the regular Radon measure 
    \begin{align*}
        \nu_{ij} := \Tilde{\nu}_{ij} + Kg_{ij}\mathcal{L}^n.
    \end{align*}
     Now for any vector field $X$ and any compactly supported test volume $\p h^2\sqrt{|g|}dx^1\wedge, \ldots, \wedge dx^n$, we have that
    \begin{align*}
         \int_{\Omega} X^{i}X^j\p h^2\sqrt{|g|}g_{ij}K\di\mathcal{L}^n =\langle Kg(X,X), \omega\rangle_{\mathcal{D}', \mathcal{D}}.
    \end{align*}
    By the definition of $B(X, \p)$ and $D$, we get that 
    \begin{align*}
        \langle\Ric(X, X)_{\mu, \infty},\omega\rangle_{\mathcal{D}', \mathcal{D}} = \sum_{i,j}\int X^{i}X^j\p h^2\sqrt{|g|}\di\nu_{ij}(x_1, \ldots, x_n).
    \end{align*}
   By  using that $\Ric_{\mu, \infty}(\cdot, \cdot)$ is a symmetric bilinear form and considering $\Ric_{\mu, \infty}(X+Y, X+Y)$, we conclude that \eqref{ricci_as_measure} holds.
    \end{proof}
\begin{remark}
From the proof of  Theorem \ref{ricciasmeasure} it also follows that $\sum_{i,j} |\nu_{ij}|(S) < \infty$ for each compact subset $S \subset \Omega$. Moreover, the components $(\nu_{ij})_{ij}$ change tensorially  under coordinate transformations.  
\end{remark}
    
    \subsection{The case $N\in [n,\infty)$}
    
Finally, we want to examine the case when $(M, \sfd_g, \mu)$ is a $\mathsf{CD}^*(K, N)$ space for some $N \geq 1$. By \thref{summary_chapter_4}, we know that $(M, \sfd_g, \mu)$ is infinitesimally Hilbertian, hence we know that it is $\rcd^*(K, N)$. Then, by Theorem \ref{cd_implies_be}, we get that 
 $(M, \sfd_g, \mu)$ satisfies the $\mathsf{BE}(K, N)$-condition
\begin{align*}
    \boldsymbol{\Gamma}_2(f,f) = \frac{1}{2}\boldsymbol{\Delta}\langle \nabla f, \nabla f \rangle - \langle \nabla \Delta f, \nabla f \rangle\mu \geq (K |\nabla f|^2 + \frac{1}{N}(\Delta f)^2)\mu, \quad \text{for all }f \in \mathrm{TestF}(M).
\end{align*}
Testing with a non-negative function $\p \in C_c\cap W^{1,2}_{\rm{loc}}(M)$, that we may assume to be supported in one coordinate patch, we obtain: 
\begin{align*}
    -\frac{1}{2}&\int_M \langle \nabla \langle \nabla f, \nabla f \rangle, \nabla \p \rangle h^2\sqrt{|g|} dx^1\ldots dx^n -\int_M \p \langle \nabla \Delta f, \nabla f \rangle  h^2\sqrt{|g|} dx^1\ldots dx^n \\ 
    & \geq \int_M (K |\nabla f|^2 + \frac{1}{N}(\Delta f)^2)\p h^2\sqrt{|g|} dx^1\ldots dx^n.
\end{align*}
Using \eqref{districcilocalweighted} and \thref{equalityriccis} and \eqref{interplay_ric_gamma2}, we get that 
\begin{align}\label{beconditiondistri}
    \int&  (\nabla f)^j(\nabla f)^k (\Gamma^s_{kj}\Gamma^p_{ps}-\Gamma^s_{kp}\Gamma^p_{js}) \p h^2\sqrt{|g|}\, dx^1\ldots dx^n - \int \Gamma^p_{jk}\partial_p((\nabla f)^j(\nabla f)^k \p h^2 \sqrt{|g|})\, dx^1\ldots dx^n \nonumber \\
    &+\int \Gamma^p_{pk}\partial_j((\nabla f)^j(\nabla f)^k \p h^2 \sqrt{|g|})\, dx^1\ldots dx^n \nonumber  \\
    & + \int\frac{2}{h^2}(\nabla f)^j(\nabla f)^k(\partial_kh\partial_jh+ h\Gamma^s_{jk}\partial_sh) \p h^2 \sqrt{|g|} \, dx^1\ldots dx^n \nonumber \\
    & + 2\int\partial_{k}h\partial_j(h(\nabla f)^j(\nabla f)^k \p  \sqrt{|g|}) \, dx^1\ldots dx^n + \int_M \p|\nabla^2 f|_{HS}^2h^2 \sqrt{|g|} \, dx^1\ldots dx^n \nonumber \\
    \geq &\int_M (K |\nabla f|^2 + \frac{1}{N}(\Delta_\mu f)^2)\p h^2 \sqrt{|g|} \, dx^1\ldots dx^n.
\end{align}
Now take a function $\Tilde{f} \in C^2_c(M,\mu)$. It follows that $\Tilde{f} \in H^{2, 2}(M,\mu) \cap D(\Delta_\mu)$, so by Proposition \ref{strongheatflowconvergencew22} the heat flow $H_t \Tilde{f} =: \Tilde{f}_t$ converges strongly to $\Tilde{f}$ in $H^{2,2}(M) \subset H^2_2(TM, \mu)$. 
Then $\Tilde{f}_t \to \Tilde{f}$ in $H^{2,2}\subset H^2_2$ so we can conclude that \eqref{beconditiondistri} holds for any $\Tilde{f} \in C^2_c(M)$ and any $\p \in C^1_c(M)$. 
With Theorem \ref{ricciasmeasure}, this gives:
\smallskip

\begin{proposition}\label{be_for_smooth_functions}
Let $M$ be a smooth manifold endowed with a $C^{0}$-Riemannian metric $g$ that admits $L^2_{\rm{loc}}$-Christoffel symbols and a measure $\mu$ defined via $\di\mu = h^2\di\vol_g$, for a positive function $h \in C^{0}\cap W^{1,2}_{\rm{loc}}(M)$. If $(M, \sfd_g, \mu)$ is a $\mathsf{CD}^*(K, N)$-space, then for all $f \in C^2_c(M)$, it holds 
   \begin{align}\label{smoothbecond}
    \sum_{ij}(\nabla f)_i(\nabla f)_j \nu_{ij} + (|\nabla^2 f|^2_{HS} - K|\nabla f|^2 - \frac{1}{N}(\Delta_\mu f)^2)\mu \geq 0  
\end{align}
in fixed local coordinates.
\end{proposition}
An application of Theorem 4.7 in \cite{cheeger1982finite} shows:
\begin{lemma}\label{injectivity_radius_bound}
    Let $M$ be a smooth manifold, $g$ a $C^{0}$-Riemannian metric on $M$ and $U \subset V \subset M$ both open such that $\overline{U} \subset V$  and $\overline{V}$ is compact. Moreover denote by $g_\e = g * \rho_\e$. 
    Denote $G = \norm{g}_{C^{0}(\overline{V})}+ \norm{g^{-1}}_{C^0(\overline{V})}$. Then there exists an $\e_0>0$ such that for all $x \in U$ and $0 < \e<\e_0$, it holds
    \begin{align*}
        i_\e(x) \geq C(\e, G, U, V, \rho) >0,
    \end{align*}
    where $i_\e(x)$ denotes the injectivity radius with respect to the metric $g_\e$. 
\end{lemma}
\smallskip

\begin{lemma}\label{domain_bounded_expo_diffs}
 Let $M$ be a smooth manifold, $g$ a $C^{0}$-Riemannian metric that admits $L^2_{\rm{loc}}$-Christoffel symbols on $M$ and $U \subset V \subset M$ both open such that $\overline{U} \subset V$, and $\overline{V}$ is compact and contained in one coordinate patch. Moreover let $g_\e := g * \rho_\e$. 
    Denote $G = \norm{g}_{C^{0}(\overline{V})}+ \norm{g^{-1}}_{C^0(\overline{V})}$. For all $p \in U$, there exists an $\e_0 > 0$ such that for all $\e \in (0, \e_0)$, there exists a constant $C = C(U, V, \e, G, p, \rho) >0$ with the following property: for all vectors $\Tilde{v} \in T_pM$ with $\norm{\Tilde{v}}_{g_\e} \leq \min(1, C)$,  it holds
    \begin{align}\label{differentials_expo_bounds}
    &\big|D\exp_p^{g_\e}|_{\Tilde{v}}\big|_{euc} \leq 3, \quad \mathrm{and} \quad \big|D^2\exp_p^{g_\e}|_{\Tilde{v}}\big|_{euc} \leq 3.
    \end{align}
\end{lemma}
\begin{proof}
    By covering and rescaling, we can assume that $U = B^{euc}_R(0) \subset \R^n$ and $V = B^{euc}_{R+2}(0) \subset \R^n$.
    Let $\e_0 > 0$ small enough for Lemma \ref{injectivity_radius_bound} to hold, and such that $\e_0 \leq \frac{1}{2}$ and for all $\e \in (0, \e_0)$, $g_\e$ is a Riemannian metric on $B_R(0)$ satisfying
    \begin{align*}
       \norm{g_\e}_{C^{0}(\overline{B^{euc}_{R}(0)})}+ \norm{g_\e^{-1}}_{C^0(\overline{B^{euc}_{R}(0)})} \leq 2G.
    \end{align*}
    Choose $\e \in (0, \e_0)$ and denote by $i_\e$ the injectivity radius with respect to $g_\e$. 
    Fix a point $p \in U$ and $v \in T_pU$ such that $\norm{v}_{euc} \leq 1$. { By the relative compactness of $V$ and arguments as in the proof of Proposition \ref{lipschitz} (or alternatively Theorem 4.5 in \cite{burtscher2012length}),} there exists  $\lambda= \lambda(G) \geq 1$ such that $g, g_\e$, and $|\cdot|_{euc}$ are pairwise $\lambda$-equivalent on $V$ and their induced metrics are pairwise $\lambda$-equivalent on $B_{1/2}^{euc}(p)$.
    For a curve $c:[0,1] \to M$ and $W$ a vector field along the $c$, we denote by $\dot{W}$ the covariant time derivative and by $W'$ the derivative in local coordinates. 
    Denote $\gamma:= \gamma^\e_v$ the geodesic with respect to $g_\e$, originating in $p$ and tangent to $v$. Note that $|\dot{\gamma}(t)|_{g_\e}$ is constant in the existence domain of $\gamma$, hence we have that for all such $t$, 
    \begin{align}\label{bound_on_geodesic_speed}
        |\dot{\gamma}(t)|_{euc} \leq \lambda^2 |\dot{\gamma}(0)|_{euc}= \lambda^2 |v|_{euc}. 
    \end{align}
    For any vector $w \in T_pU$ define $J^\e_{v,w}$ to be the Jacobi field in metric $g_\e$ along $\gamma$ with $\dot{J}^\e_{v,w}(0) = \nabla_{\dot{\gamma}(0)} J^\e_{v,w}(0) = w$. 
    Recall the Jacobi equation for the smooth metric $g_\e$:
    \begin{align*}
        \Ddot{J}^\e_{v,w}(t) = \nabla_{\dot{\gamma}(t)} (\nabla_{\dot{\gamma}(t)}J^\e_{v,w})(t) = R(J^\e_{v,w}(t), \dot{\gamma}(t))\dot{\gamma}(t).
    \end{align*}
    In order to keep notation short, we write $J_w$ for $J^\e_{v,w}$. In local coordinates, we get:
    \begin{align*}
        &(\dot{J}_{w}(t))^k = (\nabla_{\dot{\gamma}(t)} J_{w}(t))^k = \frac{\di}{\di t}J_w^k(t) + J^j(t)\dot{\gamma}^i(t){\Gamma_\e}^k_{ij}(\gamma(t)), \ \mathrm{and} \\
        & (\Ddot{J}_{w}(t))^l = \frac{\di}{\di t}\Big(  \frac{\di}{\di t}J_w^l(t) + J^j(t)\dot{\gamma}^i(t){\Gamma_\e}^l_{ij}(\gamma(t))\Big) + {\Gamma_\e}^l_{pk}(\gamma(t))\dot{\gamma}^p(t) \Big(\frac{\di}{\di t}J_w^k(t) + J^j(t)\dot{\gamma}^i(t){\Gamma_\e}^k_{ij}(\gamma(t))\Big).
    \end{align*}
    Hence, the Jacobi equation in local coordinates gives
    \begin{align*}
         (R_\e)^l_{ijk}J_w^{i}(t)\dot{\gamma}^j(t)\dot{\gamma}^k(t) =& \frac{\di^2}{\di t^2} J_w^l(t) + \frac{\di}{\di t}J^j(t)\dot{\gamma}^i(t){\Gamma_\e}^l_{ij}(\gamma(t)) + J^j(t){\gamma^{i}}''(t){\Gamma_\e}^l_{ij}(\gamma(t))  \\ 
         & + J^j(t)\dot{\gamma}^i(t)d{\Gamma_\e}^l_{ij}(\gamma(t))(\dot{\gamma}(t)) + {\Gamma_\e}^l_{pk}(\gamma(t))\dot{\gamma}^p(t) \Big(\frac{\di}{\di t}J_w^k(t) \\
         & + J^j(t)\dot{\gamma}^i(t){\Gamma_\e}^k_{ij}(\gamma(t))\Big).
    \end{align*}
    We have that 
    \begin{align}
        &|D^2 g_\e| \leq C(\rho,G)\e^{-2}, \ \mathrm{and} \quad |D g^{-1}_\e| \leq C(\rho, G)\e^{-1},
    \end{align}
    where we may assume that $C>1$.
    Moreover, recall that we can use the geodesic equation and \eqref{bound_on_geodesic_speed} to bound $|\gamma''(t)|_{euc}$ in terms of $|v|_{euc}$, $G\e^{-1}$ and $\lambda$.
 Hence, all coefficients of the above linear system of ODEs are bounded by $C\e^{-2}$. Denote $Y_w^T(t) := (J^T(t), (J')^T(t))$. We get that 
\begin{align}
    &Y'_w(t) = P(t)Y(t), \nonumber \\
    & Y_w(0)^T = (0^T, w^T),
\end{align}
where $P$ is a matrix with smooth entries whose $L^\infty$-norm is bounded above by $C\e^{-2}$.  Gronwall's inequality gives that for $u \in T_pM$
\begin{align*}
    |Y_w(t)- Y_{w+su}(t) |_{euc} \leq |s||u|_{euc} e^{Ct\e^{-2}}. 
\end{align*}
Hence,
\begin{align*}
    |J'_w(t)- J'_{w+su}(t)|_{euc} \leq |s||u|_{euc} e^{Ct\e^{-2}}. 
\end{align*}
Thus,
\begin{align*}
    |J_w(t)- J_{w+su}(t)|_{euc} \leq |s||u|_{euc} \frac{\e^2}{C}(e^{Ct\e^{-2}}-1). 
\end{align*}
It is known that $J_w(t) = D\exp_p^{g_\e}|_{tv}(tw)$.
Hence, denoting $a := C\e^{-2}$, we get that for $0 < t \leq \frac{1}{a}$,
\begin{align*}
      \big|D\exp_p^{g_\e}|_{tv}(w)- D\exp_p^{g_\e}|_{tv}(w+su)\big|_{euc} \leq |s||u|_{euc} \frac{1}{at}(e^{at}-1) \leq 3|s||u|_{euc},
\end{align*}
which shows local Lipschitz continuity of the first derivative at the point $\gamma(t)$. The second derivative of the exponential map exists as $g_\e$ is smooth and by the Lipschitz continuity, it is bounded. 
Setting $su = -w$, and recalling that $J_0(t)=0$ for all $t$, we get that
\begin{align*}
    \big|D\exp_p^{g_\e}|_{tv}(w)\big|_{euc} \leq 3 |w|_{euc}.
\end{align*}
Hence the first differential is bounded at the point $tv$. Recall that $v$ is arbitrary, given that the geodesic tangent to $v$ and its variation are well defined for $t \in [0, \frac{1}{a(\e)}]$. Then the above holds for each $\Tilde{v} \in (0, \frac{1}{a(\e)}]v$  and \eqref{differentials_expo_bounds} holds for each $\Tilde{v}$ such that 
\begin{align*}
    0 < \norm{\Tilde{v}}_{g_\e} \leq \frac{a}{2\lambda}\min(i_\e(p), 1).
\end{align*}
By continuity of the differential and the second differential this also holds for $\norm{\Tilde{v}}_{g_\e} = 0$.
Since $p$ was arbitrary, taking Lemma \ref{injectivity_radius_bound} into account, we get the result. 
\end{proof}

\begin{lemma}\label{apx_with_rot_sym_functions}
    Let $K, K' \subset \R^n$ be compact sets such that $K \subset K'$. Let furthermore $\mathcal{F}$ be a family of closed balls in $\R^n$ such  that for every point $x \in K$, there exists a radius $R_x > 0$ such that for all positive $R \leq R_x$, it holds $\overline{B_R(x)} \in \mathcal{F}$. Let $\p \in C_c(K', [0, \infty))$ and $\e > 0$. Then there exists a finite family of functions $\{\chi_p\}_{p=1}^q$ such that the following holds:
    \begin{itemize}
        \item[$\mathrm{(i)}$] For every $p=1, \ldots, q$, there exists an $x \in K$ and a positive $R \leq R_x$ such that $\chi_p \in C_c(B_R(x), [0, \infty))$ and $\chi_p$ is rotationally symmetric centered at $x$.
        \item[$\mathrm{(ii)}$] $\sum_{p=1}^q \chi_p \leq \p$ and $\norm{\sum_{p=1}^q \chi_p - \p}_{C^0(K)} \leq \e$.
    \end{itemize}
\end{lemma}
\begin{proof}
    We first recall that, by Besicovitch's covering theorem, there exists a constant $c_n\geq 1$ such that for any family $\mathcal{G}$ of closed balls in $\R^n$, such that the set $A$ of their centers is bounded, there exist $c_n$ countable subfamilies $(\mathcal{G}_h)_{h=1}^{c_n}$ such that for each $h=1, \ldots, c_n$, any two different balls in $\mathcal{G}_h$ are disjoint and $A \subset \bigcup_{h=1}^{c_n}\bigcup_{B \in \mathcal{G}_h} B$. 
    \smallskip
    
    \textbf{Claim.} For every function $\p \in C_c(K', [0, \infty))$, there exists a finite family of functions $\{\psi_k\}_{k=1}^m$ such that for every $k=1, \ldots, m$ the following holds:
    \begin{enumerate}
        \item There exist $x \in K$ and a positive $R \leq R_x$ such that $\psi_k \in C_c(B_R(x), [0, \infty))$.
        \item $\psi_k$ is rotationally symmetric centered at $x$.
        \item The following estimate holds:\begin{align}\label{apx_step_rot_sym}
        \sum_{k=1}^m \psi_k \leq \p, \quad \mathrm{and} \quad \norm{\sum_{k=1}^m \psi_k - \p}_{C^0(K)} \leq \Big(1-\frac{1}{2c_n}\Big)\norm{\p}_{C^0(K)}.
    \end{align}
    \end{enumerate}

    \textit{Proof of the claim.}
    If $\norm{\p}_{C^0(K)} = 0$, we do not need any function to approximate $\p$, so we set $m=0$.
    Otherwise, using that $\p$ is uniformly continuous, we can find  $\delta > 0$ such that for all $x, y \in K'$ with $|x-y| \leq 4\delta$, we have that $|\p(x)-\p(y)| \leq \norm{\p}_{C^0(K)} \cdot \frac{1}{2c_n}$.
    For each $x \in K$ we define $\rho_x = \min(R_x, \delta)$. Now $\bigcup_{x \in K} B_{\frac{\rho_x}{2}}(x)$ is an open cover of $K$, so by compactness of $K$, there exists a finite subcover $\bigcup_{k=1}^mB_{\frac{\rho_k}{2}}(x_k)$ of $K$, where we write $\rho_k = \rho_{x_k}$ to keep notation short.
    Now we apply Besicovitch's covering theorem to find $c_n$ finite families $\mathcal{F}_h \subset \{\overline{B_{\frac{\rho_k}{2}}(x_k)}\}_k$ such that $K \subset \bigcup_{h=1}^{c_n}\bigcup_{B \in \mathcal{F}_h} B$ and the balls in each family are pairwise disjoint. 
    For each $h$ denote by $I_h$ the set of indices $k$ such that $\overline{B_{\frac{\rho_k}{2}}(x_k)} \in \mathcal{F}_h$.
    Fix such an $h$. Note that $I_h$ is finite, hence for each $k \in I_h$ we can find a radius $\frac{\rho_k}{2} < r_k \leq \rho_k$ such that for $k, k' \in I_h$, $k \neq k'$, we have that $B_{r_{k'}}(x_{k'}) \cap B_{r_k}(x_k) = \emptyset$. 
    For each $k = 1, \ldots, m$, define a function $\eta_k \in C_c(B_{r_k}(x_k), [0,1])$ such that $\eta_k$ is rotationally symmetric (with center $x_k$) and $\eta_k = 1$ on $\overline{B_{\frac{\rho_k}{2}}(x_k)}$. Define 
    \begin{align*}
        \theta_k := \min_{y \in \overline{B_{r_k}(x_k)}}\p(y) \in [0, \infty).
    \end{align*}  
    Now define $\psi_k \in C_c(B_{r_k}(x_k), [0, \infty))$ via 
    \begin{align*}
        \psi_k(x) := \frac{\theta_k}{c_n}\eta_k(x).
    \end{align*}
    We now claim that \eqref{apx_step_rot_sym} holds.
    Pick a point $x \in K'$. Let $L_x := \{l : x \in B_{r_l}(x_l)\}$. For each $h =1, \ldots, c_n$, we have that $|I_h \cap L_x| \leq 1$, hence $|L_x| \leq c_n$. By definition, we have that $\theta_l \leq \p(x)$ for each $l \in L_x$, hence
    \begin{align*}
        \sum_{k=1}^m\psi_k(x) = \sum_{l \in L_x}\psi_l(x) = \sum_{l \in L_x} \frac{\theta_l}{c_n}\eta_l(x) \leq \p(x).
    \end{align*}
    Now let $x \in K$. We have that for each $l \in L_x$ and $y \in \overline{B_{r_l}(x_l)}$ it holds $|x-y| \leq 2r_l \leq 2\rho_l \leq 2\delta$. Hence, for each $l \in L_x$, we have that $|\p(x)-\theta_l| \leq \norm{\p}_{C^0(K)}\cdot \frac{1}{2c_n}$. By the construction of our cover, we have that there exists at least one $\lambda \in L_x$ such that $x \in B_{\frac{\rho_\lambda}{2}}(x_\lambda)$. Hence
    \begin{align*}
       \sum_{k=1}^m\psi_k(x) \geq \frac{\theta_\lambda}{c_n}\eta_\lambda(x) = \frac{\theta_\lambda}{c_n}.
    \end{align*}
    Thus,
    \begin{align*}
        0 \leq \p(x) - \sum_{k=1}^m\psi_k(x) &\leq \p(x) - \frac{\theta_\lambda}{c_n} \leq \p(x) - \frac{\p(x)}{c_n} + \left|\frac{\p(x)}{c_n} - \frac{\theta_\lambda}{c_n}\right| = \Big(1-\frac{1}{c_n}\Big)\p(x) + \frac{1}{c_n}|\p(x)-\theta_\lambda| \\
        & \leq \Big(1-\frac{1}{c_n}\Big)\norm{\p}_{C^0(K)} + \frac{1}{2c^2_n}\norm{\p}_{C^0(K)} \leq \Big(1-\frac{1}{2c_n}\Big)\norm{\p}_{C^0(K)}.
    \end{align*}
  This proves the claim. \\
  Now given a function $\p \in C_c(K', [0, \infty))$, we define a sequence of functions $(\p_j)_{j=0}^\infty \subset C_c(K', [0, \infty))$ as follows: $\p_0 := \p$. Given $\varphi_j$, we apply the claim to find $m_j \geq 0$ and functions $\psi^{(j)}_k$, for $k =1, \ldots, m_j$ such that 
  \begin{itemize}
        \item $\psi^{(j)}_k \in C_c(B_{R_{x_{j,k}}(x_{j,k})}, [0, \infty))$ for some $ x_{j,k} \in K$.
        \item $\psi^{(j)}_k$ is rotationally symmetric centered at $x_{j, k}$.
        \item The following estimate holds:
        \begin{align*}
        \sum_{k=1}^{m_j} \psi^{(j)}_k \leq \p_j, \quad \mathrm{and} \quad \norm{\sum_{k=1}^{m_j} \psi^{(j)}_k - \p_j}_{C^0(K)} \leq \Big(1-\frac{1}{2c_n}\Big)\norm{\p_j}_{C^0(K)}.
    \end{align*}
   \end{itemize}
   Then set $\p_{j+1}:= \p_j - \sum_{k=1}^{m_j} \psi^{(j)}_k \geq 0$.
Inductively, we get that
   \begin{align}\label{multiple_iteration_process}
       \norm{\p_j}_{C^0(K)} \leq \Big(1-\frac{1}{2c_n}\Big)^j\norm{\p}_{C^0(K)}.
   \end{align}
   Now we can choose $J \geq \frac{\ln \e - \ln \norm{\p}_{C^0(K)}}{\ln{(1-\frac{1}{2c_n})}}$. We choose the family of functions $\{\psi_{k}^{(j)}: k=1, \ldots, m_j, j=0, \ldots J\}$ Then 
   \begin{align*}
       \p - \sum_{j=0}^J\sum_{k=1}^{m_j} \psi^{(j)}_k = \p_0 - \sum_{j=0}^J (\p_j-\p_{j+1}) = \p_{J+1} \geq 0.
   \end{align*}
  By our choice of $J$ and \eqref{multiple_iteration_process}, we get that
  \begin{align*}
      \norm{\p - \sum_{j=0}^J\sum_{k=1}^{m_j}\psi^{(j)}_k}_{C^0(K)} = \norm{\p_{J+1}}_{C^0(K)} \leq \e. 
  \end{align*}
  This finishes the proof.
\end{proof}
\begin{lemma}\label{lebeague_point_with_cts}
    Let $f \in L^1_{\rm{loc}}(\R^n)$ and $h \in C^0(\R^n)$. Then for each $p \in \R^n$, we have that if $p$ is a Lebesgue point of $f$, then $p$ is also a Lebesgue point of $f\cdot h$. 
\end{lemma}

We are now able to state and prove the second main result of the paper, namely that the  $\mathsf{CD}^*(K, N)$ condition implies that the distributional $N$-Bakry-\'Emery Ricci tensor  is bounded below by $K$,  on a smooth manifold endowed with a continuous Riemannian metric with $L^2_{\rm{loc}}$-Christoffel symbols and a $C^{0}\cap W^{1,2}_{\rm{loc}}$-weight on the volume measure.

\smallskip

\begin{theorem}\label{cdknforlipschitzmetric}
    Let $M$ be a smooth manifold, $g \in C^{0}$ a Riemannian metric that admits $L^2_{\rm{loc}}$-Christoffel symbols and $h \in C^{0}\cap W^{1,2}_{\rm{loc}}(M)$ a positive function, $V := -2 \log h \in C^{0}\cap W^{1,2}_{\rm{loc}}(M)$. Define the measure $\mu$ via $\mu:= e^{-V}\di\vol_g$. Let $K\in \R$ and $N\in [n,\infty)$. If $(M, \sfd_g, \mu)$ is a $\mathsf{CD}^*(K, N)$-space, then 
    \begin{align*}
        \Ric_{\mu, N} = \Ric_{\mu, \infty}- \frac{1}{N-n}\nabla V \otimes \nabla V \geq K g \quad \mathrm{in}\ \mathcal{D}'\mathcal{T}^0_2.
    \end{align*}
\end{theorem}
\begin{proof}
We first notice that by Lemma \ref{implications_of_cd_conditions} and Corollary \ref{summary_chapter_4}, $(M, \sfd_g, \mu)$ is an $\rcd(K, \infty)$-space, so by Theorem \ref{ricciasmeasure}, we have that $\Ric_{\mu, \infty} \in \mathcal{D}'\mathcal{T}^0_2$ can be expressed via Radon measures $(\nu_{ij})_{i,j}$ as in \eqref{ricci_as_measure}. 
Let $U \subset W \subset M$ be open such that $\overline{W}$ is compact, $\overline{U} \subset W$ and $W$ lies in one coordinate patch. It is enough to prove the statement for $M = U$, as we can cover the manifold with sets like $U$. 
Let $\rho \in C_c^\infty(B^{euc}_1(0), [0,\infty))$ be a rotationally symmetric standard mollifier and $\rho_\delta(x):= \frac{1}{\delta^n}\rho(\frac{x}{\delta})$, $\delta > 0$. For $0 < \delta \leq \delta_0 < \dist(\partial W, U)$, we define $g_\delta = \rho_\delta * g$.  We assume that for all these $\delta$, $g_\delta$ is a Riemannian metric on $U$ and that $\max(\norm{g_\delta}_{C^{0}(U)}, \norm{g_\delta^{-1}}_{C^{0}(U)}) < 2\max(\norm{g}_{C^{0}(U)}, \norm{g^{-1}}_{C^{0}(U)})$, as this is the case for $\delta_0>0$ small enough. Then $g_\delta \to g$ everywhere as $\delta \to 0$. 
Then, by the proof of Proposition \ref{length_space} there exists a $\zeta >0$ and a $\lambda \geq 1$ such that $g, g_\delta$ and $|\cdot|_{euc}$ are pairwise $\lambda$-equivalent on $U$ and $\sfd_g, \sfd_{g_\delta}$, and $|\cdot|_{euc}$ are pairwise $\lambda$-equivalent on $B^{euc}_\zeta(p)$ for all $p \in U$.  
 Fix a point $p \in U$ such that
 \begin{align}\label{convergence_properties_of_p}
     &p \ \mathrm{is \ a \ Lebesgue \ point \ of \ } Dg,\mathrm{\ and\ of \ } Dg \cdot Dg.
 \end{align}
 Then, as $\rho$ is rotationally symmetric, we get that 
 \begin{align}\label{derivative_converge}
     Dg_\delta(p) \to Dg(p).
 \end{align}
 Fix $\e > 0$. We can assume $p=0$. In order to compute local coordinates with a vanishing Levi-Cevita connection with respect to $g$ at the point $p$, we define the map $\psi: U \to \Tilde{U} \subset \R^n, x \mapsto y$ via 
\begin{align*}
     y^k = \psi^k(x) := 0 + \alpha^k_i x^i + \beta^k_{ij}x^{i}x^j,
\end{align*}
where $\alpha^k_i, \beta^k_{ij}= \beta^k_{ji}\in \R$ are constants.
Then
\begin{align*}
    &(D\psi)_{ki}(x) = \partial_{x^i}y^k = \alpha^k_i + 2\beta^k_{ij}x^j,  \\
    & \partial_{x^l} (D\psi)_{ki}(x) = 2\beta^k_{il}, \quad \mathrm{and}\\
    &(\psi_*g)_{ij} = ((D\psi)^{-T}g(D\psi)^{-1})_{ij}.
\end{align*}
Moreover, it holds
\begin{align*}
    (\partial_{x^l} (D\psi)^{-1})_{ij}(0) &= - ((D\psi)^{-1}\partial_{x^l} D\psi (D\psi)^{-1})_{ij}(0) =  - (D\psi)^{-1}_{it}\partial_{x^l} D\psi_{ts} (D\psi)^{-1}_{sj}(0) \\
    &= -2(D\psi)^{-1}_{it}\beta^t_{sl} (D\psi)^{-1}_{sj}(0).
\end{align*}
We assume $\alpha$ to be invertible and define $\alpha^{-1} =: \sigma = (\sigma_{ij})_{ij} \in \R^{n \times n}$. In order to be ``normal coordinates'' at $p=0$, we need that at the point $\psi(0) = 0 \in \Tilde{U}$, it holds
\begin{align}\label{first_cond_on_first_differential}
    (\psi_*g)_{ij}(0) = (\alpha^{-T}g\alpha^{-1})_{ij}(0) = {\rm Id}_{n}.
\end{align}
We want the first derivative of the metric with respect to the $y$-coordinates to vanish at $p=0$, hence we compute:
\begin{align}\label{system_of_equations_expo_a}
    0&= \partial_{y^k}(\psi_*g)_{ij}(0)= \partial_{x^l}((D\psi^{-1})_{ri}(D\psi^{-1})_{qj}g_{rq})\frac{\partial x^l}{\partial y^k}(0) \nonumber \\
    &= (\partial_{x^l}(D\psi^{-1})_{ri}(D\psi^{-1})_{qj}g_{rq}+(D\psi^{-1})_{ri}\partial_{x^l}(D\psi^{-1})_{qj}g_{rq}+(D\psi^{-1})_{ri}(D\psi^{-1})_{qj}\partial_{x^l}g_{rq})\frac{\partial x^l}{\partial y^k}(0) \nonumber \\
    &= (-2\sigma_{rt}\beta^t_{sl} \sigma_{si}\sigma_{qj}g_{rq}-2\sigma_{ri}\sigma_{qt}\beta^t_{sl}\sigma_{sj}g_{rq}+\sigma_{ri}\sigma_{qj}\partial_{x^l}g_{rq})\sigma_{lk}.
\end{align}
This gives $\frac{1}{2}n^2(n+1)$ equations and $\frac{1}{2}n^2(n+1)$ variables $\beta^k_{ij}$. To keep notation short, we rewrite this to 
\begin{align}\label{second_part_of_equation_sys}
    \zeta_{ijk}^m(g, \sigma)\beta^k_{ij} = \xi^m(Dg, \sigma), \quad m= 1, \ldots, \frac{1}{2}n^2(n+1),
\end{align}
for some polynomials $\zeta_{ijk}^m : (\R^{n\times n})^2 \to \R$ and $\xi^m: \R^{n \times n\times n} \times\R^{n\times n}$. 
In order to find a solution of this system of equations, consider $g_\delta$, $\delta \in (0, \delta_0)$.
By using the exponential map, we can find a map $\psi_\delta$ defined as
\begin{align*}
    \psi_\delta = (\exp_p^{g_\delta})^{-1} : \Omega_\delta \to B_{r_\delta}(0) \subset T_pU \cong \R^n, x \mapsto y_\delta,
\end{align*}
where $\Omega_\delta \subset U$ is an open neighbourhood of $0$ and $r_\delta > 0$.
Define the distance function $d^p_\delta: \Omega_\delta \to [0, \infty), y \mapsto \sfd_{g_\delta}(p,z) = \big(\sum_{i=1}^n (y_\delta^i)^2\big)^\frac{1}{2}$. It is known that $(\psi_\delta)_*(g_\delta)_{ij} = \delta_{ij}+ O_\delta(({d^p_\delta})^2)$. More precisely, this means that there exists a constant $C_{g_\delta} >0$ such that for all $x \in \Omega_\delta$,
\begin{align}
     \norm{(\psi_\delta)_*g_\delta(x)- {\rm Id}_{n}}_{\R^{n \times n}} \leq C_{g_\delta} ({d^p_\delta})^2(x).
\end{align}
Note that $|Dg_{\delta}|\leq C\lambda\delta^{-1}$ for some $C>0$.  
Take $\Tilde{\Omega}_\delta := \Omega_\delta \cap \{x: |x|_{euc} = |x-p|_{euc} \leq \frac{\delta}{2nC}, \ (d^p_\delta(x))^2 < \frac{1}{2nC_{g_\delta}}\frac{1}{4\lambda^2}\}$.
Then in $\Tilde{\Omega}_\delta$, we get that 
 \begin{align}
    \frac{1}{2\lambda}{\rm Id}_{n} \leq g_\delta \leq 2\lambda {\rm Id}_{n}, \quad \mathrm{and}\quad  \frac{1}{2}{\rm Id}_{n} \leq (\psi_\delta)_*g_\delta \leq 2{\rm Id}_{n}.
 \end{align}
Using that $\norm{v}_{g_\delta} = \norm{(\psi_\delta)_*v}_{(\psi_\delta)_*g_\delta}$, we get that the transition map $\Psi_\delta:= D\psi_\delta: T \Tilde{\Omega}_\delta \to T\psi_\delta(\Tilde{\Omega}_\delta)$ is bounded from above and from below in the Euclidean norm, meaning that for all $0 \neq v \in T_q\Tilde{\Omega}_\delta$, we have that 
\begin{align}\label{bound_on_transition_fn}
     \frac{1}{4\lambda} = \frac{\frac{1}{2}}{2\lambda} \leq \frac{\norm{\Psi_\delta(v)}_{euc}}{\norm{v}_{euc}} \leq \frac{2}{\frac{1}{2\lambda}} = 4\lambda.
\end{align}
By the inverse function theorem, we know that on $\Tilde{\Omega}_\delta$
\begin{align*}
    D\psi_\delta = (D (\exp^{g_\delta}_p))^{-1}, \quad \mathrm{hence} \quad \partial_iD\psi_\delta = -D \psi_\delta \partial_i(D(\exp^{g_\delta}_p))D \psi_\delta.
\end{align*}
Hence, using \eqref{bound_on_transition_fn} and \eqref{differentials_expo_bounds}, we get that  
\begin{align}\label{bounds_on_good_expo_map}
    |(D\psi_\delta)^{-1}|_{euc}(0) \leq 4\lambda \quad \mathrm{and} \quad    |D^2\psi_\delta|_{euc}(0) \leq 12\lambda^2.
\end{align}
Now for all $\delta>0$, we have that 
\begin{align*}
     \zeta_{ijk}^m(g_\delta, (D\psi_\delta)^{-1})\frac{1}{2}\partial_{j}(D\psi_\delta)_{ik}= \xi^m(Dg_\delta, (D\psi_\delta)^{-1}), \quad m= 1, \ldots, \frac{1}{2}n^2(n+1).
\end{align*}
With \eqref{bound_on_transition_fn} and \eqref{bounds_on_good_expo_map}, we get that there exists a convergent subsequence $\delta' \to 0$ such that $(D\psi_{\delta'}(0), D^2\psi_{\delta'}(0)) \to (D\psi(0), D^2\psi(0))$ and  
\begin{align}\label{bounds_on_g_expo_map}
    &|D\psi|_{euc}(0) \leq 4\lambda, \ |D\psi^{-1}|_{euc}(0) \leq 4\lambda, \ \mathrm{and} \ |D^2\psi|_{euc}(0) \leq 12\lambda^2.
\end{align}
By \eqref{convergence_properties_of_p}, \eqref{derivative_converge}, and the continuity of $g$, it follows that $(D\psi(0), D^2\psi(0)) $ satisfies \eqref{first_cond_on_first_differential} and \eqref{system_of_equations_expo_a}.
Now we have found coefficients for our map $\psi:U \to \Tilde{U}$ and we notice that for now the map is defined everywhere on $U$. 
By \eqref{first_cond_on_first_differential}, $D\psi(0)$ must be invertible, so by the inverse function theorem we can shrink $U$ to a neighbourhood $U'$ of $p$ such that $\psi$ is a diffeomorphism on $U'$.
Note that $D^2\psi$ is constant, so 
\begin{align}\label{uniform_nound_d_2_psi}
   |D^2\psi|_{euc}(x) \leq 12\lambda^2,  
\end{align}
for all $x\in U'$.
As $\psi_*g ={\rm Id}_{n}$, we get that the transition map $\Psi|_0:= D\psi|_0: T_0 U' \to T_0\psi(U')$ is bounded, meaning that for all $0 \neq v \in T_0\psi(U')$, we have that 
\begin{align}
     \frac{1}{\lambda} \leq \frac{\norm{\Psi(v)}_{euc}}{\norm{v}_{euc}} \leq \lambda.
\end{align}
Using \eqref{uniform_nound_d_2_psi}, we can shrink the domain $U'$ and find an open $\Omega \subset U'$ such that $p \in \Omega$ and the transition map $\Psi:= D\psi: T \Omega \to T \psi(\Omega)$ satisfies
\begin{align}\label{dpsi_bdd}
     \frac{1}{2\lambda} \leq \frac{\norm{\Psi(v)}_{euc}}{\norm{v}_{euc}} \leq 2\lambda,
\end{align}
for each $v \in T_q \Omega$, $q \in \Omega$.
As $p$ satisfies \eqref{convergence_properties_of_p}, we can use Lemma \ref{lebeague_point_with_cts}, to get that for each $f \in C^{2}_c(M)$, 
\begin{align}\label{p_nice_lebesgue_point}
    p\ \mathrm{is\ a\ Lebesgue\ point\ of\ }  x \mapsto |\nabla_{g}^2f(x)|^2_{HS}, \ \mathrm{and\ } x\mapsto (\Delta_gf(x))^2.
\end{align}
Here $\Delta_g$ denotes the Laplacian with respect to the volume measure induced by $g$. As, by the above construction, $y = \psi(x)$ defines normal coordinates at $p$ with respect to $g$, we have that
\begin{align*}
    &((\psi)_*(\nabla_{g}^2f)(0))_{ij} = \partial^2_{y^i, y^j}f(0), \quad \mathrm{and} \\
    & \Delta_{g} f = \sum_{i} \partial^2_{y^i, y^i}f(0) = \sum_i ((\psi)_*(\nabla_{g}^2f)(p))_{ii}.
\end{align*}
Now fix a smooth vector field $X$ on $U$ and denote $B := \psi_*(X(p)) \in \R^n$. Let $f: \Omega \to  \R$ be defined by $\psi_*f := \sum_i B_i y^{i} + \frac{a}{2\sqrt{n}}(y^i)^2$ for some $a \in \R$. Then
\begin{align*}
    \psi_*\nabla_{g}f(0) = B,  \quad \psi_*\nabla^2_{g}f(0) = \frac{a}{\sqrt{n}}{\rm Id}_{n}, \quad \mathrm{and} \quad  \psi_*\Delta_{g}f(0) = \sqrt{n}a.
\end{align*}
Note that $f$ is smooth because $\psi$ is smooth. 
Writing $h^2 = e^{-V}$ for a $C^{0}\cap W^{1,2}_{\rm{loc}}$-function $V$, we get: 
\begin{align*}
    &\Delta_{g,\mu} f = \Delta_g f + \langle \nabla_g V, \nabla_g f \rangle_g.
\end{align*}
Hence, 
\begin{align*}
    |\nabla^2_{g} f|^2_{HS, g}(p) - \frac{1}{N}(\Delta_{g, \mu} f)^2(p) = \frac{N-n}{N}a^2 - \frac{2\sqrt{n}}{N}a\langle X, \nabla_{g} V \rangle_{g}(p) - \frac{1}{N} \langle X, \nabla_{g} V \rangle^2_{g}(p).
\end{align*}
The choice $a= \frac{\sqrt{n}\langle X, \nabla_{g} V \rangle_{g}(p)}{N-n}$ minimises the right hand side and yields
\begin{align}\label{minimized_equality}
    |\nabla^2_{g} f|_{HS, g}^2(p) - \frac{1}{N}(\Delta_{g, \mu} f)^2(p) = -\frac{1}{N-n}\langle X, \nabla_{g} V \rangle^2_{g}(p).
\end{align}
Now there is a $\tau_0 > 0$, $\tau \leq \min(\e, 1)$ such that ${B_{\tau_0}(p)} := {B^{euc}_{\tau_0}(p)} \subset {\Omega}$ and for all $z \in {B_{\tau_0}(p)}$, we have 
\begin{align}\label{uniform_continuity_of_gradient_and_x}
    |\nabla_{g}f(z)-X(p)|_{euc} \leq \e \quad \mathrm{and} \quad |X(z)-X(p)|_{euc} \leq \e. 
\end{align}
Finally note $f = \psi^*\psi_* f$, so we can bound $\norm{f}_{W^{2, \infty}}$  in a neighbourhood of $p$ in terms of $|\psi|, |D\psi|, |D\psi^{-1}|, |D^2\psi|, a$, and $B$.
Hence, by \eqref{uniform_nound_d_2_psi} and \eqref{dpsi_bdd}, we get that
\begin{align}\label{bound_on_f_c2}
    \norm{f}_{W^{2, \infty}(B_{\tau_0}(p))} \leq C(X, V, \lambda, n, N).
\end{align}
Now, by \eqref{smoothbecond}, we have
\begin{align*}
     \sum_{i,j}(\nabla_g f)_i(\nabla_g f)_j \nu_{ij} + (|\nabla_g^2 f|^2_{HS, g} - K|\nabla_g f|_g^2 - \frac{1}{N}(\Delta_{g,\mu} f)^2)\mu \geq 0.  
\end{align*}
Recall that $\mu = e^{-V}\sqrt{|g|}\mathcal{L}^n$, where $V, g \in C^{0}(U)$. Then by the triangle inequality, \eqref{p_nice_lebesgue_point} together with Lemma \ref{lebeague_point_with_cts}, and by \eqref{uniform_continuity_of_gradient_and_x}, we get that there exists a $\tau\in(0,\tau_0)$ such that for all $\tau' \in (0, \tau)$, it holds
\begin{align*}
    \Bigg|&\Big( \sum_{i,j} (\nabla_{g} f)_i(p)(\nabla_{g} f)_j(p) \nu_{ij} + (|\nabla_{g}^2 f|^2_{HS, g}(p) - K|\nabla_{g} f|_{g}^2(p) - \frac{1}{N}(\Delta_{{g},\mu} f)^2(p))\mu \Big) \\
    &\quad  -\Big( \sum_{i,j} (\nabla_{g} f)_i(\nabla_{g} f)_j \nu_{ij} + (|\nabla_{g}^2 f|^2_{HS, g} - K|\nabla_{g} f|_{g}^2 - \frac{1}{N}(\Delta_{{g},\mu} f)^2)\mu \Big)\Bigg|_{TV}(B_{\tau'}(p)) \\
    & \quad \qquad \leq  C(n, U, g, V, X, N, K)\big( \sum_{i,j}|\nu_{ij}|_{TV}+\mathcal{L}^n\big)(B_{\tau'}(p))\cdot\e.
\end{align*}
 Similarly, using again \eqref{uniform_continuity_of_gradient_and_x}, we get that 
\begin{align*}
    \Bigg| \sum_{i,j}& (\nabla_{g} f)_i(p)(\nabla_{g} f)_j(p)\nu_{ij} - ( K|\nabla_{g} f|_{g}^2(p) + \frac{1}{N-n}(\langle \nabla_{g}f, \nabla_{g}V\rangle^2_{g})(p))\mu \\
    & - \sum_{i,j} X^iX^j \nu_{ij} +  (K|X|_g^2 + \frac{1}{N-n}(\langle X, \nabla_{g}V\rangle^2_{g}))\mu\Bigg|_{TV}(B_{\tau'}(p)) \\
    &\qquad \quad  \leq C(n, U, g, V, X, N, K)\e\big( \sum_{i,j}|\nu_{ij}|_{TV}+\mathcal{L}^n\big)(B_{\tau'}(p)).
\end{align*}
Together with \eqref{minimized_equality}, this yields
\begin{align*}
    \big(\sum_{i,j} X^iX^j &\nu_{ij} -  ( K|X|_g^2 + \frac{1}{N-n}(\langle X, \nabla_{g}V\rangle^2_{g}))\mu\big)(B_{\tau'}(p)) \nonumber \\
    &\geq -C(n, U, g, V, X, N, K)\e\big( \sum_{i,j}|\nu_{ij}|_{TV}+\mathcal{L}^n\big)(B_{\tau'}(p)).
\end{align*}
We infer that for $\chi \in C_c(B_\tau(p), [0, \infty)])$ with $\chi$ rotationally symmetric around $p$, it holds
\begin{align}\label{integral_property_rot_symm}
    \int_{B_\tau(p)} \chi \, &\di(\sum_{i,j} X^iX^j \nu_{ij} -  ( K|X|_g^2 + \frac{1}{N-n}(\langle X, \nabla_{g}V\rangle^2_{g}))\mu)  \nonumber \\
    &\geq -C(n, U, g, V, X, N, K)\e \int_{B_\tau(p)} \chi \, \di\big( \sum_{i,j}|\nu_{ij}|_{TV}+\mathcal{L}^n\big).
\end{align}
As $p$ was chosen arbitrarily among all points satisfying \eqref{convergence_properties_of_p}, we have that for all $p \in U$, satisfying \eqref{convergence_properties_of_p}, and each $\e>0$, there exists a $\tau_{p, \e} > 0$ such that \eqref{integral_property_rot_symm} holds for  $\chi \in C_c(B^{euc}_{\tau_{p, \e}}(p), [0, \infty))$ such that $\chi$ is rotationally symmetric around $p$.
Now suppose there exists a function $\phi \in C_c(U, [0, \infty))$ such that 
\begin{align*}
    \langle \Ric_{\mu, \infty}(X,X) - K|X|^2_g\mu- \frac{1}{N-n}  (\nabla V \otimes \nabla V)(X,X)\mu, \phi \; dx^1 \ldots dx^n \rangle_{\mathcal{D}', \mathcal{D}} = \kappa < 0. 
\end{align*}
We can assume that $\norm{\phi}_{sup} = 1$. Choose 
\begin{align}\label{choice_of_eps_last_proof}
    0 < \e \leq \frac{-\kappa}{8C\cdot (\mathcal{L}^n(U) + \sum_{i,j}|\nu_{ij}|_{TV}(U))},
\end{align}
where $C \geq 1$ is as in \eqref{integral_property_rot_symm}. 
Write 
\begin{align*}
    \Ric_{\mu, \infty}(X,X) := \mathfrak{r}_X^s + \mathfrak{r}_X^{a},
\end{align*}
where $\mathfrak{r}_X^s$ denotes the singular part and $\mathfrak{r}_X^{a}$ denotes the absolutely continuous part with respect to the Lebesgue measure. As $\Ric_{\mu, \infty}(X,X) - K|X|^2_g\mu \geq 0$, and $K|X|^2_g\mu$ is absolutely continuous with respect to the Lebesgue measure, we know that 
\begin{align}\label{singular_ricci_positive}
    \mathfrak{r}_X^s \geq 0.
\end{align}
Moreover, there exist two Borel sets $U_s, U_a \subset U$ such that $U_s \cap U_a = \emptyset$, $U_s \cup U_a = U$, $\mathfrak{r}^s_X(U_a) = 0$, and $\mathcal{L}^n(U_s) = 0$. 
Define 
\begin{align*}
    \mathcal{N}:= \{p \in U: \ p\ \mathrm{does \ not \ satisfy\  \eqref{convergence_properties_of_p}}\}.
\end{align*} 
Then $\mathcal{L}^n(\mathcal{N}) = 0$ and hence $\mathcal{L}^n(\mathcal{N} \cup U_s) = 0.$
By the outer regularity of Borel measures, there exists an open set $O \subset U$ such that 
\begin{align}\label{open_set_of_small_measure_containing_problems}
    &(\mathcal{N} \cup U_s) \subset O, \ \mathrm{and} \nonumber \\
    & |\mathfrak{r}^a_X|_{TV}(O) + (|K||X|^2_g + \frac{1}{N-n}  (\nabla V \otimes \nabla V)(X,X))\mu(O) \leq \frac{\e}{4}.
\end{align}
Denote $\supp\, \phi = K'$, $K:= K' \setminus O$ and note that both $K$ and $K'$ are compact. 
Moreover, denote $\mathcal{F}:= \{\overline{B_{\tau}(p)}: \ p \in K, 0 < \tau \leq \frac{1}{2}\tau_{p, \e}\}$. 
We can now apply Lemma \ref{apx_with_rot_sym_functions} to find an $m \in \N$ and functions $\chi_k \in C_c(B_{\tau_k}(p_k), [0, \infty))$ for $k=1, \ldots, m$, $p_k \in K$, $\chi_k$ is rotationally symmetric centered at $p_k$, $\tau_k \leq \frac{\tau_{p_k, \e}}{2}$, such that $\sum_{k=1}^m \chi_k \leq \phi$ and ${\norm{\sum_{k=1}^m \chi_k - \phi}_{C^0(K)} \leq \e}$. Denote $\psi:= \sum_{k=1}^m \chi_k$.
Define the signed Radon measure $\mathfrak{s}$ on $U$ as
\begin{align*}
    \mathfrak{s} = -K|X|^2_g\mu - \frac{1}{N-n}(\nabla V \otimes \nabla V)(X, X)\mu.  
\end{align*}
By \eqref{choice_of_eps_last_proof}, \eqref{open_set_of_small_measure_containing_problems}, and the bound ${\norm{\phi - \psi}_{C^0(K)} \leq \e}$, we have that 
\begin{align*}
    \Bigg|\int_U (\phi- \psi) \, \di(\mathfrak{s} + \mathfrak{r}^a_X) \Bigg| &=\Bigg|\int_{K'} (\phi- \psi) \, \di(\mathfrak{s} + \mathfrak{r}^a_X)\Bigg| \\
    & \leq  \Bigg|\int_{K} (\phi- \psi) \, \di(\mathfrak{s} + \mathfrak{r}^a_X)\Bigg| + \Bigg|\int_{O} (\phi- \psi) \, \di(\mathfrak{s} + \mathfrak{r}^a_X)\Bigg| \\
    & \leq \e \cdot (|\mathfrak{s}|_{TV} + |\mathfrak{r}^{a}_X|_{TV})(U) + \norm{\phi}_{sup}(|\mathfrak{r}^{a}_X|_{TV} + |\mathfrak{s}|_{TV})(O) \\
    &\leq -\frac{\kappa}{8} + \frac{\e}{4} \leq -\frac{\kappa}{2}.
\end{align*}
Then, using \eqref{singular_ricci_positive} and that $\phi \geq \psi$, we get that
\begin{align*}
    \int_U \phi \, \di(\mathfrak{s} + \mathfrak{r}^s_X + \mathfrak{r}^a_X) &= \int_U \psi \, \di(\mathfrak{s} + \mathfrak{r}^s_X + \mathfrak{r}^a_X) + \int_U (\phi- \psi) \, \di(\mathfrak{s} + \mathfrak{r}^a_X) + \int_U (\phi- \psi) \, \di(\mathfrak{r}^s_X) \\
    &\geq \int_U \psi \, \di(\mathfrak{s} +\mathfrak{r}^s_X + \mathfrak{r}^a_X) + \frac{\kappa}{2}.
\end{align*}
Finally, using \eqref{integral_property_rot_symm}, and $\norm{\psi}_{sup} \leq \norm{\phi}_{sup} \leq 1$, we get that
\begin{align*}
 \int_U & \psi \, \di (\mathfrak{s} +\mathfrak{r}^s_X + \mathfrak{r}^a_X) \\
    &= \sum_{k=1}^m \left( \sum_{i,j} \int_{B_{\tau_k}(p_k)}X^iX^j \chi_k\, \di\nu_{ij} -  \int_{B_{\tau_k}(p_k)} \Big( K|X|_g^2 + \frac{1}{N-n}(\langle X, \nabla_{g}V\rangle^2_{g}) \Big)\chi_k\,\di\mu \right) \\
    &  \geq -C(n, U, g, V, X, N, K)\e \sum_{k=1}^m \Big(\int_{B_{\tau_k}(p_k)} \sum_{i,j}\chi_k\,d|\nu_{ij}|_{TV}+ \int_{B_{\tau_k}(p_k)}\chi_k\,\di\mathcal{L}^n\Big) \\
    &\geq \frac{\kappa}{8(\mathcal{L}^n(U) + \sum_{i,j}|\nu_{ij}|_{TV}(U))}\Big(\int_{U} \sum_{i,j}\psi\,d|\nu_{ij}|_{TV}+ \int_U\psi\,\di\mathcal{L}^n\Big) \\
    &\geq \frac{\kappa}{8}.
\end{align*}
But then
\begin{align*}
    \kappa &= \left\langle \Ric_{\mu, \infty}(X,X) - K|X|^2_g\mu- \frac{1}{N-n}  (\nabla V \otimes \nabla V)(X,X)\mu, \phi dx^1 \ldots dx^n \right\rangle_{\mathcal{D}', \mathcal{D}}  \\
    &= \int_U \phi \, \di(\mathfrak{s} + \mathfrak{r}^s_X + \mathfrak{r}^a_X) \geq \frac{5\kappa}{8}.
\end{align*}
Since $\kappa < 0$, this is a contradiction. As $X$ was arbitrary, the proof is complete.
\end{proof}

\section{Distributional  Ricci curvature lower bounds imply $\rcd$}

In this section, we prove the reverse implication, namely from distributional to synthetic Ricci lower bounds, under the assumption of the volume growth condition \eqref{condition_on_measure}. Note that such a condition is necessary, as all $\mathsf{CD}(K,\infty)$ spaces satisfy it (see Remark \ref{rem:CDVol}). 
At the end of the section, we establish such volume growth condition for weighted manifolds with $C^1$-metrics and distributional Bakry-\'Emery $N$-Ricci curvature bounded below by $K$ for finite $N$.
\smallskip

\begin{theorem}\label{from_distri_torcd}
     Let $M$ be a smooth manifold, $g$ a continuous Riemannian metric with $L^2_{\rm{loc}}$-Christoffel symbols and let $h \in C^{0}(M)\cap W^{1,2}_{\rm{loc}}(M)$ be a positive function. Denote $V := -2 \log h \in C^{0}\cap W^{1,2}_{\rm{loc}}(M)$. Define the measure $\mu$ via $\di\mu:= e^{-V}\di\vol_g$. Let $N \in [n, \infty]$ and $K \in \R$. Assume that the metric measure space $(M, \sfd_g, \mu)$ satisfies \eqref{condition_on_measure}.
     
      If the distributional Bakry-\'Emery $N$-Ricci curvature tensor is bounded below by $K$, i.e.
     \begin{align*}
         \Ric_{\mu, N} \geq Kg \  \mathrm{in} \ {\mathcal{D}'}\mathcal{T}^0_2,
     \end{align*}
     then $(M, \sfd_g, \mu)$ satisfies the $\mathsf{RCD}^*(K, N)$-condition  if $N\in [n,\infty)$, or the $\rcd(K, \infty)$-condition if $N=\infty$.
\end{theorem}
\begin{proof}
    By Theorem \ref{cd_implies_be}, it suffices to prove that $(M, \sfd_g, \mu)$ satisfies the $\mathsf{BE}(K, N)$-condition. 
   For any smooth function $f \in C^\infty(M)$ and any non-negative, compactly supported test volume $\omega$, we have that 
    \begin{align}\label{given_inequality_distri_to_rcd}
         \langle&\Ric_{\mu, N}(\nabla f, \nabla f) - Kg(\nabla f, \nabla f), \omega \rangle_{\mathcal{D}', \mathcal{D}} \nonumber \\
         &= \left\langle\Ric_{\mu, \infty}(\nabla f, \nabla f) -\frac{1}{N-n}\langle\nabla V,\nabla f\rangle_g^2  - Kg(\nabla f, \nabla f), \omega \right\rangle_{\mathcal{D}', \mathcal{D}} \geq 0.
     \end{align}
    Let $n\geq 1$ be the dimension of $M$ and denote by $\Delta_g$ the Laplacian induced by the volume measure $\di\vol_g$. 
     \smallskip
     
     \textbf{Claim.} The following inequality holds at $\mu$-almost every point $p \in M$:
     \begin{align}\label{be_compared_to_tensor}
         |\nabla^2 f|^2_{HS} - \frac{1}{N}(\Delta_g f + \langle \nabla V, \nabla f \rangle)^2 \geq -\frac{1}{N-n} \langle \nabla V, \nabla f \rangle^2.
     \end{align}
     \textit{Proof of the claim.} Fix $p \in M$ such that $p$ is a Lebesgue point of $Dg$ and $(Dg)^2$. Denote $B:= |\langle \nabla V, \nabla f \rangle|(p)$. For $(e_i)_{i=1}^n$, a $g$-orthonormal basis at $T_pM$, we have that 
     \begin{align*}
         |\nabla^2 f|^2_{HS}(p)= \sum_{i, j=1}^n \nabla^2f(e_i, e_j)^2 \quad  \mathrm{and}\quad \Delta_gf(p) = \sum_{i=1}^n \nabla^2f(e_i, e_i).
     \end{align*}
     From the Cauchy-Schwartz inequality, we get that 
     \begin{align*}
         n|\nabla^2 f|^2_{HS}(p) \geq (\Delta_gf(p))^2.
     \end{align*}
Denote $a = |\nabla^2 f|_{HS}$. Then
\begin{align*}
   |\nabla^2 f|^2_{HS} &- \frac{1}{N}(\Delta_g f + \langle \nabla V, \nabla f \rangle)^2 + \frac{1}{N-n} \langle \nabla V, \nabla f \rangle^2 \\
   & \geq |\nabla^2 f|^2_{HS} - \frac{1}{N}(|\Delta_g f| + |\langle \nabla V, \nabla f \rangle|)^2 + \frac{1}{N-n} \langle \nabla V, \nabla f \rangle^2 \\
   &\geq a^2 - \frac{1}{N}(\sqrt{n}a+B)^2 +\frac{1}{N-n}B^2 \\
   &= \frac{1}{N}(\sqrt{N-n}a+\frac{\sqrt{n}}{\sqrt{N-n}}B)^2 \geq 0.
\end{align*}
This proves the claim. \\
We can now plug \eqref{be_compared_to_tensor} into \eqref{given_inequality_distri_to_rcd} and infer that
\begin{align*}
    \Big\langle\Ric_{\mu, \infty}(\nabla f, \nabla f) + |\nabla^2 f|^2_{HS} - \frac{1}{N}(\Delta_\mu f)^2  - Kg(\nabla f, \nabla f), \omega \Big \rangle_{\mathcal{D}', \mathcal{D}} \geq 0.
\end{align*}
Using a partition of unity, we may assume that $\omega$ is supported in one coordinate patch and we can locally assume that $\omega = \phi h^2 \sqrt{|g|}dx^1 \wedge \ldots \wedge dx^n$ for some $\phi \in C^{0}_c(M, [0, \infty))\cap W^{1,2}_{\rm{loc}}(M)$. In local coordinates, we get that \eqref{beconditiondistri} holds for any $f \in C^2(M)$.
Using Lemma \ref{ccinfty_dense_in_lp} and Proposition \ref{w22closure}, we get that \eqref{beconditiondistri} holds for each $f \in H^{2,2}(M) \subset H^2_2(M, \mu)$. This means exactly that the Bakry-Émery condition $\mathsf{BE}(K, N)$ is satisfied.
\end{proof}
The combination of Theorem \ref{first_main_theorem}, Theorem \ref{cdknforlipschitzmetric}, and Theorem \ref{from_distri_torcd}  yields:
\smallskip
\begin{theorem}\label{equivalence_result}
     Let $M$ be a smooth manifold, $g$ a continuous Riemannian metric with $L^2_{\rm{loc}}$-Christoffel symbols and let $h \in C^{0}\cap W^{1,2}_{\rm{loc}}(M)$ be  a positive function. Denote $V := -2 \log h \in C^{0}\cap W^{1,2}_{\rm{loc}}(M)$ and define the measure $\mu$ via $\di\mu:= e^{-V}\di\vol_g$. Let  $N \in [n, \infty]$ and $K \in \R$. The following are equivalent:
     \begin{itemize}
         \item[$\mathrm{(i)}$] $(M, \sfd_g, \mu)$ is a $\mathsf{CD}^*(K, N)$-space.
        \item[$\mathrm{(ii)}$] The distributional Bakry-Émery $N$-Ricci curvature tensor is bounded below by $K$ and $(M, \sfd_g, \mu)$ satisfies \eqref{condition_on_measure}. 
     \end{itemize}
\end{theorem}
{
\begin{remark}
    In Theorem \ref{equivalence_result} (i), the $\cd^*(K, N)$-condition corresponds to the $\rcd^*(K, N)$-condition by Corollary \ref{summary_chapter_4}.
\end{remark}
{\begin{remark}\label{problems_with_rgularity_below_c0}
    To define distributional lower Ricci curvature bounds, the minimal requirement is that $g, g^{-1} \in L^\infty_{loc}$ admits $L^2_{loc}$-Christoffel symbols; this is known in the literature as Geroch-Traschen class, after \cite{GT87}. Under the assumption that  $g, g^{-1} \in L^\infty_{loc}$,  Norris (\cite{norris1997heat}) and De Cecco-Palmieri (\cite{de1988distanza},  \cite{de1991integral}) defined a distance that turns $(M, \sfd_g)$ into a length space, when only considering curves $\gamma:\Lip([0,1], M)$ for which $\mathcal{L}^1$-almost every point is outside a null set $N$, which includes all non-Lebesgue points of $g$. It is thus natural to ask whether the equivalence of distributional and synthetic Ricci curvature lower bounds can be generalized to the Geroch-Traschen class.
    \\  It does not seem  immediate to extend the present paper, though. For instance, it is not obvious if Proposition \ref{metric_speed_explicit} and the identification of the slope of $C^1$-functions as the norm of the gradient remains true in such a higher generality. Related to this,  De Giorgi \cite {de1990conversazioni} introduced the notion of quasi-Riemannian metric spaces, which are Lipschitz manifolds with an elliptic metric $g$ such that the slope of Lipschitz functions  coincide almost everywhere with the norm of the gradient; moreover, he formulated several conjectures on the topic. 
    Since such  identification is crucial for our comparison of the classical and the synthetic Sobolev spaces (\thref{summary_chapter_4}), we assume $g$ to be at least continuous. 
\end{remark}}
}

\subsection*{Remarks on the volume growth condition}
{
We conclude the paper by pointing out that, in the case of a Riemannian manifold with $g \in C^0\cap W^{1,p}_{loc}(M)$ for some $p >n=\dim M$, one can drop the volume growth condition in Theorem \ref{equivalence_result} (ii). More precisely we prove that, in this case, the exponential bound on volume growth follows from the distributional lower Ricci curvature bound. The assumption $g \in C^0\cap W^{1,p}(M)$, with $p >n$ is used in order to build on top of the approximation results from Subsection \ref{subsec:regularisation} and the volume bounds established by Chen-Wei \cite{chen2022improved} (after Petersen-Wei \cite{PW97}) for Riemannian manifolds with Ricci curvature bounded below in an $L^p$-sense.

\begin{proposition}\label{prop:VolumeW1p}
    Let $M$ be a smooth manifold and let $g \in C^0\cap W^{1,p}_{loc}(M)$ be Riemannian metric on $M$,   for some $p >n$. Suppose that the distributional Ricci curvature tensor is bounded below by $K$, for some $K\in \R$. Then $(M, \sfd_g, \di\vol_g)$ satisfies the volume growth condition \eqref{condition_on_measure}. 
\end{proposition}
\begin{proof}
    Fix a point $p \in M$ and a sequence $R_j \in (0, \infty)$ such that $\lim_{j \to \infty}R_j = \infty$. For each $j$, the set $B_j := \overline{B^{\sfd_g}_{R_j}(p)}$ is compact. For $\e \in (0,1)$, define $g_\e = \rho_\e * g$.  By local uniform convergence $g_\e\to g$, as $\e\to 0$,  we get that, for each $j\in \N$, there exists an $\e_j \in (0,1)$, such that:
    \begin{itemize}
        \item for each Borel set $\Omega \subset B_j$, it holds $\frac{1}{2}\vol_g(\Omega) \leq \vol_{g_{\e_j}}(\Omega) \leq 2 \vol_g(\Omega)$;
        \item for each $x \in B_j$, it holds that $g$ and $g_{\e_j}$ are $2$-equivalent; i.e., $\frac{1}{2} g\leq g_{\e_j}\leq 2 g$. 
    \end{itemize}
    Proposition \ref{lp_ricci_conv} yields that, for each compact set $K\subset M$: 
    \begin{equation}\label{eq:RicgepsMollif}
        \norm{\Ric[g_\e]-\rho_\e* \Ric[g]}_{L^{p/2}(K)} \to 0,\quad   \text{as\ } \e \to 0.
    \end{equation}
    For any smooth Riemannian manifold $(M, \Tilde{g})$, in \cite{chen2022improved} the authors consider the quantity $\Bar{k}[\Tilde{g}](H, q, R, x)$, for $H\in \R, 2q>n, R>0, x\in M$,  defined as
    \begin{align*}
     \Bar{k}[\Tilde{g}](H, q, R,x) =  \Big(\frac{1}{\vol_{\Tilde{g}}(B_R^{\sfd_{\Tilde{g}}}(x))}\int_{(B_R^{\sfd_{\Tilde{g}}}(x))}\rho_H^q\di \vol_{\Tilde{g}}\Big)^{\frac{1}{q}},
    \end{align*}
    where $\rho_H(y)=\max(-\rho(y)+(n-1)H, 0)$ and $\rho(y)$ denotes the smallest eigenvalue of the Ricci curvature tensor at $y$. In other words, $\rho_H$ denotes the Ricci curvature lying below the threshold $(n-1)H$.
    The convergence  \eqref{eq:RicgepsMollif} implies that for all $j$, we can choose $\e_j$ small enough such that $\Bar{k}[g_{\e_j}](\frac{K}{n-1}, \frac{p}{2}, 1,x) \leq \delta_0$ for all $x \in B_{R_j-1}(p)$, where  $\delta_0 = \delta_0(\frac{K}{n-1}, \frac{p}{2}, 1)$ is given in \cite[Theorem 1.2]{chen2022improved}. Then, \cite[Theorem 1.2]{chen2022improved} yields that, for $R \leq \frac{R_j-1}{4}$: 
    \begin{align}
        \vol_g(B^{\sfd_{g}}_R(p)) \leq 2\vol_g(B^{\sfd_{g_{\e_j}}}_{2R}(p)) \leq 2(1+ C(n, p, K))e^{R-1}V(K, n, R),
    \end{align}
    where $V(K, n, R)$ denotes the volume of the ball of radius $R$ in the $n$-dimensional space form of constant curvature $K/(n-1)$. Now the result follows by the exponential volume growth of metric balls in space forms.
\end{proof}

In the presence of a weight on the volume measure, the results of \cite{{chen2022improved}} are not yet available in the literature for $L^p$-lower bounds on the Bakry-Émery-$N$-Ricci curvature tensor (though we expect analogous statements). For this reason,   below we include a variant of Proposition \ref{prop:VolumeW1p}  which allows to consider a $C^1$-weight, under the stronger assumption that $g\in C^1$. The proof is again by approximation, building on top of \cite{graf2020singularity}.
}
\begin{proposition}\label{prop:VolumeC1}
    Let $M$ be a smooth manifold and $g$ be a $C^1$-Riemannian metric on $M$. Moreover, let $h \in C^1(M, (0, \infty))$. Let $N \in [n, \infty)$ and $K \in \R$. Suppose that the distributional Bakry-Émery-$N$-Ricci curvature tensor is bounded below by $K$. Then $(M, \sfd_g, h^2\di\vol_g)$ satisfies the volume growth condition \eqref{condition_on_measure}. 
\end{proposition}
\begin{proof}
    Fix a point $p \in M$ and a sequence $R_j \in (0, \infty)$ such that $\lim_{j \to \infty}R_j = \infty$. For each $j$, the set $B_j := \overline{B^{\sfd_g}_{R_j}(p)}$ is compact. For $\e \in (0,1)$, define $g_\e = \rho_\e * g$ and $h_\e = \rho_\e * h$. Noting that \cite[Lemma 4.6] {graf2020singularity} holds similarly for $K < 0$, we get that, for each $j\in \N$, there exists an $\e_j \in (0,1)$ such that 
    \begin{itemize}
        \item for each Borel set $\Omega \subset B_j$, it holds that $\frac{1}{2}h^2\vol_g(\Omega) \leq h^2_{\e_j}\vol_{g_{\e_j}}(\Omega) \leq 2 h^2\vol_g(\Omega)$;
        \item for each $x \in B_j$, it holds  that $g$ and $g_{\e_j}$ are $2$-equivalent; i.e., $\frac{1}{2} g\leq g_{\e_j}\leq 2 g$;
        \item for each $X \in TB_j$, it holds that $\Ric_{\mu, \infty}[g_{\e_j}](X,X) \geq (K-1)g_{\e_j}(X,X)$. 
    \end{itemize}
    Fix  $j \i \N$.  To keep notation short,  write $g_j$ for $g_{\e_j}$ and  $h_j$ for $h_{\e_j}$.  For each $R < R_j$, we can apply the  weighted Bishop-Gromov theorem for $(B_j, \sfd_{g_j}, h^2_j\di\vol_{g_{j}})$  together with the fact that $H:=\sup_j \sup_{B_1} |h_j|<\infty$ to get  constants $C, D>0$, depending only on $K$, $N$, $H$, and $\vol_g(B_1(p))$, such that
    \begin{align*}
        h^2_{j}\vol_{g_{j}}(B_R^{\sfd_{g_j}}(p)) \leq C e^{DR^2}.
    \end{align*}
    From the choice of $\e_j$, it follows that for all $R \leq \frac{1}{2} R_j$:
    \begin{align*}
        h^2\vol_{g}(B_R^{\sfd_{g}}(p)) \leq 2C e^{4DR^2}.
    \end{align*}
    As $j$ was arbitrary, the proof is complete. 
\end{proof}
{\begin{corollary}\label{cor:FinalEquivalence}
     Let $M$ be a smooth manifold and assume one of the following two conditions hold: 
     \begin{itemize}
         \item[(1)]\label{Hp1-Unweighted} Let $g$ be a $C^1$-Riemannian metric and let $h \in C^{1}$ be  a positive function. Denote $V := -2 \log h \in C^{1}$ and define the measure $\mu$ via $\di\mu:= e^{-V}\di\vol_g$.
         \item[(2)]  Let $g$ be a Riemannian metric such that $g \in C^0\cap W^{1,p}_{loc}(M)$, for some $p >n$. Define the measure $\mu$ via $\di \mu = \di \vol_g$. 
     \end{itemize}
       Let  $N \in [n, \infty)$ and $K \in \R$. Then the following are equivalent:
     \begin{itemize}
         \item[$\mathrm{(i)}$] $(M, \sfd_g, \mu)$ is a $\mathsf{CD}^*(K, N)$-space.
        \item[$\mathrm{(ii)}$]\label{Cond(ii)DisRic} The distributional Bakry-Émery $N$-Ricci curvature tensor is bounded below by $K$. 
     \end{itemize}
\end{corollary}

\begin{remark}[Equivalent formulations of Corollary \ref{cor:FinalEquivalence}]\label{rem:FinalEquiv}
Under the assumptions of Corollary \ref{cor:FinalEquivalence}, the $\mathsf{CD}^*(K,N)$ condition is equivalent to $\mathsf{RCD}^*(K,N)$ (which, in turn, is equivalent to $\mathsf{RCD}(K,N)$; see Remark \ref{rem:RCD*=RCD})  since the associated metric measure space is infinitesimally Hilbertian (see Corollary \ref{summary_chapter_4}).
\\Moreover in the unweighted case (i.e., assumption {\rm (1)}), the statement {\rm (ii)} in the equivalence is in turn equivalent to \emph{distributional Ricci curvature tensor bounded below by $K$.}
\end{remark}

\phantomsection
\addcontentsline{toc}{section}{References}
\bibliography{bibliographie}
\bibliographystyle{abbrv}
\end{document}